\documentclass[10pt,a4paper,francais]{mysmfart}

\usepackage[latin1]{inputenc}
\usepackage[T1]{fontenc}
\usepackage[francais, english]{babel}
\usepackage{amsopn}                    
\usepackage{amsmath} 
\usepackage{amssymb}
\usepackage[all,2cell,emtex]{xy}
\UseAllTwocells
\xyoption{2cell}
\input xypic       

\usepackage[hidelinks]{hyperref}

\newcommand\pdfinfty{\texorpdfstring{$\infty$}{oo}}

\makeatletter

\def\@begintheorem#1#2[#3]{%
  \def\@theoremhead{\normalfont\the\thm@headfont
    \@ifempty{#1}{\let\thmname\@gobble}{\let\thmname\@iden}%
    \@ifempty{#2}{\let\thmnumber\@gobble}{\let\thmnumber\@iden}%
    \@ifempty{#3}{\let\thmnote\@gobble}{\let\thmnote\@iden}%
    \thm@swap\swappedhead\thmhead{#1}{#2}{#3}}%
  \sbox\@tempboxa{\@theoremhead}%
  \ifdim\wd\@tempboxa>0.7\linewidth \smf@skippttrue\fi
  \ifsmf@skippt
    \global\smf@skipptfalse
    \item [\thm@indent]%
          {\sloppy\@theoremhead\parskip\z@\@@par}%
    \nobreak\everypar{}%
    \let\thmheadnl\relax
  \else
    \item[\hskip\labelsep\thm@indent\unhbox\@tempboxa\the\thm@headpunct]%
  \fi
  \@restorelabelsep
  \thmheadnl 
  \hypertarget{\csname @currentHref\endcsname}{}
  \ignorespaces}

\makeatother

\SilentMatrices

\SelectTips{cm}{10}

\newcommand{\gmrelax}[1]{\relax}

\newcommand{\gmvieux}[1]{\relax}
\newcommand{\nbd}{\nobreakdash}

\renewcommand{\leq}{\leqslant}
\renewcommand{\geq}{\geqslant}

\theoremstyle{plain} 
\newtheorem{thm}{Th\'eor\`eme}[section]
\newtheorem{prop}[thm]{Proposition}
\newtheorem{lemme}[thm]{Lemme}
\newtheorem{cor}[thm]{Corollaire}
\newtheorem{conj}[thm]{Conjecture}
\theoremstyle{remark}                                             
\newtheorem{rem}[thm]{Remarque}

\newtheorem{exemple}[thm]{Exemple}

\newtheorem{sch}[thm]{Scholie}
\theoremstyle{definition}                                         

\newtheorem{paragr}[thm]{}
\theoremstyle{plain} 
\swapnumbers

\swapnumbers

\numberwithin{equation}{thm}

\newcommand{\e}{\varepsilon}
\newcommand{\W}{\mathcal{W}}

\newcommand{\cat}{{\mathcal{C} \mspace{-2.mu} \it{at}}}
\newcommand{\ord}{{\mathcal{O} \mspace{-2.mu} \it{rd}}}
\newcommand{\ens}{{\mathcal{E} \mspace{-2.mu} \it{ns}}}

\newcommand{\ob}{\operatorname{\mathsf{Ob}}}
\newcommand{\fl}[1]{\operatorname{\mathsf{Fl}}_{#1}}
\newcommand{\sHom}{\operatorname{\kern.5truept\underline{\kern-.5truept\mathsf{Hom}\kern-1truept}\kern1.5truept}}
\newcommand{\Hom}{\operatorname{\mathsf{Hom}}}
\newcommand{\op}[1]{{#1}^{\circ}}

\newcommand{\pref}[1]{{\widehat{ #1 }}}
\newcommand{\cats}{\mathbf{\Delta}}
\newcommand{\simpl}{\pref{\cats}}
\newcommand{\smp}[1]{ \varDelta_{#1}}
\newcommand{\face}[2]{\delta^{#1}_{#2}}
\newcommand{\dgn}[2]{\sigma^{#1}_{#2}}
\newcommand{\sdeg}[2]{{#1}^{\mathrm{deg}}_{#2}}
\newcommand{\sndeg}[2]{{#1}^{\mathrm{ndeg}}_{#2}}
\newcommand{\cs}{\mathcal{CS}}
\newcommand{\Hot}{\mathsf{Hot}}
\newcommand{\loc}[1]{\overline{#1}}
\newcommand{\CW}{\mathrm{CW}}

\newcommand{\card}{\mathsf{card}}

\newcommand{\sauf}{\mathchoice{\raise 1.8pt\hbox{${\scriptstyle\kern
2.5pt\smallsetminus\kern 2.5pt}$}}{\raise 1.8pt\hbox{${\scriptstyle\kern
2.5pt\smallsetminus\kern 2.5pt}$}}{\raise
1.8pt\hbox{${\scriptscriptstyle\kern 1.5pt\smallsetminus\kern
1.5pt}$}}{\raise 1.8pt\hbox{${\scriptscriptstyle\kern
1.5pt\smallsetminus\kern 1.5pt}$}}}

\newcommand{\toto}{{\hskip -2.5pt\xymatrixcolsep{1.3pc}\xymatrix{\ar[r]&}\hskip -2.5pt}}
\renewcommand{\to}{{\hskip -2.5pt\xymatrixcolsep{1pc}\xymatrix{\ar[r]&}\hskip -2.5pt}}

\newcommand{\todouble}{\xymatrixcolsep{1pc}\xymatrix{\ar@<.5ex>[r]\ar@<-.5ex>[r]&}}
\newcommand{\todoubleop}{\xymatrixcolsep{1pc}\xymatrix{\ar@<.5ex>[r]&\ar@<.5ex>[l]}}
\renewcommand{\hookrightarrow}{{\hskip -1.5pt\raise 1.5pt\vbox{\xymatrixcolsep{.9pc}\xymatrix{\ar@{^{(}->}[r]&}}\hskip -3.5pt}}
\renewcommand{\longmapsto}{{\hskip -2.5pt\xymatrixcolsep{1.3pc}\xymatrix{\ar@{|->}[r]&}\hskip -2.5pt}}
\renewcommand{\mapsto}{{\hskip -2.5pt\xymatrixcolsep{.9pc}\xymatrix{\ar@{|->}[r]&}\hskip -2.5pt}}

\def\limproj{\mathop{\oalign{\rm lim\cr
\hidewidth$\longleftarrow$\hidewidth\cr}}}%
\def\limind{\mathop{\oalign{\rm lim\cr
\hidewidth$\longrightarrow$\hidewidth\cr}}}%
\newcommand{\cm}[2]{\mathchoice {#1\raise -1.8pt\vbox{\hbox{$\kern -.8pt/#2$}}} {#1\raise -1.8pt\vbox{\hbox{$\kern -.8pt/#2$}}\kern .8pt} {#1\raise -1.8pt\vbox{\hbox{$\scriptstyle\kern -.8pt /#2$}}} {#1\raise -1.8pt\vbox{\hbox{$\scriptscriptstyle\kern -.8pt /#2$}}} }

\newcommand{\mc}[2]{\mathchoice {\raise -1.8pt\vbox{\hbox{$#1\backslash$}}#2} {\raise -1.8pt\vbox{\hbox{$#1\backslash$}}#2} {\raise -1.8pt\vbox{\hbox{$\scriptstyle#1\backslash$}}#2} {\raise -1.8pt\vbox{\hbox{$\scriptscriptstyle#1\backslash$}}#2} }

\newcommand{\nCat}[1]{{#1}\hbox{\protect\nbd-}\kern1pt\cat}
\newcommand{\ooCat}{\nCat{\infty}}
\newcommand{\comp}[1]{\mathop{\kern 1.5pt*^{}_{#1}\kern 1.5pt}}
\newcommand{\Par}{\operatorname{\mathsf{Par}}}
\newcommand{\transf}[2]{{#1}^{}_{#2}}                  
\newcommand{\incl}[1]{i_{#1}}                          
\newcommand{\tb}[1]{\tau_{\leq #1}^{\mathrm b}}
\newcommand{\ti}[1]{\tau_{\leq #1}^{\mathrm i}}
\newcommand{\KDual}[1]{\mathcal{D}_{#1}}
\newcommand{\Dual}{\KDual{}}
\newcommand{\opp}[1]{#1^{\vee}}
\newcommand{\opi}[1]{{}^{\vee}#1}
\newcommand{\K}{J}

\newcommand{\augm}{e}                                                     
\newcommand{\ADC}{\mathcal{C}_{\mathrm{da}}}                              
\newcommand{\ADCfin}{\underline{\mathbb{Z}}}                              
\newcommand{\ADCdescfin}{\underline{\mathbb{Z'}\kern -2.5pt}\kern 2pt}    
\newcommand{\KDualADC}[1]{{D}_{#1}}

\newcommand{\Stlambda}{\lambda}                                           
\newcommand{\Stnu}{\nu}                                                   
\newcommand{\ch}[1]{[\,#1\,]}                                             
\newcommand{\cellule}[1]{#1}
\newcommand{\Pos}[1]{{#1}_+}                                              
\newcommand{\Neg}[1]{{#1}_-}                                              
\newcommand{\leqb}[1]{\mathrel{\leq_{#1}}}                                
\newcommand{\leqbb}{\mathrel{\leq_{\mathbb{N}}}}                          
\newcommand{\leqbbs}{\mathrel{<_{\mathbb{N}}}}                            
\newcommand{\atom}[1]{\langle{#1}\rangle}

\newcommand{\tablg}[2]{\begin{pmatrix}#1^0_0 &#1^0_1 &\dots &#1^0_{#2}\cr\noalign{\vskip 3pt} #1^1_0 &#1^1_1 &\dots &#1^1_{#2}\end{pmatrix}}
\newcommand{\tabll}[2]{\begin{pmatrix}#1^0_0 &#1^0_1 &\dots &#1^0_{#2-1} &#1^0_{#2}\cr\noalign{\vskip 3pt} #1^1_0 &#1^1_1 &\dots &#1^1_{#2-1} &#1^1_{#2}\end{pmatrix}}
\newcommand{\tabld}[2]{\begin{pmatrix}#1^0_0 &\dots &#1^0_{#2-1} &#1^0_{#2}\cr\noalign{\vskip 3pt} #1^1_0 &\dots &#1^1_{#2-1} &#1^1_{#2}\end{pmatrix}}
\newcommand{\tablnu}[3]{\begin{pmatrix}#1^0_0 &\dots &#1^0_{#2} &#3\cr\noalign{\vskip 3pt} #1^1_0 &\dots &#1^1_{#2} &#3\end{pmatrix}}

\newcommand{\cnb}[1]{\widetilde{#1}}   

\newcommand{\Ch}[1]{C#1}               
\newcommand{\Chdeg}[1]{D#1}            
\newcommand{\Chnorm}[1]{c#1}           
\newcommand{\NerfADC}{N}               
\newcommand{\Wsimpl}{\W_{\simpl}}
\newcommand{\WADC}{\W_{\ADC}}

\newcommand{\Chcs}{K}                  

\newcommand{\ADCO}[1]{(\Chnorm{\smp{#1}})}
\newcommand{\Or}[1]{\mathcal{O}_{#1}}     

\newcommand{\NStr}[1]{N_{#1}}             
\newcommand{\cStr}[1]{c_{#1}}             
\newcommand{\NStreet}{\NStr{\infty}}      
\newcommand{\cStreet}{\cStr{\infty}}      
\newcommand{\WooCat}{\W_{\ooCat}}
\newcommand{\WnCat}[1]{\W_{\nCat{#1}}}

\newcommand{\OrCS}[1]{\mathcal{O}({#1})}
\newcommand{\OrCs}{\mathcal{O}}           

\author{Dimitri Ara}
\address{%
Institut de Mathématiques de Marseille\\
Université d'Aix-Marseille\\
Case Postale 907\\
163 avenue de Luminy\\
13288 Marseille Cedex 9\\
France}
\email{dimitri.ara@univ-amu.fr}
\urladdr{http://iml.univ-mrs.fr/\raise -3.3pt\vbox{\hbox{$\widetilde{ \ }\,$}}ara/}

\author{Georges Maltsiniotis}
\address{CNRS, Institut de Math\'ematiques de Jussieu\\
Universit\'e Paris 7 Diderot\\
Case Postale 7012\\
B\^atiment Sophie Germain\\
75205 Paris Cedex 13\\
France}
\email{georges.maltsiniotis\at imj-prg.fr}
\urladdr{http://webusers.imj-prg.fr/\raise -3.3pt\vbox{\hbox{$\widetilde{ \ }\,$}}georges.maltsiniotis/}

\title[Le type d'homotopie de l'oriental d'un complexe simplicial]{Le type d'homotopie de la $\infty$-catégorie associée à un complexe simplicial}

\begin{document}

\frontmatter

\begin{abstract} 
Cet article fait partie d'une série d'articles sur la théorie de l'homotopie des $n$\nbd-catégories strictes. Dans le premier article de cette série, nous avons dégagé des conditions suffisantes assurant l'existence d'une structure de catégorie de modèles à la Thomason sur la catégorie des $n$\nbd-catégories strictes. Le but principal du présent article est de démontrer une de ces conditions. Pour ce faire, on associe à tout complexe simplicial 
une $\infty$\nbd-catégorie stricte engendrée librement au sens des polygraphes. On conjecture que cette $\infty$\nbd-catégorie a le même type d'homotopie que le complexe simplicial de départ et on démontre cette conjecture lorsque le complexe simplicial provient d'un ensemble ordonné. On introduit une notion d'objet quasi-initial d'une $\infty$\nbd-catégorie et on prouve que les orientaux de Street admettent un tel objet. Un des outils essentiels utilisé dans ce texte est la théorie des complexes dirigés augmentés de Steiner.
\end{abstract}

\begin{altabstract}
This paper is part of a series of papers about homotopy theory of strict $n$\nbd-categories. In the first paper of this series, we gave conditions that guarantee the existence of a Thomason model category structure on the category of strict $n$\nbd-categories. The main goal of our paper is to show one of these conditions. To do so, we associate to any simplicial complex a strict $\infty$\nbd-category generated by a computad. We conjecture that this $\infty$\nbd-category has the same homotopy type as the corresponding simplicial complex and we prove this conjecture when the simplicial complex comes from a poset. We introduce the notion of a quasi-initial object of an $\infty$\nbd-category and we show that Street's orientals admit such an object. One of the main tools used in this text is Steiner's theory of augmented directed complexes.
\end{altabstract}



\altkeywords{augmented directed complexes, contractions, homotopy types,
  orientals, quasi-initial objects, simplicial complexes, strict
  $\infty$-categories, Thomason equivalences of
 $\infty$\nobreakdash-categories}

\maketitle

\newpage

\tableofcontents

\mainmatter

\section*{Introduction}

Cet article fait partie d'un projet consacré à la théorie de l'homotopie des $n$\nbd-catégories strictes (avec $n$ éventuellement infini), composé actuellement de \cite{DG}, \cite{Dapp}, \cite{DG3},  \cite{DG2} et \cite{DG1}, dont le but est de généraliser la théorie de l'homotopie des catégories développée par Grothendieck, Quillen et Thomason aux $n$\nbd-catégories strictes. Ainsi, dans ce texte, toutes les $n$\nbd-catégories et $n$\nbd-foncteurs seront supposés stricts.
\smallbreak

La théorie de l'homotopie des petites catégories naît avec l'introduction par \hbox{Grothendieck} dans~\cite{Nerf} du foncteur nerf de source la ca\-tégorie $\cat$ des petites catégories et de but $\simpl$, celle des ensembles simpliciaux. Ce foncteur permet de définir une classe $\W_\cat$ d'équivalences faibles dans $\cat$, en définissant $\W_\cat$ comme la classe des foncteurs entre petites catégories dont l'image par le nerf est une équivalence d'homotopie faible d'ensembles simpliciaux. En vertu d'un résultat de Quillen (exposé dans la thèse d'Illusie~\cite{IL}), le foncteur nerf induit une équivalence de catégories
$$\mathrm{Ho}(\cat)\toto\mathrm{Ho}(\simpl)$$
entre la localisation de $\cat$ par $\W_\cat$ et la catégorie homotopique usuelle des ensembles simpliciaux, elle-même équivalente à la catégorie homotopique des $\CW$\nbd-complexes, grâce à un théorème de Milnor~\cite{Mil}. Par ailleurs, Thomason démontre l'existence d'une structure de catégorie de modèles de Quillen sur $\cat$ dont la classe des équivalences faibles est égale à $\W_\cat$~\cite{Th}.
\smallbreak

Dans un article séminal~\cite{S1}, Street a introduit un foncteur nerf  
$$\NStreet:\ooCat\toto\simpl$$
de la catégorie des petites $\infty$\nbd-catégories vers celle des ensembles simpliciaux, dont la restriction à $\cat$ (considérée comme sous-catégorie pleine de $\ooCat$) est le nerf de Grothendieck. Ce foncteur permet de définir une classe $\WooCat$ d'équivalences faibles dans $\ooCat$, en posant $\WooCat=\NStreet^{-1}(\Wsimpl)$, où $\Wsimpl$ désigne la classe des équivalences d'homotopie faibles d'ensembles simpliciaux. De même, pour tout entier $n>0$, la restriction $\NStr{n}$ du nerf de Street à la catégorie $\nCat{n}$ des petites $n$\nbd-catégories (considérée comme sous-catégorie pleine de $\ooCat$) définit une classe $\WnCat{n}$ d'équivalences faibles dans $\nCat{n}$ par la formule $\WnCat{n}=\NStr{n}^{-1}(\Wsimpl)$. On conjecture que pour tout~$n$, $1\leq n\leq\infty$, il existe une structure de catégorie de modèles sur $\nCat{n}$ dont la classe des équivalences faibles est égale à $\WnCat{n}$ et qui est équivalente, au sens de Quillen, avec la structure de catégorie de modèles classique sur les ensembles simpliciaux, généralisant ainsi les théorèmes de Thomason et Quillen cités précédemment concernant le cas $n=1$.
\smallbreak

Cette conjecture est également établie pour $n=2$. Dans~\cite{Chiche}, extrait de sa thèse~\cite{ChicheThese}, Jonathan Chiche prouve que le foncteur $\NStr{2}$ induit une équivalence de catégories
$$\mathrm{Ho}(\nCat{2})\toto\mathrm{Ho}(\simpl)$$
entre la localisation de $\nCat{2}$ par $\WnCat{2}$ et la catégorie homotopique des ensembles simpliciaux. Dans~\cite{DG}, les auteurs du présent article construisent une structure de catégorie de modèles sur $\nCat{2}$ dont la classe des équivalences faibles est égale à~$\WnCat{2}$, et une adjonction de Quillen entre celle-ci et la catégorie de modèles des ensembles simpliciaux. Il résulte alors facilement du théorème de Jonathan Chiche que cette adjonction est une équivalence. Mentionnons que le texte antérieur~\cite{Wor} traite également la question d'une généralisation $2$\nbd-catégorique du théorème de Thomason mais que celui-ci comporte de sérieuses erreurs (voir l'introduction de~\cite{DG} pour plus de détails).
\smallbreak

En vue d'une preuve de cette conjecture dans le cas général, les auteurs ont démontré dans~\cite{DG} un \og théorème de Thomason abstrait\fg{}, permettant de dégager deux conditions techniques impliquant, pour $1\leq n\leq\infty$ arbitraire, l'existence d'une catégorie de modèles sur $\nCat{n}$ dont la classe des équivalences faibles est égale à $\WnCat{n}$, et d'une adjonction de Quillen de cette dernière avec la catégorie des modèles des ensembles simpliciaux. La motivation initial du présent article est de démontrer la première de ces conditions (condition (e) de~\cite[scholie~5.14]{DG}), à savoir que si $\cStr{n}$ désigne l'adjoint à gauche de $\NStr{n}$, pour tout ensemble simplicial $X$ qui est le nerf d'un ensemble ordonné, le morphisme d'adjonction $X\to\NStr{n}\cStr{n}(X)$ est une équivalence faible simpliciale (corollaire~\ref{corconde}).
\smallbreak

Afin d'établir ce résultat, on est conduit à introduire un certain nombre de notions importantes et à démontrer plusieurs théorèmes présentant un intérêt indépendant. On rappelle que le couple de foncteurs adjoints 
$$\cStreet:\simpl\toto\ooCat\ ,\qquad\NStreet:\ooCat\toto\simpl$$
est obtenu par le procédé de Kan à partir de l'objet cosimplicial défini par les orientaux de Street~\cite{S1,S2,S2cor}. Pour $n\geq0$, cet objet cosimplicial associe au simplexe standard~$\smp{n}$ le $n$\nbd-ième oriental $\Or{n}$, qui est une $n$\nbd-catégorie librement engendrée au sens des polygraphes par les simplexes non dégénérés de $\smp{n}$. La définition formelle de $\Or{n}$ par Street est combinatoire et assez difficile à utiliser. Dans ce texte, on se servira d'une définition équivalente due à Steiner~\cite{Steiner,SteinerOr}, plus proche de l'algèbre homologique et plus maniable (une description analogue des orientaux figure aussi dans~\cite{B-P}).
\smallbreak

Dans~\cite{Steiner}, Steiner définit un couple de foncteurs adjoints
$$\Stlambda:\ooCat\toto\ADC\ ,\qquad\Stnu:\ADC\toto\ooCat\ ,$$
où $\ADC$ est la catégorie des complexes dirigés augmentés, catégorie dont les objets sont les complexes de chaînes de groupes abéliens en degrés positifs, augmentés, et munis d'un sous-monoïde gradué \og de positivité\fg{} (sans aucune compatibilité entre ce sous-monoïde et la différentielle ou l'augmentation du complexe). La catégorie des complexes de chaînes de groupes abéliens en degrés positifs est équivalente à celle des $\infty$\nbd-catégories en groupes abéliens~\cite{Bourn}. En oubliant la structure de groupe abélien, on associe ainsi à un tel complexe une $\infty$\nbd-catégorie. Le foncteur $\Stnu$ de Steiner associe à un complexe dirigé augmenté une sous-$\infty$\nbd-catégorie de cette $\infty$\nbd-catégorie, déterminée par le sous-monoïde de positivité et l'augmentation (voir~\ref{defStnu}). En associant à un ensemble simplicial $X$ son complexe de chaînes normalisé, muni de l'augmentation définie par le morphisme de $X$ vers le point, et du sous-monoïde de positivité défini par les combinaisons linéaires à coefficients positifs de simplexes non dégénérés de $X$, on obtient un foncteur $\Chnorm{}:\simpl\to\ADC$. Par définition, l'oriental $\Or{n}$, $n\geq0$, est l'image du simplexe standard $\smp{n}$ par le foncteur composé $\Stnu\Chnorm$. 
\smallbreak

Le foncteur $\Chnorm$ commute aux limites inductives. On dispose ainsi d'un couple de foncteurs adjoints 
$$\Chnorm:\simpl\toto\ADC\ ,\qquad\NerfADC:\ADC\toto\simpl\ ,$$
obtenu par le procédé de Kan à partir de l'objet cosimplicial de $\ADC$ défini par la restriction du foncteur $\Chnorm$ aux représentables, et on démontre que $\NerfADC$ est canoniquement isomorphe au composé $\NStreet\Stnu$. Le foncteur $\NerfADC$ permet de définir une classe d'équivalences faibles $\WADC$ dans $\ADC$ par la formule $\WADC=\NerfADC^{-1}(\Wsimpl)=\Stnu^{-1}(\WooCat)$. On conjecture que le foncteur $\NerfADC$ induit une équivalence des catégories localisées. En tenant compte de la conjecture mentionnée précédemment, on obtiendrait donc un triangle commutatif à isomorphisme près d'équivalences de catégories
$$
\xymatrixcolsep{1pc}
\xymatrix{
&\mathrm{Ho}(\ADC)\ar[rr]^{\loc{\Stnu}}\ar[rd]_{\loc{\NerfADC}}
&&\mathrm{Ho}(\ooCat)\ar[ld]^{\loc{\NStreet}}
\\
&&\mathrm{Ho}(\simpl)
&&&\hskip -20pt,\hskip 20pt
}$$
induites par les foncteurs $\Stnu$, $\NerfADC$ et $\NStreet$.
\smallbreak

Une notion cruciale introduite dans cet article est celle d'objet quasi-initial d'une $n$\nbd-catégorie, généralisant la notion d'objet initial d'une catégorie. Pour $n$ fini, cette notion est définie par récurrence sur $n$. Un objet d'une $1$\nbd-catégorie (autrement dit d'une catégorie) est quasi-initial s'il est initial. Pour $n\geq2$, un objet $x$ est quasi-initial si pour tout objet $y$, la $(n-1)$\nbd-catégorie $\sHom(x,y)$ des flèches de source itérée~$x$ et but itéré~$y$ admet un objet quasi-initial. Pour $n=\infty$, on peut poser une définition analogue par \og coïnduction\fg. Une définition dans un style plus \og classique\fg{} est présentée dans l'appendice~\ref{AppB}. La notion d'objet quasi-final est définie de façon duale. Un résultat,  très facile à établir mais important, est que ces notions sont stables par troncation intelligente. (Si $m<n$, le foncteur de troncation intelligente $\ti{m}$ est l'adjoint à gauche de l'inclusion $\nCat{m}\to\nCat{n}$.) Du point de vue homotopique, la propriété importante d'une $n$\nbd-catégorie, $1\leq n\leq\infty$, admettant un objet quasi-initial ou quasi-final est qu'elle est faiblement contractile (le morphisme canonique vers l'objet final de $\nCat{n}$ appartient à $\WnCat{n}$). Cette propriété résulte d'un théorème~A $n$\nbd-catégorique~\cite{DG3}.
\smallbreak

Une deuxième notion importante est celle de contraction d'une $\infty$\nbd-catégorie. Il s'agit d'une transformation (oplax) satisfaisant à une condition de \og normalisation\fg{} et une condition \og d'idempotence\fg{} (voir~\ref{defoocontraction}). On démontre qu'une $\infty$\nbd-catégorie admettant une contraction admet un objet quasi-initial (la réciproque n'étant pas vraie), ainsi qu'un résultat dual concernant l'existence d'un objet quasi-final.
\smallbreak

Un des principaux résultats de cet article est que les orientaux de Street admettent une contraction et une contraction duale, et en particulier un objet quasi-initial et un objet quasi-final. Pour obtenir ce résultat, on introduit une notion de contraction (et de contraction duale) pour un complexe dirigé augmenté, et on montre que le foncteur~$\Stnu$ de Steiner transforme une telle contraction (ou contraction duale) en une contraction (ou contraction duale) $\infty$\nbd-catégorique. Pour conclure, on prouve que pour $n\geq0$, le complexe dirigé augmenté $\Chnorm{\smp{n}}$ associé au simplexe standard $\smp{n}$ admet à la fois une contraction et une contraction duale. Il est à noter que cette contraction duale ne s'obtient \emph{pas} par une dualité à partir de la contraction considérée, ni par inversion de l'ordre des sommets du simplexe $\smp{n}$. Les deux ont une origine simpliciale.
\smallbreak

Plus généralement, on conjecture que pour tout couple d'entiers $i,n$ tel que \hbox{$0<i<n$} et tout couple $x,y$ de $i$\nbd-flèches parallèles de $\Or{n}$, la $(n-i-1)$\nbd-catégorie des flèches de $\Or{n}$ de source itérée $x$ et de but itéré $y$ ou bien est vide ou bien admet à la fois un objet quasi-initial et un objet quasi-final. On démontre cette conjecture pour $n$ arbitraire et $i=1$. Un cas particulier de ce résultat joue un rôle crucial dans la preuve de la \og condition~(e)\fg{} de~\cite{DG}, motivation initiale de cet article, mentionnée précédemment dans cette introduction.
\smallbreak

Dans~\cite{Steiner}, Steiner présente des conditions suffisantes pour que l'image d'un complexe dirigé augmenté par le foncteur $\Stnu$ soit une $\infty$\nbd-catégorie librement engendrée au sens des polygraphes (autrement dit un objet cofibrant pour la structure de catégorie de modèles \og catégorique\fg, \og canonique\fg{} ou \og folklorique\fg{} sur $\ooCat$~\cite{LMW}), obtenant ainsi une variante de résultats antérieures de Street~\cite{S1,S2,S2cor}. On démontre que si $X$ est l'ensemble simplicial associé à un complexe simplicial, le complexe dirigé augmenté $\Chnorm{X}$ satisfait à ces conditions. Cela permet d'associer à tout complexe simplicial une $\infty$\nbd-catégorie engendrée librement au sens des polygraphes. De plus, on conjecture que le type d'homotopie de cette $\infty$\nbd-catégorie est le même que celui du complexe simplicial de départ. De façon plus précise, si $X$ est l'ensemble simplicial associé à un complexe simplicial, on démontre que $\Stnu\Chnorm{X}$ est canoniquement isomorphe à $\cStreet(X)$, et la conjecture affirme que le morphisme d'adjonction $X\to\NStreet\cStreet(X)$ est une équivalence d'homotopie faible simpliciale. On établie cette conjecture dans le cas particulier où le complexe simplicial est celui associé à un ensemble ordonné~$E$. L'ensemble de ses simplexes est alors l'ensemble des parties finies totalement ordonnées non vides de~$E$, et l'ensemble simplicial correspondant $X$ est le nerf de $E$. En fait, pour un tel~$X$, on a mieux: pour \emph{tout} $n\geq1$, le morphisme d'adjonction $X\to\NStr{n}\cStr{n}(X)$ est une équivalence faible simpliciale (condition (e) de~\cite{DG}).
\smallbreak

La stratégie pour démontrer ce théorème se déroule comme suit. On remarque d'abord que grâce à l'égalité du triangle pour une adjonction, il suffit de prouver que pour tout ensemble ordonné $E$, le morphisme d'adjonction $\cStr{n}\NStr{n}(E)\to E$ est une équivalence faible $n$\nbd-catégorique, autrement dit que ce morphisme appartient à~$\WnCat{n}$. Pour $n=1$, le nerf de Grothendieck étant pleinement fidèle, ce morphisme est un isomorphisme, on peut donc supposer que $n\geq2$. Ensuite, on observe que $\cStr{n}\NStr{n}(E)$ est le $n$\nbd-tronqué intelligent de $\cStreet\NStreet(E)$ qui est isomorphe à l'image du nerf de $E$ par le foncteur $\Stnu\Chnorm$. Grâce à la description du foncteur $\Stnu$ de Steiner, on en déduit que les objets de $\cStr{n}\NStr{n}(E)$ s'identifient aux éléments de $E$, et les $1$-flèches aux parties finies non vides totalement ordonnées $S$ de $E$, la source (resp. le but) de $S$ étant son minimum (resp. son maximum). En particulier, si $x,y$ sont deux éléments de~$E$, pour qu'il existe une $1$\nbd-flèche de $\cStr{n}\NStr{n}(E)$ de source $x$ et but $y$, il faut et il suffit que $x\leq y$. De plus, le morphisme d'adjonction induit l'identité sur les objets. 
\smallbreak

Or, en vertu d'une conséquence de la comparaison du nerf de Street avec une variante bisimpliciale dudit nerf~\cite{DG2}, pour montrer qu'un $n$\nbd-foncteur $F$ est une équivalence faible, il suffit de montrer que pour tout couple d'objets $x,y$ de sa source, le \hbox{$(n-1)$}\nbd-foncteur induit
$$\sHom(x,y)\toto\sHom(Fx,Fy)$$
est une équivalence faible. Comme la $n$\nbd-catégorie but du morphisme d'adjonction considéré est un ensemble ordonné, pour montrer que ce morphisme est une équivalence faible, il suffit donc de montrer que pour tout couple $x,y$ d'éléments de $E$ tel que $x\leq y$, la $(n-1)$\nbd-catégorie $\sHom_{\cStr{n}\NStr{n}(E)}(x,y)$ est faiblement contractile, et pour cela, il suffit de voir qu'elle admet un objet quasi-initial.
\smallbreak

Or, en vertu de la description des $1$\nbd-flèches de $\cStr{n}\NStr{n}(E)$, l'ensemble $S_0=\{x,y\}$ représente un objet de $\sHom_{\cStr{n}\NStr{n}(E)}(x,y)$, et un objet arbitraire est représenté par une partie finie non vide totalement ordonnée $S$ de $E$, telle que $x=\min(S)$ et $y=\max(S)$. En utilisant à nouveau le fait que $\cStr{n}\NStr{n}(E)$ est le $n$\nbd-tronqué de l'image du nerf de $E$ par le foncteur $\Stnu\Chnorm$, ainsi que la description du foncteur $\Stnu$ de Steiner et un argument de commutation aux limites inductives, on prouve que la $(n-2)$\nbd-catégorie $\sHom_{\cStr{n}\NStr{n}(E)}(S_0,S)$ des flèches de $\cStr{n}\NStr{n}(E)$ de $1$\nbd-cellule source itérée $S_0$ et but itéré $S$, s'identifie au $(n-2)$\nbd-tronqué intelligent de $\sHom_{\Or{p}}(\{0,p\},\{0,1,\dots,p\})$, où \hbox{$p=\card(S)-1$}, qui admet un objet quasi-initial. On en déduit que $S_0$ est un objet quasi-initial de $\sHom_{\cStr{n}\NStr{n}(E)}(x,y)$, ce qui prouve l'assertion.
\smallbreak

La preuve esquissée ci-dessus diffère très légèrement de la démonstration exposée dans le corps de l'article. En effet, pour abréger les préliminaires $n$\nbd-catégoriques, la notion d'objet quasi-initial ou quasi-final d'une $n$\nbd-catégorie n'y est définie que pour $n$ \emph{fini}, le cas $n=\infty$, ainsi que la notion de contraction $n$\nbd-catégorique, étant laissés pour l'appendice~\ref{AppB}. En contrepartie il en résulte parfois des énoncés moins naturels ou des preuves plus techniques, bien que plus directes.
\smallbreak

Terminons cette introduction en mentionnant quelques textes en rapport
avec le présent article. Dans~\cite{Bourn}, Dominique Bourn établit un
dictionnaire entre les homotopies des complexes de chaînes en degrés
positifs et les pseudo-transformations des $\infty$\nbd-groupoïdes abéliens
correspondant. Dans un texte non publié~\cite{Albert}, dont nous avons pris
connaissance après une première rédaction de ce travail, Albert Burroni
introduit une notion de \og structure initiale\fg{} essentiellement
équivalente à notre notion de contraction. Dans \cite{MetResPol}, François
Métayer, et dans \cite{LafMetPolRes}, ce dernier et Yves Lafont
étudient la notion de $n$-cylindre dans une $\infty$\nbd-catégorie, ce qui
permet de dégager les formules explicites de la définition d'une
transformation entre $\infty$\nbd-foncteurs. Enfin, dans~\cite{LS}, Ross
Street et Stephen Lack étudient, entre autres, les tronqués en basse
dimension des orientaux.

\subsection*{Plan de l'article}
Dans la première section, après quelques rappels sur les $n$\nbd-catégories, $0\leq n\leq\infty$, et les foncteurs de troncation bête et intelligente, on introduit, pour $n<\infty$, la notion de $n$\nbd-catégorie admettant un objet quasi-initial ou quasi-final. On rappelle les notions d'ensemble de cellules qui engendre par compositions, ou qui engendre librement au sens des polygraphes une $\infty$\nbd-catégorie, et on établit quelques résultats relatifs à ces deux notions. On termine la section par des considérations sur les divers endofoncteurs de dualité dans la catégorie des $\infty$\nbd-catégories.
\smallbreak

La deuxième section est consacrée à la catégorie des complexes dirigés augmentés, introduite par Steiner. On décrit les limites inductives et projectives dans cette catégorie, on introduit la notion de complexe dirigé augmenté décent et celle de morphisme constant entre deux tels complexes. On considère le concept d'homotopie entre deux morphismes de complexes dirigés augmentés et les divers endofoncteurs de dualité dans cette catégorie.
\smallbreak

Dans la troisième section, on rappelle la théorie de Steiner donnant naissance à un couple de foncteurs adjoints entre la catégorie des complexes dirigés augmentés et celle des petites $\infty$\nbd-catégories. On commence par la description de ces deux foncteurs. On rappelle les notions de base d'un complexe dirigé augmenté, de base unitaire, de base sans boucles et de base fortement sans boucles. On énonce le principal résultat de Steiner: la sous-catégorie pleine des complexes dirigés augmentés formée de ceux qui admettent une base unitaire sans boucles se plonge de façon pleinement fidèle dans la catégorie des $\infty$\nbd-catégories, et l'image d'un tel complexe est une $\infty$\nbd-catégorie librement engendrée au sens des polygraphes.
\smallbreak

Dans la quatrième section, on dégage des conditions suffisantes pour que la $\infty$\nbd-catégorie associée à un complexe dirigé augmenté par le foncteur de Steiner soit une $n$\nbd-catégorie admettant un objet quasi-initial ou quasi-final. On introduit les notions de contraction et contraction duale d'un complexe dirigé augmenté. Dans le cas particulier où le complexe dirigé augmenté est décent, on interprète ces notions en termes d'homotopies entre l'identité de ce complexe et un endomorphisme constant.
\smallbreak

Dans la cinquième section, on associe fonctoriellement à chaque ensemble simplicial trois complexes dirigés augmentés décents admettant une base, le complexe dirigé augmenté des chaînes, celui des chaînes dégénérées et le complexe dirigé augmenté normalisé des chaînes, définissant ainsi trois foncteurs commutant aux limites inductives. La restriction du dernier de ces trois foncteurs aux ensembles simpliciaux représentables fournit un objet cosimplicial qui définit par le procédé de Kan un foncteur nerf de la catégorie des complexes dirigés augmentés vers celle des ensembles simpliciaux. À l'aide de ce foncteur nerf, on définit des équivalences faibles dans la catégorie des complexes dirigés augmentés, et on conjecture que la catégorie localisée correspondante est équivalente à la catégorie homotopique des $\CW$\nbd-complexes.
\smallbreak

Dans la sixième section, on rappelle la définition des orientaux de Street. On démontre que le complexe dirigé augmenté qui les définit \emph{via} le foncteur de Steiner admet une contraction et une contraction duale. On en déduit en particulier que les orientaux de Street admettent un objet quasi-initial et un objet quasi-final.
\smallbreak

Dans la septième section, on rappelle la définition du nerf de Street obtenu par le procédé de Kan à partir de l'objet cosimplicial fourni par les orientaux. À l'aide du nerf de Street, on définit des équivalences faibles dans la catégorie des $\infty$\nbd-catégories, et on conjecture que la catégorie localisée correspondante est équivalente à la catégorie homotopique des $\CW$\nbd-complexes. On énonce des conséquences d'un théorème A $\infty$\nbd-catégorique
et d'un théorème de comparaison du nerf de Street avec une variante bisimpliciale dudit nerf, théorèmes obtenus dans deux articles en préparation.
\smallbreak

Dans la huitième section, après quelques rappels sur les complexes simpliciaux, on montre que le complexe dirigé augmenté normalisé des chaînes d'un ensemble simplicial défini par un complexe simplicial admet une base unitaire fortement sans boucles. En particulier, la $\infty$\nbd-catégorie associée est librement engendrée au sens des polygraphes.
\smallbreak

La neuvième section est consacrée à la définition d'un isomorphisme fonctoriel entre la $\infty$\nbd-catégorie associée au complexe dirigé augmenté normalisé des chaînes d'un ensemble simplicial défini par un complexe simplicial et l'image par l'adjoint à gauche du nerf de Street de cet ensemble simplicial.
\smallbreak

La dernière section est consacrée à l'étude de la $\infty$\nbd-catégorie image par l'adjoint à gauche du nerf de Street du nerf d'un ensemble ordonné, ainsi que du $n$-tronqué intelligent de cette $\infty$\nbd-catégorie. On montre que, pour $n>0$, le morphisme d'adjonction correspondant est une équivalence faible $n$\nbd-catégorique. En particulier, pour tout~$n>0$, le type d'homotopie de cette $n$\nbd-catégorie est isomorphe à celui de l'ensemble ordonné de départ.
\smallbreak

Le but de l'appendice~\ref{AppA} est de donner une description explicite du $2$\nbd-tronqué intelligent de la $\infty$\nbd-catégorie image par l'adjoint à gauche du nerf de Street du nerf d'un ensemble ordonné.
\smallbreak

Dans l'appendice~\ref{AppB}, on généralise les notions d'objet quasi-initial et d'objet quasi-final d'une $n$\nbd-catégorie, introduites pour $n$ fini dans la première section, au cas où $n=\infty$. On introduit la notion de contraction d'une $\infty$\nbd-catégorie, on montre qu'une $\infty$\nbd-catégorie admettant une contraction admet un objet quasi-initial et que la $\infty$\nbd-catégorie associée à un complexe dirigé augmenté admettant une contraction admet une contraction. Ces considérations permettent de retrouver la plupart des résultats de la section 4 et des propriétés des orientaux démontrées dans la section~6.

\section{Notations et terminologie pour les \pdfinfty-catégories}\label{notoocat}

\begin{paragr}\label{notooCat}
Toutes les $\infty$\nbd-catégories considérées dans ce texte sont des $\infty$\nbd-catégories \emph{strictes} et tous les $\infty$\nbd-foncteurs des $\infty$\nbd-foncteurs \emph{stricts}. On note $\ooCat$ la catégorie des petites $\infty$\nbd-catégories et $\infty$\nbd-foncteurs entre celles-ci. Si $C$ est une $\infty$\nbd-catégorie on note $\ob(C)$, ou plus simplement $C_0$, l'ensemble de ses objets et pour $i>0$, on note $\fl{i}(C)$ ou $C_i$ l'ensemble de ses $i$\nbd-flèches. Pour toute $i$\nbd-flèche $x$ de $C$, on note $s(x)$ sa source, $t(x)$ son but et $1_x$ la $(i+1)$\nbd-flèche unité de $x$. Pour $j$ un entier tel que $0\leq j<i$, on note $s_j(x)$ (resp. $t_j(x)$) la $j$\nbd-cellule source itérée (resp. but itéré) de~$x$. Pour $0\leq j<i$, si $x$ et $y$ sont deux $i$\nbd-flèches \emph{$j$\nbd-composables}, autrement dit si la $j$\nbd-cellule source itérée de $x$ est égale à la $j$\nbd-cellule but itéré de $y$, on note \smash{$x\comp{j}y$} le composé correspondant. 
\smallbreak

Plus généralement, pour $i,j,k$ tels que $0\leq k<\min\{i,j\}$, si $x$ est une $i$\nbd-flèche et $y$ une $j$\nbd-flèche telles que $s_k(x)=t_k(y)$, on notera \smash{$x\comp{k}y$} la $(\max\{i,j\})$\nbd-flèche de $C$ obtenue en composant $x$ avec la $i$\nbd-flèche identité itérée de $y$ ou la $j$\nbd-flèche identité itérée de $x$ avec $y$ selon que $i\geq j$ ou $i\leq j$. Par convention, si $i<j$, l'opération \smash{$\comp{i}$} sera prioritaire sur l'opération \smash{$\comp{j}$}. Par exemple:
$$u\comp{0}v\comp{1}w\comp{2}x\comp{1}y\comp{0}z=\bigl((u\comp{0}v)\comp{1}w\bigr)\comp{2}\bigl(x\comp{1}(y\comp{0}z)\bigr)\ .$$
Cette convention combinée avec la propriété d'associativité des opérations \smash{$\comp{i}$} permet de supprimer beaucoup de parenthèses. Néanmoins, pour faciliter la lecture, on n'abusera pas de cette règle.
\smallbreak

Pour tout $\infty$\nbd-foncteur $F:C\to D$ et tout $i\geq0$, on note $F_i:C_i\to D_i$ la restriction de $F$ aux $i$\nbd-cellules.
\end{paragr}

\begin{paragr}\label{defcelparal}
On dit que deux $i$\nbd-cellules, $i\geq0$, d'une $\infty$\nbd-catégorie $C$, sont \emph{parallèles} si ou bien $i=0$, ou bien $i>0$ et elles ont même source et même but. Pour tout couple de $i$\nbd-cellules parallèles $x,y$, on note $\sHom_C(x,y)$ la $\infty$\nbd-catégorie dont les $j$\nbd-cellules, $j\geq0$, sont les $(i+j+1)$-flèches de $C$ de $i$\nbd-cellule source itérée $x$ et de $i$\nbd-cellule but itéré~$y$ (les sources, buts, unités et compositions étant induits de ceux de $C$). On note $\Hom_C(x,y)$ l'ensemble des objets de cette $\infty$\nbd-catégorie, autrement dit, l'ensemble des $(i+1)$\nbd-flèches de $C$ de source $x$ et de but $y$. Si $x$ est une $i$\nbd-cellule de $C$, la \emph{$\infty$\nbd-catégorie des $i$\nbd-cellules parallèles à $x$}, notée $\Par_C(x)$, est la $\infty$\nbd-catégorie dont les objets sont les $i$\nbd-cellules de $C$ parallèles à $x$ et dont les $j$\nbd-flèches sont les $(i+j)$\nbd-flèches de $C$ de $i$\nbd-cellule source itérée et de $i$\nbd-cellule but itéré des $i$\nbd-cellules parallèles à $x$. En particulier, si $i=0$, autrement dit si $x$ est un objet de $C$, on a $\Par_C(x)=C$ et si $i>0$, on a $\Par_C(x)=\sHom_C(s(x),t(x))$. Si $y$ et $z$ sont deux $i$\nbd-cellules de $C$ parallèles à $x$, on a $\sHom_{\Par_C(x)}(y,z)=\sHom_C(y,z)$.
\end{paragr}

\begin{paragr}\label{defnCat}\label{tronq}
Soit $n\geq0$ un entier. Une $n$\nbd-catégorie est une $\infty$\nbd-catégorie $C$ dont toute $i$\nbd-flèche, pour $i>n$, est une identité. En particulier, une $0$\nbd-catégorie n'est rien d'autre qu'un ensemble et une $1$\nbd-catégorie une catégorie ordinaire. Ainsi, les ensembles, les catégories ou les ensembles ordonnés seront souvent considérés comme des $\infty$\nbd-catégories. Pour revenir au cas général, si $i$ est un entier tel que \hbox{$0\leq i<n$}, et $x,y$ deux $i$\nbd-cellules d'une $n$\nbd-catégorie $C$, la $\infty$\nbd-catégorie $\sHom_C(x,y)$ est une \hbox{$(n-i-1)$}\nbd-catégorie. Pour $i$ un entier tel que $0\leq i\leq n$, et $x$ une $i$\nbd-cellule de $C$, la $\infty$\nbd-catégorie $\Par_C(x)$ est une $(n-i)$\nbd-catégorie.
\smallbreak

On note $\nCat{n}$ la sous-catégorie pleine de $\ooCat$ formée des $n$\nbd-catégories. Le foncteur d'inclusion $\incl{n}:\nCat{n}\to\ooCat$ admet à la fois un adjoint à gauche et un adjoint à droite. L'adjoint à droite $\tb{n}$ associe à toute petite $\infty$\nbd-catégorie $C$ le $n$\nbd-\emph{tronqué bête} de $C$, $n$\nbd-catégorie dont les $i$\nbd-cellules, $0\leq i\leq n$, sont les $i$\nbd-cellules de $C$, les sources, buts, unités et compositions venant de ceux dans $C$. L'adjoint à gauche $\ti{n}$ associe à $C$ le $n$\nbd-\emph{tronqué intelligent} de $C$, $n$\nbd-catégorie dont les $i$\nbd-cellules, $0\leq i<n$, sont les $i$\nbd-cellules de $C$ et les $n$\nbd-cellules sont les $n$\nbd-cellules de $C$, modulo la relation d'équivalence identifiant deux $n$\nbd-cellules de $C$ s'il existe un zigzag de $(n+1)$\nbd-flèches de $C$ les reliant. Les sources, buts, unités et compositions sont induits par ceux de $C$. Les morphismes d'adjonction
$$\tb{n}(C)=\incl{n}\tb{n}(C)\toto C\qquad\hbox{et}\qquad C\toto\incl{n}\ti{n}(C)=\ti{n}(C)$$
induisent, pour tout $i\geq0$, respectivement une injection et une surjection des ensembles des $i$\nbd-cellules. En considérant $\tb{n}(C)$ comme une sous-$\infty$\nbd-catégorie de $C$, on a une suite d'inclusions
$$\tb{0}(C)\hookrightarrow\tb{1}(C)\hookrightarrow\cdots\hookrightarrow\tb{n}(C)\hookrightarrow\tb{n+1}(C)\hookrightarrow\cdots\hookrightarrow C$$
faisant de $C$ une limite inductive des $\tb{n}(C)$. Le foncteur $\tb{n}$ admet lui-même un adjoint à droite qui associe à une $n$\nbd-catégorie $C$ la $(n+1)$\nbd-catégorie dont le $n$\nbd-tronqué bête est $C$ et qui  pour tout couple $(x,y)$ de $n$\nbd-cellules parallèles admet exactement une $(n+1)$\nbd-flèche de source $x$ et but $y$. En particulier, le foncteur de troncation bête commute à la fois aux limites projectives et aux limites inductives.
\end{paragr}

\begin{paragr}\label{pseudofinal}
On définit la notion de \emph{$n$\nbd-catégorie admettant un objet quasi-final} (resp. \emph{quasi-initial}) ainsi que celle d'\emph{objet quasi-final} (resp. \emph{quasi-initial}) d'une telle $n$\nbd-catégorie par récurrence sur $n$.
\begin{itemize}
\item Une $0$\nbd-catégorie admettant un objet quasi-final (resp. quasi-initial) est une $0$\nbd-catégorie ayant exactement un objet (autrement dit, un ensemble à un élément) et cet unique objet en est alors un objet quasi-final (resp. quasi-initial).
\item Pour $n>0$, une $n$\nbd-catégorie $C$ admet un objet quasi-final (resp. quasi-initial) s'il existe un objet $x$ de $C$ tel que pour tout objet $y$ de $C$ la $(n-1)$\nbd-catégorie $\sHom_C(y,x)$ (resp. $\sHom_C(x,y)$) admette un objet quasi-final (resp. quasi-initial), et un tel objet $x$ est alors un objet quasi-final (resp. quasi-initial) de~$C$.
\end{itemize}
Si on déroule cette définition, on remarque que, pour $n=1$, une catégorie admet un objet quasi-final (resp. quasi-initial) si et seulement si elle admet un objet final (resp. initial) et qu'alors un objet quasi-final (resp. quasi-initial) est exactement un objet final (resp. initial). Pour $n=2$, une $2$\nbd-catégorie admet un objet quasi-final (resp. quasi-initial) si et seulement si, dans la terminologie de~\cite{ChicheThese}, elle admet un objet admettant un objet final (resp. coadmet un objet admettant un objet initial). La notion d'objet quasi-final d'une $2$\nbd-catégorie a été introduite dans~\cite{Jay}, sous le nom d'objet final local.
\end{paragr}

\begin{paragr}\label{prodpseudo}
On vérifie immédiatement qu'un produit de $n$\nbd-catégories admettant un objet quasi-final (resp. quasi-initial) admet un objet quasi-final (resp. quasi-initial).
\end{paragr}

\begin{prop}\label{tronqpseudofinal}
Soient $n\geq0$ un entier, et $C$ une $n$\nbd-catégorie admettant un objet quasi-final \emph{(resp.} quasi-initial\emph{).} Alors pour tout $m\geq0$, la $m$\nbd-catégorie $m$\nbd-tronqué intelligent de $C$ admet un objet quasi-final \emph{(resp.} quasi-initial\emph{).}
\end{prop}

\begin{proof}
On raisonne par récurrence sur $n$. Pour $m\geq n$, il n'y a rien à
démontrer puisque $\ti{m}(C)=C$, ce qui prouve en particulier l'assertion
pour~$n=0$. Si $n>0$, l'hypothèse sur $C$ implique qu'il existe un objet $x$
de $C$ tel que pour tout objet $y$ de $C$ la $(n-1)$\nbd-catégorie
$\sHom_C(y,x)$ (resp. $\sHom_C(x,y)$) admette un objet quasi-final (resp.
quasi-initial). Si $m=0$, l'assertion est évidente puisque l'ensemble
$\ti{0}(C)$ a alors exactement un élément. Si $m>0$, l'hypothèse de
récurrence appliquée à la $(n-1)$\nbd-catégorie $\sHom_C(y,x)$ (resp.
$\sHom_C(x,y)$) implique que la
$(m-1)$\nbd-catégorie~\hbox{$\sHom_{\ti{m}(C)}(y,x)=\ti{m-1}(\sHom_C(y,x))$}
(resp. \hbox{$\sHom_{\ti{m}(C)}(x,y)=\ti{m-1}(\sHom_C(x,y))$}) admet un objet
quasi-final (resp. quasi-initial), ce qui prouve l'assertion.
\end{proof}

\begin{paragr}\label{engparcomp}
Soit $C$ une $\infty$\nbd-catégorie. Un \emph{ensemble multiplicatif} de cellules de $C$ est un ensemble $M$ de cellules de $C$ satisfaisant aux deux conditions suivantes:
\begin{itemize}
\item[a)] pour tout $i\geq0$, si $x$ est une $i$\nbd-cellule de $C$ appartenant à $M$, la \hbox{$(i+1)$}\nbd-flèche identité de $x$ appartient aussi à $M$;
\item[b)] pour tout couple d'entiers $i,j$ tels que $0\leq j<i$, si $x$ et $y$ sont deux $i$\nbd-flèches $j$\nbd-composables de $C$ appartenant à $M$, alors la $i$\nbd-flèche composée $x\comp{j}y$ appartient à $M$.
\end{itemize}
On dit qu'un ensemble $E$ de cellules d'une $\infty$\nbd-catégorie $C$ \emph{engendre $C$ par compositions} si l'ensemble de toutes les cellules de $C$ est le plus petit ensemble multiplicatif de cellules de $C$ contenant $E$. Dans ce cas, $E$ contient forcément l'ensemble des objets de~$C$, et pour tout $i>0$, toute $i$\nbd-flèche de $C$ est un composé d'un nombre fini de $i$\nbd-flèches, qui sont dans $E$ ou sont des unités itérées de $j$\nbd-cellules, $0\leq j<i$, appartenant à $E$.
\end{paragr}

\begin{paragr}\label{englibr}
Soient $C$ une $\infty$\nbd-catégorie et $E$ un ensemble de cellules de $C$. Pour $i\geq0$, on pose $E_i=E\cap C_i$. On dit que l'ensemble de cellules $E$ \emph{engendre librement $C$ au sens des polygraphes} si les deux conditions suivantes sont satisfaites:
\begin{itemize}
\item[a)] $E_0=C_0$;
\item[b)] pour toute $\infty$\nbd-catégorie $D$, tout $i\geq0$, tout $\infty$\nbd-foncteur $F:\tb{i}(C)\to D$, et toute application $f:E_{i+1}\to D_{i+1}$ tels que pour tout $x\in E_{i+1}$, on ait $s(f(x))=F(s(x))$ et $t(f(x))=F(t(x))$, il existe un et un seul $\infty$\nbd-foncteur $F':\tb{i+1}(C)\to D$ tel que 
$$F'|\tb{i}(C)=F\qquad\hbox{et}\qquad F'|E_{i+1}=f$$ 
($\tb{i}(C)$ et $\tb{i+1}(C)$ étant considérées comme des sous-$\infty$\nbd-catégories de $C$ par les morphismes d'adjonction (cf.~\ref{tronq})).
\end{itemize}
\end{paragr}

\begin{prop}
Soient $C$ une $\infty$\nbd-catégorie et $E$ un ensemble de cellules de $C$ qui l'engendre librement au sens des polygraphes. Alors $E$ engendre $C$ par compositions.
\end{prop}

\begin{proof}
Pour $i\geq0$, on pose $E_i=E\cap C_i$ et \smash{$E_{\leq i}=\mathop{\cup}\nolimits_{\,0\leq j\leq i}E_j$}. On va montrer par récurrence sur $i\geq0$ que l'ensemble $E_{\leq i}$ engendre $\tb{i}(C)$ par compositions, ce qui prouvera la proposition. Pour $i=0$, l'assertion est évidente grâce à la condition (a) de~\ref{englibr}. Supposons donc l'assertion démontrée pour $i$ et démontrons-la pour~$i+1$. Soit~$D$ la sous-$(i+1)$\nbd-catégorie de $\tb{i+1}(C)$ dont les $j$\nbd-cellules, $0\leq j\leq i$, sont les $j$\nbd-cellules de~$C$, et dont les $(i+1)$\nbd-flèches sont les composés d'éléments de~$E_{i+1}$ et d'identités itérées de $j$\nbd-cellules de $C$, $0\leq j\leq i$. Par hypothèse de récurrence l'ensemble $E_{\leq i+1}$ engendre $D$ par compositions. Il suffit donc de montrer que $\tb{i+1}(C)=D$, ou encore que l'inclusion $G:D\to\tb{i+1}(C)$ admet une section. Par définition, on a une inclusion $F:\tb{i}(C)\to D$ et une inclusion $f:E_{i+1}\to D_{i+1}$ satisfaisant aux hypothèses de la condition (b) de~\ref{englibr}. On en déduit l'existence d'un $\infty$\nbd-foncteur $F':\tb{i+1}(C)\to D$ tel que $F'|\tb{i}(C)=F$ et $F'|E_{i+1}=f$. La restriction de $GF'$ à $\tb{i}(C)$ est donc l'inclusion $\tb{i}(C)\to\tb{i+1}(C)$ et la restriction à $E_{i+1}$ l'inclusion $E_{i+1}\to C_{i+1}$, ce qui, en vertu de la partie unicité de la condition (b) de~\ref{englibr}, prouve que $F'$ est une section de $G$, et achève la démonstration.
\end{proof}

\begin{prop}\label{englibrclef}
Soient $A$ une petite catégorie, $P:A\to\ooCat$ un foncteur, $C$ une $\infty$\nbd-catégorie, et $\alpha:P\to C$ un morphisme de foncteurs de $P$ vers le foncteur constant de valeur $C$. On suppose fixés un ensemble de cellules $E$ qui engendre librement $C$ au sens des polygraphes, et pour tout objet $a$ de $A$, un ensemble de cellules~$E_a$ qui engendre librement $P(a)$ au sens des polygraphes, satisfaisant aux conditions suivantes : 
\begin{itemize}
\item[\emph{a)}] pour tout objet $a$ de $A$, $\alpha_a(E_a)\subset E$;
\item[\emph{b)}] pour toute flèche $u:a\to b$ de $A$, $P(u)(E_a)\subset E_b$;
\item[\emph{c)}] l'application $\limind_AE_a\to E$, induite par $\alpha$, est bijective.
\end{itemize}
Alors le morphisme $\limind_A\alpha:\limind_AP\to C$ est un isomorphisme de $\ooCat$.
\end{prop}

\begin{proof}
Pour tout $i\geq0$, on pose $E_i=E\cap C_i$ et pour tout objet $a$ de~$A$, $E_{a,i}=E_a\cap P(a)_i$. Par hypothèse, pour tout $i\geq0$, l'application $\limind_AE_{a,i}\to E_i$, induite par $\alpha$, est bijective. On va montrer par récurrence sur $i$ que pour tout $i\geq0$, le tronqué bête (cf.~\ref{tronq})
$$
\tb{i}\limind_A\alpha:\tb{i}\limind_AP\simeq\limind_A\tb{i}P\toto\tb{i}C
$$
de $\limind_A\alpha$ est un isomorphisme, ce qui prouvera la proposition. Pour $i=0$, cela résulte de la bijectivité de l'application $\limind_AE_{a,0}\to E_0$ et du fait qu'en vertu de la condition~(a) de~\ref{englibr}, on a $C_0=E_0$, et pour tout $a$ dans $A$, $P(a)_0=E_{a,0}$. Supposons l'assertion établie pour $i$ et prouvons-la pour $i+1$. Soit $D$ une $\infty$\nbd-catégorie, et \hbox{$\beta:\tb{i+1}P\to D$} un morphisme de foncteurs vers le foncteur constant de valeur~$D$. Il s'agit de montrer qu'il existe un unique $\infty$\nbd-foncteur $G:\tb{i+1}C\to D$ tel que pour tout objet $a$ de $A$ on ait $\beta_a=G\circ\tb{i+1}\alpha_a$. En vertu de l'hypothèse de récurrence, il existe un unique $\infty$\nbd-foncteur $F:\tb{i}C\to D$ tel que pour tout objet $a$ de $A$ on ait $\beta_a|\tb{i}P(a)=F\circ\tb{i}\alpha_a$. D'autre part, la bijectivité de l'application $\limind_AE_{a,i+1}\to E_{i+1}$ implique l'existence d'une unique application $f:E_{i+1}\to D_{i+1}$ telle que pour tout objet $a$ de $A$ et tout $x_a\in E_{a,i+1}$, on ait $\beta_a(x_a)=f(\alpha_a(x_a))$. Or, pour tout $x\in E_{i+1}$, il existe un objet $a$ de $A$ et $x_a\in E_{a,i+1}$ tel que $x=\alpha_a(x_a)$. On a donc
$$\begin{aligned}
s(f(x))&=s(f(\alpha_a(x_a)))=s(\beta_a(x_a))=\beta_a(s(x_a))\cr
&=F\alpha_a(s(x_a))=F(s(\alpha_a(x_a)))=F(s(x))\ ,
\end{aligned}$$
et de même $t(f(x))=F(t(x))$. Comme $E$ engendre librement $C$ au sens des polygraphes, en vertu de la condition~(b) de~\ref{englibr}, il existe donc un unique $\infty$\nbd-foncteur $G:\tb{i+1}C\to D$ tel que 
$$G|\tb{i}C=F\qquad\hbox{et}\qquad G|E_{i+1}=f\ .$$
Il reste à prouver que pour tout objet $a$ de $A$, on a $\beta_a=G\circ\tb{i+1}\alpha_a$. Or, on a les égalités
$$\beta_a|\tb{i}P(a)=F\circ\tb{i}\alpha_a=G|\tb{i}C\circ\tb{i}\alpha_a=G\circ\tb{i+1}\alpha_a|\tb{i}P(a)\ ,$$
et pour tout $x_a\in E_{a,i+1}$, on a $\beta_a(x_a)=f(\alpha_a(x_a))=G(\alpha_a(x_a))$, d'où 
$$\beta_a|E_{a,i+1}=G\circ\tb{i+1}\alpha_a|E_{a,i+1}\ .$$
Comme $E_{a}$ engendre librement $P(a)$ au sens des polygraphes, en vertu de la partie unicité de la condition~(b) de~\ref{englibr}, on a donc $\beta_a=G\circ\tb{i+1}\alpha_a$, ce qui achève la démonstration.
\end{proof}

\begin{rem}
On vérifie facilement que la proposition reste vraie si l'on remplace l'hypothèse que pour tout objet $a$ de $A$, l'ensemble $E_a$ engendre $P(a)$ librement au sens des polygraphes par l'hypothèse plus faible qu'il l'engendre par compositions. En revanche, on doit toujours supposer que l'ensemble $E$ engendre librement $C$ au sens des polygraphes. On n'aura pas besoin de ce résultat dans la suite.
\end{rem}

\begin{paragr}\label{dualooCat}
Soit $C$ une $\infty$\nbd-catégorie. On note $\op{C}$ la $\infty$\nbd-catégorie \emph{duale} de $C$, obtenue en inversant le sens des $i$\nbd-flèches pour \emph{tout} $i>0$. On définit ainsi un endofoncteur involutif $\Dual:\ooCat\to\ooCat$. Plus généralement, pour toute partie $\K$ de~\hbox{$\mathbb{N}^*=\mathbb{N}\sauf\{0\}$}, on définit un endofoncteur involutif $\KDual{\K}:\ooCat\to\ooCat$ en associant à une $\infty$\nbd-catégorie $C$ la $\infty$\nbd-catégorie obtenue à partir de $C$ en inversant le sens des $i$\nbd-flèches pour $i\in \K$. On dira que $\KDual{\K}(C)$ est le \emph{$\K$\nbd-dual} de $C$. Par définition, on a~$\Dual=\KDual{\mathbb{N}^*}$. Deux autres cas particuliers sont importants: le \emph{dual pair} et le \emph{dual impair} de $C$. Le dual pair (resp. impair) de $C$, noté $\opp{C}$ (resp.~$\opi{C}$), est le $\K$\nbd-dual de~$C$ pour $\K=2\mathbb{N}^*$ (resp.~pour $\K=2\mathbb{N}^*-1$). On remarque que $\op{C}=\opi{(\opp{C})}=\opp{(\opi{C})}$.
\smallbreak

Pour tout entier $n\geq0$, une $n$\nbd-catégorie $C$ admet un objet quasi-final si et seulement si la $n$\nbd-catégorie $\op{C}$ admet un objet quasi-initial, et un objet de $C$ est quasi-final si et seulement si il est un objet quasi-initial de $\op{C}$.
\end{paragr}

\section{La catégorie des complexes dirigés augmentés}

\begin{paragr}\label{ADCdef}
Dans~\cite{Steiner}, Steiner a introduit la notion de complexe dirigé augmenté. On rappelle qu'un \emph{complexe dirigé augmenté} est un triplet $(K,K^*,\augm)$, où
$$K=\kern 10pt
\xymatrix{
\cdots\ar[r]^-{d_{n+1}}
&K_n\ar[r]^-{d_{n}}
&K_{n-1}\ar[r]^-{d_{n-1}}
&\cdots\ar[r]^-{d_{2}}
&K_1\ar[r]^-{d_{1}}
&K_0
}$$
est un complexe de chaînes de groupes abéliens en degrés positifs, $\augm:K_0\to\mathbb{Z}$ une augmentation (de sorte que $d_nd_{n+1}=0$ pour $n\geq1$, et $\augm d_1=0$), et $K^*=(K_n^*)^{}_{n\in\mathbb{N}}$, est un sous-monoïde gradué du monoïde gradué sous-jacent au complexe $K$, autrement dit, pour tout $n\geq0$, $K_n^*$ est un sous-monoïde du groupe abélien $K_n$. On ne demande aucune compatibilité des différentiels $d_n$ ou de l'augmentation avec les $K_n^*$. Le complexe $K$ sera considéré souvent comme un complexe indexé par $\mathbb{Z}$ en posant $K_i=0$ pour $i<0$ et $d_i=0$ pour $i\leq0$:
$$\xymatrix{
\cdots\ar[r]^-{d_{n+1}}
&K_n\ar[r]^-{d_{n}}
&K_{n-1}\ar[r]^-{d_{n-1}}
&\cdots\ar[r]^-{d_{2}}
&K_1\ar[r]^-{d_{1}}
&K_0\ar[r]
&0\ar[r]
&0\ar[r]
&\cdots
}\ .$$
Le complexe dirigé augmenté $(K,K^*,\augm)$ sera noté souvent, par abus, simplement $K$.
\smallbreak

Un \emph{morphisme de complexes dirigés augmentés} $f:(K,K^*,\augm)\to(K',K'^*,\augm')$ est un morphisme de complexes $f:K\to K'$ compatible à l'augmentation (au sens où $\augm'f_0=\augm$) et tel que pour tout $n\geq0$, on ait $f_n(K^*_n)\subset K'^*_n$. On note $\ADC$ la catégorie des complexes dirigés augmentés. 
\end{paragr}

\begin{paragr}\label{monoADC}\label{induitADC}
Un morphisme de complexes dirigés augmentés $i:(K',K'^*,\augm')\to(K,K^*,\augm)$ est un monomorphisme si et seulement si pour tout $n\geq0$, l'application $i_n:K'_n\to K_n$ est injective. On dit alors que $(K',K'^*,\augm')$ est un \emph{sous-complexe dirigé augmenté} de~$(K,K^*,\augm)$. On dit que le monomorphisme $i$ est \emph{strict} (et que le sous-complexe dirigé augmenté $(K',K'^*,\augm')$ de $(K,K^*,\augm)$ est \emph{strict}) si pour tout $n\geq0$, on a l'égalité $K'^*_n=i^{-1}_n(K^*_n)$. On vérifie facilement que cette terminologie est compatible avec la notion générale de monomorphisme strict de~\cite{SGA4-1}.
\smallbreak

Soient $(K,K^*,\augm)$ un complexe dirigé augmenté et $i:(K',\augm')\to(K,\augm)$ un monomorphisme de complexes augmentés. Alors il existe une unique structure de complexe dirigé augmenté sur $(K',\augm')$ faisant de $i$ un monomorphisme strict de complexes dirigés augmentés. On dira qu'elle est la structure \emph{induite} par celle de $(K,K^*,\augm)$. Elle est définie en posant $K'^*_n=i^{-1}_n(K^*_n)$, pour $n\geq0$.
\end{paragr}

\begin{paragr}\label{quotientADC}
Soit $i:(K',K'^*,\augm')\to(K,K^*,\augm)$ un monomorphisme de complexes dirigés augmentés. On suppose que $\augm'=0$ (ce qui est par exemple le cas si $K'_0=0$). On définit alors un \emph{complexe dirigé augmenté quotient} $(K'',K''^*,\augm'')$ comme suit. Le complexe~$K''$ est le complexe quotient de $K$ par le sous-complexe $K'$. Pour $n\geq0$, le sous-monoïde $K''^*_n$ de $K''_n$ est l'image de $K^*_n$ par la surjection canonique $K_n\to K_n/K'_n$. L'augmentation $\augm''$ est déduite de l'augmentation $\augm$ par passage au quotient, ce qui est licite puisque par hypothèse $\augm\, i_0=\augm'=0$. La surjection canonique $K\to K''$ définit alors un morphisme de complexes dirigés augmentés satisfaisant à la propriété universelle des quotients.
\end{paragr}

\begin{paragr}\label{limindADC}
La catégorie $\ADC$ est cocomplète. Soit $K:I\to\ADC$, $i\mapsto(K(i),K(i)^*,\augm(i))$, un foncteur de source une petite catégorie $I$. La limite inductive $(L,L^*,\augm)$ de $K$ est obtenue comme suit. Le complexe $L$ est la limite inductive des $K(i)$ dans la catégorie des complexes. L'augmentation $\augm$ est définie par la propriété universelle des limites inductives. Pour $n\geq0$, $L^*_n$ est le sous-monoïde de $L_n$ engendré par la réunion des images des sous-monoïdes $K(i)^*_n$ par les morphismes canoniques $K(i)_n\to L_n$, pour~$i$ dans $I$. La propriété universelle se vérifie alors sans difficulté. On remarque que les foncteurs \og complexe sous-jacent\fg{} et \og complexe augmenté sous-jacent\fg{} commutent aux limites inductives.
\end{paragr}

\begin{paragr}\label{ADCfin}
La catégorie $\ADC$ admet un objet final $\ADCfin$ qui est défini par
$$
\ADCfin_n =\left\{\begin{matrix}
0\,,
&\hbox{ si }n>0\,,
\cr
\mathbb{Z}\,,
&\hbox{ si }n=0\,,
\end{matrix}\right.\qquad\ADCfin^*_0=\mathbb{Z}\qquad\hbox{et}\qquad \augm=1_{\mathbb{Z}}\ .
$$
En effet, pour tout complexe dirigé augmenté $(K,K^*,\augm)$, il existe un et un seul morphisme de complexes augmentés $p$ vers $\ADCfin$, qui est défini par $p_n=0$, pour $n>0$ et~$p_0=\augm$.  
\end{paragr}

\begin{paragr}\label{limprojADC}
Plus généralement la catégorie $\ADC$ admet des petites limites projectives. Soit $K:I\to\ADC$, $i\mapsto(K(i),K(i)^*,\augm(i))$, un foncteur de source une petite catégorie $I$. La limite projective $(L,L^*,\augm)$ de $K$ est obtenue comme suit. Le complexe augmenté $(L,e)$ est la limite projective des $(K(i),\augm(i))$ dans la catégorie des complexes augmentés. Explicitement, pour $n>0$, $L_n$ est la limite projective des $K(i)_n$ dans la catégorie des groupes abéliens. Pour $n=0$, $L_0$ est le sous-groupe de la limite projective des $K(i)_0$ dans la catégorie des groupes abéliens défini par 
$$L_0=\{(x_i)_{i\in\ob\,I}\in\limproj K(i)_0\,|\,\forall j,k\in\ob I\,,\ \augm(j)(x_j)=\augm(k)(x_k)\}\ .$$
Les sous-monoïdes $L^*_n$ sont définis de façon analogue. Pour $n>0$, $L^*_n$ est la limite projective des $K(i)^*_n$ dans la catégorie des monoïdes commutatifs. Pour $n=0$, $L^*_0$ est le sous-monoïde de la limite projective des $K(i)^*_0$ dans la catégorie des monoïdes commutatifs défini par 
$$L^*_0=\{(x_i)_{i\in\ob\,I}\in\limproj K(i)^*_0\,|\,\forall j,k\in\ob I\,,\ \augm(j)(x_j)=\augm(k)(x_k)\}\ .$$
La propriété universelle se vérifie facilement.
\end{paragr}

\begin{paragr}\label{ADClocpres}
On peut montrer sans difficulté que la catégorie $\ADC$ est localement présentable.
\end{paragr}

\begin{paragr}\label{ADCconc}
On dit qu'un morphisme $f:(K,K^*,\augm)\to(K',K'^*,\augm')$ de $\ADC$ est \emph{concentré en degré~$0$} si $f_n=0$ pour $n>0$. L'application $f\mapsto f_0$ établit une bijection entre les morphismes concentrés en degré $0$ et les morphismes de groupes abéliens \hbox{$f_0:K_0\to K'_0$} tels que $f_0d_1=0$, $f_0(K^*_0)\subset K'^*_0$ et $\augm'f_0=\augm$. On dira alors que $f$ est le \emph{morphisme concentré en degré $0$ défini par $f_0$}.  
\end{paragr}

\begin{paragr}\label{ADCdesc}
On dit qu'un complexe dirigé augmenté $(K,K^*,\augm)$ est \emph{décent} si
pour tout~$x$ dans $K^*_0$, on a $\augm(x)\in\mathbb{N}$. L'objet final $\ADCfin$ de la catégorie $\ADC$ n'est pas décent. La sous-catégorie pleine de $\ADC$ formée des objets décents admet un objet final $\ADCdescfin$ qui est défini par 
$$\ADCdescfin_n =\left\{\begin{matrix}
0\,,
&\hbox{ si }n>0\,,
\cr
\mathbb{Z}\,,
&\hbox{ si }n=0\,,
\end{matrix}\right.\qquad\ADCdescfin^{*}_0=\mathbb{N}\qquad\hbox{et}\qquad \augm=1_{\mathbb{Z}}\ .
$$
\end{paragr}

\begin{paragr}\label{ADCcons}
On dit qu'un morphisme $f:(K,K^*,\augm)\to(K',K'^*,\augm')$ entre complexes dirigés augmentés décents est \emph{constant} s'il se factorise par $\ADCdescfin$. Alors $f$ est un morphisme concentré en degré $0$, $f_0$ se factorise par $\mathbb{Z}$,
$$K_0\toto\mathbb{Z}\toto K'_0\ ,$$
et si $c_0$ est l'image dans $K'_0$ de $1\in\mathbb{Z}$ par la flèche de droite, alors $c_0\in K'^*_0$, $\augm'(c_0)=1$ et pour tout $x\in K_0$, on a $f_0(x)=\augm(x) c_0$. Réciproquement, pour tout $c_0\in K'^*_0$ tel que $\augm'(c_0)=1$, le morphisme de groupes abéliens 
$$f_0:K_0\toto K'_0\ ,\qquad x\longmapsto \augm(x) c_0\ ,$$
satisfait aux conditions $f_0d_1=0$, $f_0(K^*_0)\subset K'^*_0$ et $\augm'f_0=\augm$ du paragraphe~\ref{ADCconc}, et le morphisme de complexes dirigés augmentés $f$ concentré en degré $0$ défini par $f_0$ est un morphisme constant. On dira que $f$ est le \emph{morphisme constant défini par $c_0$} et que $c_0$ est la \emph{valeur} du morphisme constant $f$.
\end{paragr}

\begin{paragr}\label{ADChmt}
Soient $f,g:(K,K^*,\augm)\todouble(K',K'^*,\augm')$ deux morphismes de complexes dirigés augmentés de même source et même but. Une \emph{homotopie de $f$ vers $g$} est une famille $h=(h_n)_{n\in\mathbb{N}}$, où $h_n:K_n\to K'_{n+1}$ est un morphisme de groupes abéliens, telle que:
\begin{itemize}
\item[a)] pour tout $n\geq0$, on a $h_n(K^*_n)\subset K'^*_{n+1}$;
\item[b)] pour tout $n\geq0$ et tout $x\in K_n$, on a $(d'_{n+1}h_n+h_{n-1}d_n)(x)=g_n(x)-f_n(x)$ 
(en posant $h_{-1}=0$).
\end{itemize}
\smallbreak

Conformément à l'usage, quand aucune ambiguïté n'en résulte, on omettra les indices de $d$, $h$, $f$, $g$, etc.
\end{paragr}

\begin{paragr}\label{dualADC}
Soient $(K,K^*,\augm)$ un complexe dirigé augmenté, et $\K$ une partie de \hbox{$\mathbb{N}^*=\mathbb{N}\sauf\{0\}$}. Le \emph{$\K$\nbd-dual} de $(K,K^*,\augm)$ est le complexe dirigé augmenté $(K',K^*,\augm)$, où $K'$ est le complexe ayant même groupe abélien gradué sous-jacent que $K$ et dont la différentielle $d'$ est définie par 
$$d'_j=\left\{
\begin{aligned}
&-d_j\ ,\qquad j\in\K\ ,\cr
&\phantom{{}-{}}d_j\ ,\qquad j\in\mathbb{N}^*\sauf\K\ ,
\end{aligned}
\right.$$
autrement dit, $d'_j=(-1)^{\chi(j)}d_j$, où $\chi:\mathbb{N}^*\to\{0,1\}$ est la fonction caractéristique de~$\K$. On définit ainsi un endofoncteur involutif $\KDualADC{\K}:\ADC\to\ADC$ associant à $(K,K^*,\augm)$ son $\K$\nbd-dual. Le \emph{dual} du complexe dirigé augmenté $(K,K^*,\augm)$ est son $\mathbb{N}^*$\nbd-dual, et il est noté $\op{(K,K^*,\augm)}$. 
\end{paragr}

\begin{paragr}\label{dualhmtpADC}
Soient $f,g:(K,K^*,\augm)\todouble(K',K'^*,\augm')$ deux morphismes de complexes dirigés augmentés. On vérifie aussitôt qu'une homotopie $h$ de morphismes de complexes dirigés augmentés de $f$ vers $g$ induit une homotopie de complexes dirigés augmentés de $\op{g}$ vers $\op{f}$. Pour une partie arbitraire $\K$ de $\mathbb{N}^*$, l'homotopie $h$ ne définit pas en général une homotopie entre $\KDualADC{\K}(f)$ et $\KDualADC{\K}(g)$, mais seulement une variante \emph{$\K$\nbd-tordu} de la notion d'homotopie qu'on n'aura pas besoin d'expliciter ici (la variante $\mathbb{N}^*$\nbd-tordu d'une homotopie de $f$ vers $g$ étant exactement une homotopie de $g$ vers $f$).
\end{paragr}

\section{Rappels de la théorie de Steiner}

Tous les résultats de cette section sont dus à Steiner~\cite{Steiner}. Pour commencer, Steiner définit un couple de foncteurs adjoints
$$\Stlambda:\ooCat\toto\ADC\ ,\qquad\Stnu:\ADC\toto\ooCat$$
comme suit. 

\begin{paragr}\label{defStlambda}
Soit $C$ une $\infty$\nbd-catégorie. Pour $i\geq0$, le groupe de chaînes $\Stlambda(C)_i$ est le groupe abélien engendré par les éléments $\ch{x}$, pour $x$ une $i$\nbd-cellule de $C$, soumis aux relations $\ch{x\comp{j} y}=\ch{x}+\ch{y}$, pour $0\leq j<i$ et $x,y$ des $i$\nbd-cellules $j$\nbd-composables de $C$. Le sous-monoïde $\Stlambda(C)_i^*$ de $\Stlambda(C)_i$ est le sous-monoïde engendré par les $\ch{x}$, $x\in C_i$. Pour $i>0$, la différentielle $d_i:\Stlambda(C)_i\to\Stlambda(C)_{i-1}$ est définie sur les générateurs par $d_i(\ch{x})=\ch{t(x)}-\ch{s(x)}$. L'augmentation $\augm:\Stlambda(C)_0\to\mathbb{Z}$ est définie par $\augm(\ch{x})=1$, pour $x\in C_0$. On vérifie facilement que si $x$ est une $i$\nbd-cellule identité de $C$, on a $\ch{x}=0$, et par suite, si $C$ est une $n$\nbd-catégorie, on a $\Stlambda(C)_i=0$ pour $i>n$.
\smallbreak

Si $F:C\to D$ est un $\infty$\nbd-foncteur, le morphisme de complexes dirigés augmentés $\Stlambda(F):\Stlambda(C)\to\Stlambda(D)$ est défini en posant $\Stlambda(F)_i(\ch{x})=\ch{F(x)}$ pour $i\geq0$ et $x\in C_i$.
\end{paragr}

\begin{paragr}\label{defStnu}
Soit $(K,K^*,\augm)$ un complexe dirigé augmenté noté plus simplement $K$. Les $i$\nbd-cellules, $i\geq0$, de la $\infty$\nbd-catégorie $\Stnu(K)$ sont les tableaux
$$\cellule{x}=\tabld{x}{i}$$
tels que:
\begin{itemize}
\item[a)] $x^\e_k\in K^*_k$, pour $\e\in\{0,1\}$ et $0\leq k\leq i$;
\item[b)] $d^{}_k(x^\e_k)=x^1_{k-1}-x^0_{k-1}$, pour $\e\in\{0,1\}$ et $1\leq k\leq i$;
\item[c)] $\augm(x^\e_0)=1$, pour $\e\in\{0,1\}$;
\item[d)] $x^0_i=x^1_i$.
\end{itemize}
On posera souvent $x^{}_i=x^0_i=x^1_i$.
En particulier, les objets de $\Stnu(K)$ sont les tableaux~$\cellule{x}$ de la forme
$$\cellule{x}=\begin{pmatrix} x^{}_0\cr\noalign{\vskip 3pt} x^{}_0\end{pmatrix}\qquad\hbox{avec}\ \ x^{}_0\in K^*_0\,,\ \augm(x^{}_0)=1\ .$$
Ainsi, on identifiera l'ensemble des objets de $\Stnu(K)$ avec l'ensemble des éléments $x^{}_0$ de~$K^*_0$ tels que $\augm(x^{}_0)=1$.
En revenant au cas général d'une $i$\nbd-cellule $\cellule{x}$ comme ci-dessus, $i\geq0$, la $(i+1)$\nbd-flèche unité de $\cellule{x}$ est donnée par le tableau
$$1_{\cellule{x}}=\tablnu{x}{i}{0}\ .$$
Si $i>0$, la source et le but de la $i$\nbd-flèche $\cellule{x}$ sont respectivement les $(i-1)$\nbd-cellules
$$s(\cellule{x})=\tablnu{x}{i-2}{x^0_{i-1}}\quad\ \hbox{et}\quad\ t(\cellule{x})=\tablnu{x}{i-2}{x^1_{i-1}}\ .$$
Si $j$ est un entier tel que $0\leq j<i$, et si
$$\cellule{y}=\tabld{y}{i}$$
est une $i$\nbd-flèche de $\Stnu(K)$ dont la $j$\nbd-cellule but itérée est égale à la $j$\nbd-cellule source itérée de $x$, autrement dit, si
$$x^\e_k=y^\e_k\,,\ \e\in\{0,1\}\,,\ 0\leq k<j\qquad\hbox{et}\qquad x^0_j=y^1_j\ ,$$
alors le composé $\cellule{x}\comp{j}\cellule{y}$ est défini par
$$\cellule{x}\comp{j}\cellule{y}=
\begin{pmatrix}
y^0_0
&\dots
&y^0_j
&x^0_{j+1}+y^0_{j+1}
&\dots
&x^0_{i}+y^0_{i}
\cr\noalign{\vskip 3pt}
x^1_0
&\dots
&x^1_j
&x^1_{j+1}+y^1_{j+1}
&\dots
&x^1_{i}+y^1_{i}
\end{pmatrix}\ .
$$
On remarque que si $n\geq0$ est un entier tel que pour tout $i>n$, on ait $K_i=0$, alors $\Stnu(K)$ est une $n$\nbd-catégorie.
\smallbreak

Si $f:K\to K'$ est un morphisme de complexes dirigés augmentés, l'image de la $i$\nbd-cellule $\cellule{x}$ de $\Stnu(K)$ par le $\infty$\nbd-foncteur $\Stnu(f)$ est définie par le tableau
$$\Stnu(f)(\cellule{x})=
\begin{pmatrix}
f(x^0_1)
&\dots
&f(x^0_{i-1})
&f(x^0_{i})
\cr\noalign{\vskip 3pt}
f(x^1_1)
&\dots
&f(x^1_{i-1})
&f(x^1_{i})
\end{pmatrix}\ .$$
\end{paragr}

\begin{prop}\label{adjSteiner}
Le couple $(\Stlambda,\Stnu)$ est un couple de foncteurs adjoints.
\end{prop}

\begin{proof}
Voir~\cite[théorème~2.11]{Steiner}.
\end{proof}

\begin{paragr}\label{baseADC}
Une \emph{base} d'un complexe dirigé augmenté $(K,K^*,\augm)$ est un sous-ensemble gradué $B=(B_i)_{i\geq0}$ de l'ensemble gradué sous-jacent au complexe $K$ tel que 
\begin{itemize}
\item[a)] pour tout $i\geq0$, $B_i$ est une base du $\mathbb{Z}$\nbd-module $K_i$;
\item[b)] pour tout $i\geq0$, $B_i$ engendre le sous-monoïde $K^*_i$ de $K_i$.
\end{itemize}
\end{paragr}

\begin{paragr}\label{ordpartADC}
Soit $(K,K^*,\augm)$ un complexe dirigé augmenté. Pour $i\geq0$, le sous-monoïde~$K^*_i$ induit une relation de préordre $\leq$ sur $K_i$, compatible avec sa structure de groupe, définie par
$$x\leq y\ \Longleftrightarrow\ y-x\in K^*_i\ ,$$
et alors on a
$$K^*_i=\{x\in K_i\,|\,x\geq0\}\ .\ \footnotemark$$
\footnotetext{
On remarquera, d'ailleurs, que se donner une structure de complexe dirigé augmenté sur un complexe augmenté $(K,\augm)$ revient à se donner, pour chaque $i\geq0$, une relation de préordre sur $K_i$, compatible avec la structure de groupe. Ce complexe dirigé augmenté sera décent si et seulement si l'augmentation induit un morphisme d'ensembles préordonnés de $K_0$ vers $\mathbb{Z}$, pour la relation d'ordre naturelle sur~$\mathbb{Z}$.
}%
Si ce complexe dirigé augmenté admet une base $B$, alors cette relation de préordre est une relation d'ordre faisant de $K_i$ un groupe réticulé (groupe ordonné tel que toute paire d'éléments $x,y$ admet une borne inférieure et une borne supérieure, qu'on notera respectivement $x\wedge y$ et $x\vee y$). L'ensemble $B_i$ est alors exactement l'ensemble des éléments minimaux de $K^*_i$. Ainsi, pour un complexe dirigé augmenté admettant une base $B$, cette dernière est unique et \emph{ne constitue pas} une donnée supplémentaire. On se gardera néanmoins de croire qu'un morphisme de complexes dirigés augmentés admettant des bases envoie nécessairement les élément de la base de sa source sur des éléments de la base de son but.
\end{paragr}

\begin{paragr}\label{baseunitADC}
Soit $(K,K^*,\augm)$ un complexe dirigé augmenté admettant une base $B$, et soit~$i\geq0$. Tout élément $x\in K_i$ se décompose de façon unique en
$$x=\Pos{x}-\Neg{x}\ ,\qquad\Pos{x}\,,\,\Neg{x}\geq0\ ,\quad \Pos{x}\wedge\Neg{x}=0\ $$
(si $x$ s'écrit comme combinaison linéaire d'éléments distincts de $B_i$, $\Pos{x}$ est la somme des termes à coefficients positifs et $-\Neg{x}$ la somme des termes à coefficients négatifs).
Soit toujours $x\in K_i$. On définit un tableau
$$\atom{x}=\tabll{\atom{x}}{i}\ ,$$
où les $\atom{x}^\e_k$, $\e\in\{0,1\}$, $0\leq k\leq i$, sont définis par récurrence descendante en posant
$$\begin{aligned}
&\atom{x}^0_i=\atom{x}^1_i=x\cr
&\atom{x}^0_{k-1}=\Neg{(d_{k}(\atom{x}^0_{k}))}\ ,\quad0<k\leq i\ ,\cr
&\atom{x}^1_{k-1}=\Pos{(d_{k}(\atom{x}^1_{k}))}\ ,\quad0<k\leq i\ .
\end{aligned}$$
On vérifie facilement que pour $\e\in\{0,1\}$ et $0<k\leq i$, on a $d_k(\atom{x}^\e_{k})=\atom{x}^1_{k-1}-\atom{x}^0_{k-1}$ et $\atom{x}^\e_{k-1}\in K^*_{k-1}$, ce qui implique que le tableau $\atom{x}$ est une $i$\nbd-cellule de $\Stnu(K)$ si et seulement si on a $x\in K^*_i$ et $\augm(\atom{x}^0_0)=\augm(\atom{x}^1_0)=1$. On dit que la base $B$ est \emph{unitaire} si pour tout $i\geq0$ et tout $b\in B_i$, $\atom{b}$ est une $i$\nbd-cellule de $\Stnu(K)$, autrement dit si $\augm(\atom{b}^0_0)=\augm(\atom{b}^1_0)=1$. En particulier, pour tout $b\in B_0$, on a alors $\augm(b)=1$, ce qui implique aussitôt que $B_0=\{b\in K^*_0\,|\,\augm(b)=1\}=\ob(\Stnu(K))$. Si $B$ est une base unitaire, $i\geq0$ un entier, et $b\in B_i$, on dit que la $i$\nbd-cellule $\atom{b}$ de $\Stnu(K)$ est \emph{l'atome} associé à $b$.
\end{paragr}

\begin{rem}
Un complexe dirigé augmenté $(K,K^*,\augm)$ admettant une base unitaire $B$ est décent. En effet, en vertu de ce qui précède, pour tout $b\in B_0$, on a  $\augm(b)=1$, et comme le monoïde $K_0^*$ est engendré par $B_0$, on en déduit que $\augm(K^*_0)\subset\mathbb{N}$.
\end{rem}

\begin{paragr}\label{defsansboucles}
Soit $(K,K^*,\augm)$ un complexe dirigé augmenté admettant une base $B$. Pour $i\geq0$, on note $\leqb{i}$ la relation de préordre sur $B$ engendrée par la relation $R_i$ définie par
$$a\,R_i\,b \ \Longleftrightarrow\ \hbox{il existe $j,k>i$ tels que $a\in B_j$, $b\in B_k$ et $\atom{a}^1_i\wedge\atom{b}^0_i>0$\ .}$$
On dit que la base $B$ est \emph{sans boucles} si pour tout $i\geq0$, la relation de préordre $\leqb{i}$ est une relation d'ordre. 
\end{paragr}

\begin{thm}\label{isoadjunitsansboucles}
Pour tout complexe dirigé augmenté $(K,K^*,\augm)$ admettant une base unitaire sans boucles, le morphisme d'adjonction $\Stlambda\Stnu(K)\to K$ est un isomorphisme. En particulier, la restriction du foncteur $\Stnu:\ADC\to\ooCat$ à la sous-catégorie pleine de $\ADC$ formée des complexes dirigés augmentés admettant une base unitaire sans boucles est pleinement fidèle.
\end{thm}

\begin{proof}
Voir~\cite[théorème 5.6]{Steiner}.
\end{proof}

\begin{thm}\label{unitsansbouclespolygr}
Soit $(K,K^*,\augm)$ un complexe dirigé augmenté admettant une base unitaire sans boucles $B$. Alors la $\infty$\nbd-catégorie $\Stnu(K)$ est librement engendrée au sens des polygraphes par l'ensemble des atomes $\atom{b}$, pour $b\in B$.
\end{thm}

\begin{proof}
Il s'agit d'une reformulation de~\cite[théorème 6.1]{Steiner}.
\end{proof}

\begin{paragr}\label{deffortsansboucles}
Soit $(K,K^*,\augm)$ un complexe dirigé augmenté admettant une base $B$. On note $\leqbb$ la relation de préordre sur $B$ engendrée par la relation $R_{\mathbb{N}}$ définie par
$$a\,R_{\mathbb{N}}\, b\ \Longleftrightarrow\ 
\left\{\begin{aligned}
&\hbox{ il existe $i>0$ tel que $a\in B_{i-1}\,,\ b\in B_i$ et $a\leq\Neg{(d_i(b))}$}\cr
&\hbox{ou}\cr
&\hbox{ il existe $i>0$ tel que $a\in B_i\,,\ b\in B_{i-1}$ et $\Pos{(d_i(a))}\geq b$}\ .
\end{aligned}\right.$$
On dit que la base $B$ est \emph{fortement sans boucles} si pour tout $i\geq0$, la relation de préordre $\leqbb$ est une relation d'ordre. 
\end{paragr}

\begin{prop}\label{propfortsansboucles}
Une base fortement sans boucles d'un complexe dirigé augmenté est sans boucles.
\end{prop}

\begin{proof}
Voir~\cite[proposition 3.7]{Steiner}.
\end{proof}

\begin{prop}\label{dualStnu}
Pour toute partie $\K$ de $\mathbb{N}^*=\mathbb{N}\sauf\{0\}$, le carré de foncteurs 
$$
\xymatrix{
\ADC\ar[r]^{\KDualADC{\K}}\ar[d]_{\Stnu}
&\ADC\ar[d]^{\Stnu}
\\
\ooCat\ar[r]_{\KDual{\K}}
&\ooCat
}
$$
\emph{(cf.~\ref{dualooCat} et~\ref{dualADC})}
est commutatif à isomorphisme canonique près. En particulier, pour tout complexe dirigé augmenté $K$, on a un isomorphisme canonique $\Stnu(\op{K})\simeq\op{\Stnu(K)}$.
\end{prop}

\begin{proof}
On vérifie immédiatement que pour $\K\subset\mathbb{N}^*$, et $K$ un complexe dirigé augmenté, les applications
$$
\cellule{x}=\tablg{x}{i}\longmapsto\,\,\cellule{y}=\tablg{y}{i},\ \ 
\begin{matrix}
\cellule{x}\in\bigl(\KDual{\K}\Stnu(K)\bigr)_i=\Stnu(K)_i\,,\cr
\noalign{\vskip 3pt}
i\geq0\,,
\end{matrix}
$$
où
$$
y^\e_i=\left\{\begin{aligned}
&x^{1-\e}_i\ ,\ \phantom{x^\e_i}\hbox{si $i\in\K$}\,,\cr
&x^\e_i\ ,\quad\phantom{x^{1-\e}_i}\hbox{sinon}\,,
\end{aligned}\right.\quad\ \e\in\{0,1\}\,,\ \ i\geq0\,, 
$$
définissent l'isomorphisme voulu.
\end{proof}

\section{Contraction d'un complexe dirigé augmenté}

Dans cette section, on se fixe un complexe dirigé augmenté $(K,K^*,\augm)$, noté simplement~$K$. On va dégager des conditions suffisantes sur $K$ pour que $\Stnu(K)$ admette un objet quasi-initial ou quasi-final.

\begin{paragr}\label{defnegl}
Soient $i\geq0$, et $h$ un morphisme de groupes abéliens de source $K_i$. On dit qu'une $i$\nbd-cellule
$$\cellule{x}=\tablnu{x}{i-1}{x_i}$$
de $\Stnu(K)$ est \emph{$h$\nbd-négligeable} si $h(x_i)=0$.
\end{paragr}

\begin{lemme}\label{lemmeclef1}
Soient $i>0$, et trois morphismes de groupes abéliens
$$h_{i-1}:K_{i-1}\toto K_i\ ,\quad h_i:K_i\toto K_{i+1}\ ,\quad h_{i+1}:K_{i+1}\toto K_{i+2}$$
satisfaisant aux conditions:
$$d_{i+1}h_i+h_{i-1}d_i=1_{K_i}\qquad\hbox{\emph{(resp.}\ \ }d_{i+1}h_i+h_{i-1}d_i=-1_{K_i}\ )\ ,\leqno(*)$$
$$h_i(K_i^*)\subset K^*_{i+1}\ ,\leqno(**)$$
$$h_{i+1}h_i=0\ .\leqno(*\kern -2pt*\kern -2pt*)$$
Si
$$\cellule{x}=\tablnu{a}{i-1}{x_i}\qquad\hbox{et}\qquad\cellule{y}=\tablnu{a}{i-1}{y_i}$$
sont deux $i$\nbd-cellules parallèles de $\Stnu(K)$, et si $\cellule{x}$ est $h_i$\nbd-négligeable, alors il existe une $(i+1)$\nbd-cellule $h_{i+1}$\nbd-négligeable $\cellule{z}$ de $C$ de source $\cellule{x}$ et de but $\cellule{y}$ \emph{(resp.} de source $\cellule{y}$ et de but~$\cellule{x}$\emph{)}. Plus précisément, si l'on pose $z_{i+1}=h_i(y_i)$ et
$$\cellule{z}=
\begin{pmatrix}
a^0_0
&\dots
&a^0_{i-1}
&x_i
&z_{i+1}
\cr\noalign{\vskip 3pt}
a^1_0
&\dots
&a^1_{i-1}
&y_i
&z_{i+1}
\end{pmatrix}
\qquad\hbox{\emph{(resp.}}\ \ 
\cellule{z}=
\begin{pmatrix}
a^0_0
&\dots
&a^0_{i-1}
&y_i
&z_{i+1}
\cr\noalign{\vskip 3pt}
a^1_0
&\dots
&a^1_{i-1}
&x_i
&z_{i+1}
\end{pmatrix}
\ ),
$$
alors $\cellule{z}$ est une $(i+1)$\nbd-cellule $h_{i+1}$\nbd-négligeable de $\Stnu(K)$ de $\cellule{x}$ vers $\cellule{y}$ \emph{(resp.} de $\cellule{y}$ vers~$\cellule{x}$\emph{)}.
\end{lemme}

\begin{proof}
Comme $y_i\in K^*_i$, on a, en vertu de la condition ($**$), $z_{i+1}\in K^*_{i+1}$. D'autre part, $d_ix_i=a^1_{i-1}-a^0_{i-1}=d_iy_i$, d'où $d_i(y_i-x_i)=0$ et par suite, comme $\cellule{x}$ est $h_i$\nbd-négligeable, on a
$$d_{i+1}z_{i+1}=d_{i+1}h_i(y_i-x_i)=(d_{i+1}h_{i}+h_{i-1}d_i)(y_i-x_i)\ .$$
On a donc, en vertu de la condition ($*$),
$$d_{i+1}z_{i+1}=y_i-x_i\qquad\hbox{(resp.}\ \ d_{i+1}z_{i+1}=x_i-y_i\ ).$$
On en déduit que $\cellule{z}$ est une $(i+1)$\nbd-cellule de $\Stnu(K)$ de $\cellule{x}$ vers $\cellule{y}$ (resp. de $\cellule{y}$ vers $\cellule{x}$). 
Enfin, en vertu de la condition $(*\kern -2pt*\kern -2pt*)$, la $(i+1)$\nbd-cellule $\cellule{z}$ est $h_{i+1}$\nbd-négligeable.
\end{proof}

\begin{prop}\label{existpsin1}
Soient $i>0$, et une famille de morphismes de groupes abéliens
$$h_j:K_j\toto K_{j+1}\ ,\quad j\geq i-1\ ,$$
satisfaisant aux conditions:
$$d_{j+1}h_j+h_{j-1}d_j=1_{K_j}\quad\hbox{\emph{(resp.}\ \ }d_{j+1}h_j+h_{j-1}d_j=-1_{K_j}\ )\ ,\qquad j\geq i\ ,\leqno(*_j)$$
$$h_j(K^*_j)\subset K^*_{j+1}\ ,\qquad j\geq i\ ,\leqno(**_j)$$
$$h_{j+1}h_j=0\ ,\qquad j\geq i\ .\leqno(*\kern -2pt*\kern -2pt*_j)$$
Si $n\geq i$ est un entier tel que pour tout $j>n$, $K_j=0$, alors toute $i$\nbd-cellule $\cellule{x}$ $h_i$\nbd-négligeable de la $n$\nbd-catégorie $\Stnu(K)$ est un objet quasi-initial \emph{(resp.} quasi-final\emph{)} de la \hbox{$(n-i)$}\nbd-catégorie $\Par_{\Stnu(K)}(\cellule{x})$ des $i$\nbd-cellules parallèles à $\cellule{x}$ \emph{(\emph{cf}.~\ref{defcelparal} et~\ref{defnCat})}.
\end{prop}

\begin{proof}
On raisonne par récurrence descendante sur $i$ de $n$ à $1$.
Soient $\cellule{x}$ une $i$\nbd-cellule $h_i$\nbd-négligeable de la $n$\nbd-catégorie $\Stnu(K)$ et $\cellule{y}$ une $i$\nbd-cellule parallèle à $\cellule{x}$. En vertu du lemme~\ref{lemmeclef1}, il existe une $(i+1)$\nbd-cellule $h_{i+1}$\nbd-négligeable $\cellule{z}$ de $\Stnu(K)$ de $\cellule{x}$ vers $\cellule{y}$ (resp.~de $\cellule{y}$ vers $\cellule{x}$). Si $i=n$, comme $\Stnu(K)$ est une $n$\nbd-catégorie, $\cellule{z}$ est forcément une identité, ce qui implique que $\cellule{y}=\cellule{x}$, et prouve l'assertion dans ce cas. Si $i<n$, l'hypothèse de récurrence, appliquée à la famille de morphismes de groupes abéliens $(h_j)_{j\geq i}$ et la $(i+1)$\nbd-cellule $z$, implique que $\cellule{z}$ est un objet quasi-initial (resp.~quasi-final) de la \hbox{$(n-i-1)$}\nbd-catégorie $\Par_{\Stnu(K)}(\cellule{z})$ qui n'est autre que $\sHom_{\Stnu(K)}(\cellule{x},\cellule{y})$ (resp.~$\sHom_{\Stnu(K)}(\cellule{y},\cellule{x})$\,), autrement dit $\sHom_{\Par_{\Stnu(K)}(\cellule{x})}(\cellule{x},\cellule{y})$ (resp.~$\sHom_{\Par_{\Stnu(K)}(\cellule{x})}(\cellule{y},\cellule{x})$\,). On en déduit que $\cellule{x}$ est un objet quasi-initial (resp.~quasi-final) de la $(n-i)$\nbd-catégorie $\Par_{\Stnu(K)}(\cellule{x})$.
\end{proof}

\begin{lemme}\label{lemmeclef2}
Soient $c_0$ une $0$\nbd-cellule de $\Stnu(K)$, et deux morphismes de groupes abéliens
$$h_0:K_0\toto K_1\ ,\quad h_1:K_1\toto K_2$$
satisfaisant aux conditions suivantes:
$$d_1h_0(x)=x-\augm(x)c_0\qquad\hbox{\emph{(resp.}\ \ }d_1h_0(x)=\augm(x)c_0-x\ )\ ,\qquad x\in K_0\ ,\leqno(*)$$
$$h_0(K_0^*)\subset K^*_{1}\ ,\leqno(**)$$
$$h_{1}h_0=0\ .\leqno(*\kern -2pt*\kern -2pt*)$$
Pour toute $0$\nbd-cellule $\cellule{y_0}$ de $\Stnu(K)$, il existe une $1$\nbd-cellule $h_{1}$\nbd-négligeable $\cellule{z}$ de $\Stnu(K)$ de source $\cellule{c_0}$ et de but $\cellule{y_0}$ \emph{(resp.} de source $\cellule{y_0}$ et de but~$\cellule{c_0}$\emph{)}. Plus précisément, si l'on pose $z_{1}=h_0(y_0)$ et
$$
\cellule{z}=
\begin{pmatrix}
c_0&z_1\cr
\noalign{\vskip 3pt}
y_0&z_1
\end{pmatrix}
\qquad\hbox{\emph{(resp.}}\ \ 
\cellule{z}=
\begin{pmatrix}
y_0&z_1\cr
\noalign{\vskip 3pt}
c_0&z_1
\end{pmatrix}
\ ),
$$
alors $\cellule{z}$ est une $1$\nbd-cellule $h_{1}$\nbd-négligeable de $\Stnu(K)$ de $\cellule{c_0}$ vers $\cellule{y_0}$ \emph{(resp.} de $\cellule{y_0}$ vers~$\cellule{c_0}$\emph{)}.
\end{lemme}

\begin{proof}
Comme $y_0\in K^*_0$, on a, en vertu de la condition ($**$), $z_{1}\in K^*_{1}$. D'autre part, en vertu de la condition $(*)$,
$$d_1(z_1)=y_0-\augm(y_0)c_0\qquad\hbox{(resp.}\ \ d_1(z_1)=\augm(y_0)c_0-y_0\ ).$$
Comme $\augm(y_0)=1$, on en déduit que $\cellule{z}$ est une $1$\nbd-cellule de $\Stnu(K)$ de $\cellule{c_0}$ vers $\cellule{y_0}$ (resp. de~$\cellule{y_0}$ vers $\cellule{c_0}$). 
Enfin, en vertu de la condition $(*\kern -2pt*\kern -2pt*)$, la $1$\nbd-cellule $\cellule{z}$ est $h_{1}$\nbd-négligeable.
\end{proof}

\begin{prop}\label{existpsin2}
Soient $c_0$ une $0$\nbd-cellule de $\Stnu(K)$, 
et une famille de morphismes de groupes abéliens
$$h_j:K_j\toto K_{j+1}\ ,\quad j\geq0\ ,$$
satisfaisant aux conditions:
$$d_1h_0(x)=x-\augm(x)c_0\qquad\hbox{\emph{(resp.}\ \ }d_1h_0(x)=\augm(x)c_0-x\ )\ ,\qquad x\in K_0\ ,\leqno(*_0)$$
$$d_{j+1}h_j+h_{j-1}d_j=1_{K_j}\quad\hbox{\emph{(resp.}\ \ }d_{j+1}h_j+h_{j-1}d_j=-1_{K_j}\ )\ ,\qquad j\geq 1\ ,\leqno(*_j)$$
$$h_j(K^*_j)\subset K^*_{j+1}\ ,\qquad j\geq 0\ ,\leqno(**_j)$$
$$h_{j+1}h_j=0\ ,\qquad j\geq 0\ .\leqno(*\kern -2pt*\kern -2pt*_j)$$
Si $n\geq0$ est un entier tel que pour tout $j>n$, $K_j=0$, alors $c_0$ est un objet quasi-initial \emph{(resp.} quasi-final\emph{)} de la $n$\nbd-catégorie $\Stnu(K)$. De plus, $c_0$ est une $0$\nbd-cellule négligeable de $\Stnu(K)$, et toute $0$\nbd-cellule négligeable de $\Stnu(K)$ est égale à $c_0$.
\end{prop}

\begin{proof}
En vertu du lemme précédent, les conditions $(*_0),\,(**_0),\,(*\kern -2pt*\kern -2pt*_0)$ impliquent que pour toute $0$\nbd-cellule $y_0$ de $\Stnu(K)$, il existe une $1$\nbd-cellule $h_{1}$\nbd-négligeable~$\cellule{z}$ de $\Stnu(K)$ de source $\cellule{c_0}$ et de but $\cellule{y_0}$ (resp. de source $\cellule{y_0}$ et de but~$\cellule{c_0}$). En vertu de la proposition~\ref{existpsin1}, les conditions $(*_j)$, $(**_j)$, $(*\kern -2pt*\kern -2pt*_j)$, $j>0$, impliquent que $z$ est un objet quasi-initial (resp.~quasi-final) de la \hbox{$(n-1)$}\nbd-catégorie $\Par_{\Stnu(K)}(\cellule{z})$ qui n'est autre que $\sHom_{\Stnu(K)}(\cellule{c_0},\cellule{y_0})$ (resp.~$\sHom_{\Stnu(K)}(\cellule{y_0},\cellule{c_0})$\,). On en déduit que $\cellule{c_0}$ est un objet quasi-initial (resp. quasi-final) de $\Stnu(K)$.
\smallbreak

Montrons que $c_0$ est une $0$\nbd-cellule $h_0$\nbd-négligeable de $\Stnu(K)$. En effet, en vertu de la condition $(*_0)$, on a 
$$d_1h_0(c_0)=c_0-\augm(c_0)c_0=0\ ,\qquad\hbox{(resp. \ }d_1h_0(c_0)=\augm(c_0)c_0-c_0=0\ ).$$
D'autre part, la condition $(*_1)$ implique que $d_2h_1h_0+h_0d_1h_0=h_0$, et par suite, en vertu de la condition $(*\kern -2pt*\kern -2pt*_0)$, $h_0d_1h_0=h_0$. On a donc $h_0(c_0)=h_0d_1h_0(c_0)=0$. Enfin, si $x_0$ est une $0$\nbd-cellule $h_0$\nbd-négligeable, la condition $(*_0)$ implique que 
$$x_0-c_0=x_0-\augm(x_0)c_0=d_1h_0(x_0)=0\,,\ \hbox{(resp. \ }c_0-x_0=\augm(x_0)c_0-x_0=d_1h_0(x_0)=0\ ),$$
d'où $x_0=c_0$.
\end{proof}

\begin{paragr}\label{defcontraction}
Les données et les hypothèses de la proposition précédente peuvent s'exprimer de façon beaucoup plus compacte.
Au complexe dirigé augmenté $K$, on associe un complexe non borné 
$$\cnb{K} =
\xymatrixcolsep{1.8pc}
\xymatrix{
\cdots\ar[r]^-{d_{n+1}}
&K_n\ar[r]^-{d_{n}}
&K_{n-1}\ar[r]^-{d_{n-1}}
&\cdots\ar[r]^-{d_{2}}
&K_1\ar[r]^-{d_{1}}
&K_0\ar[r]^-{\augm}
&\mathbb{Z}\ar[r]
&0\ar[r]
&\cdots
}\,,$$
défini par
$$\cnb{K}_j=\left\{
\begin{aligned}
&K_j\,,\phantom{0\mathbb{Z}} j\geq0\,,\cr
&\mathbb{Z}\,,\phantom{0K_j} j=-1\,,\cr
&0\,,\phantom{\mathbb{Z}K_j} j<-1\,,
\end{aligned}
\right.
\kern 20pt
\cnb{d}_j=\left\{
\begin{aligned}
&d_j\,,\phantom{0\augm} j\geq1\,,\cr
&\augm\,,\phantom{0d_j} j=0\,,\cr
&0\,,\phantom{\augm d_j} j<0\,,
\end{aligned}
\right.
$$
\emph{muni}, pour tout $j\in\mathbb{Z}$, du sous-monoïde $\cnb{K}^*_j$ du groupe abélien $\cnb{K}_j$ défini par
$$\cnb{K}^*_j=\left\{
\begin{aligned}
&K^*_j\,,\phantom{0\mathbb{N}} j\geq0\,,\cr
&\mathbb{N}\,,\phantom{0K^*_j} j=-1\,,\cr
&0\,,\phantom{\mathbb{N}K_j} j<-1\,.
\end{aligned}
\right.$$
Une \emph{contraction} (resp. une \emph{contraction duale}) du complexe dirigé augmenté $K$ est une homotopie de carré nul du complexe $\cnb{K}$ (resp. du complexe $\cnb{\op{K}}$, où $\op{K}$ désigne le complexe dirigé augmenté dual de $K$ (cf.~\ref{dualADC})), de l'endomorphisme nul vers l'endomorphisme identité, compatible aux sous-monoïdes $\cnb{K}^*_j$, $j\in\mathbb{Z}$, ce qui revient à la donnée d'une famille de morphismes de groupes abéliens $h_j:\cnb{K}_j\to\cnb{K}_{j+1}$, $j\geq-1$, 
$$
\xymatrixcolsep{1.6pc}
\xymatrix{
\cdots\ar[r]^-{\pm d_{j+2}}
&K_{j+1}\ar[r]^-{\pm d_{j+1}}
&K_j\ar[r]^-{\pm d_{j}}\ar[dl]_(.54){h_{j}\kern -3pt}
&K_{j-1}\ar[r]^-{\pm d_{j-1}}\ar[dl]^(.37){h_{j-1}}
&\cdots\ar[r]^-{\pm d_{2}}
&K_1\ar[r]^-{\pm d_{1}}
&K_0\ar[r]^-{\augm}\ar[ld]_(.54){h_0\kern -3pt}
&\mathbb{Z}\ar[r]\ar[ld]_(.54){h_{-1}\kern -3pt}
&0\ar[r]\ar[ld]^(.37){0}
&\cdots
\\
\cdots\ar[r]^-{\pm d_{j+2}}
&K_{j+1}\ar[r]^-{\pm d_{j+1}}
&K_j\ar[r]^-{\pm d_{j}}
&K_{j-1}\ar[r]^-{\pm d_{j-1}}
&\cdots\ar[r]^-{\pm d_{2}}
&K_1\ar[r]^-{\pm d_{1}}
&K_0\ar[r]^-{\augm}
&\mathbb{Z}\ar[r]
&0\ar[r]
&\cdots
}$$
satisfaisant aux conditions
$$\augm h_{-1}=1_{\mathbb Z}\ ,\leqno(\,\cnb{*}_{-1})$$
$$d_1h_0+h_{-1}\augm=1_{K_0}\quad\hbox{(resp.\ \ }-d_1h_0+h_{-1}\augm=1_{K_0}\ )\ ,\leqno(\,\cnb{*}_0)$$
$$d_{j+1}h_j+h_{j-1}d_j=1_{K_j}\quad\hbox{(resp.\ \ }-d_{j+1}h_j-h_{j-1}d_j=1_{K_j}\ )\ ,\qquad j\geq1\ ,\leqno(\,\cnb{*}_j)$$
$$h_{-1}(\mathbb{N})\subset K^*_0\ ,\leqno(\cnb{**}_{-1})$$
$$h_j(K^*_j)\subset K^*_{j+1}\ ,\quad j\geq0\ ,\leqno(\cnb{**}_j)$$
$$h_{j+1}h_j=0\ ,\qquad j\geq-1\ .\leqno(\cnb{*\kern -2pt*\kern -2pt*}_j)$$
La donnée du morphisme $h_{-1}:\mathbb{Z}\to K_0$ revient à la donnée de l'élément $c_0=h_{-1}(1)$. Pour $j\geq1$, les conditions $(\cnb{*}_j)$ coïncident avec les conditions $(*_j)$ de la proposition~\ref{existpsin2}. Pour $j=0$, la relation $(\cnb{*}_0)$ exprime que pour tout $x\in K_0$,
$$d_1h_0(x)+e(x).c_0=x\qquad\hbox{(resp.\ \ }-d_1h_0(x)+e(x).c_0=x\ ),$$
autrement dit, elle est équivalente à la condition $(*_0)$ de la proposition~\ref{existpsin2}. Pour $j\geq0$, les conditions~$(\cnb{**}_j)$ et~$(\cnb{*\kern -2pt*\kern -2pt*}_j)$ coïncident respectivement avec les conditions~$(**_j)$ et~$(*\kern -2pt*\kern -2pt*_j)$ de la proposition~\ref{existpsin2}. Pour $j=-1$, la condition $(\cnb{*}_{-1})$ équivaut à
$$\augm(c_0)=1\ ,\leqno(1)$$
la condition $(\cnb{**}_{-1})$ à
$$c_0\in K^*_0\ ,\leqno(2)$$
et la condition $(\cnb{*\kern -2pt*\kern -2pt*}_{-1})$ à
$$h_0(c_0)=0\ .\leqno(3)$$
On remarque que les conditions $(1)$, $(2)$ et $(3)$ signifient exactement que $c_0$ est une $0$\nbd-cellule $h_0$\nbd-négligeable de $\Stnu(K)$ (la condition (3) étant en vertu de la proposition~\ref{existpsin2}, conséquence des autres conditions). Ainsi, les propositions~\ref{existpsin2} et~\ref{existpsin1} impliquent le théorème suivant:
\end{paragr}

\begin{thm}\label{thmcontraction}
Soient $n\geq0$, et $K$ un complexe dirigé augmenté tel que \hbox{$K_j=0$}, pour $j>n$. Si $K$ admet une contraction \emph{(resp.} une contraction duale\emph{),} alors la $n$\nbd-catégorie $\Stnu(K)$ admet un objet quasi-initial \emph{(resp.} quasi-final\emph{).} Plus précisément, si $h$ est une contraction \emph{(resp.} une contraction duale\emph{)} de $K$, et si l'on pose \hbox{$c_0=h_{-1}(1)$}, alors $c_0$ est un objet quasi-initial \emph{(resp.} quasi-final\emph{)} de $\Stnu(K)$. De plus, pour tout $i$, $0\leq i\leq n$, si $x$ est une $i$\nbd-cellule $h_i$\nbd-négligeable de $\Stnu(K)$, alors $x$ est un objet quasi-initial \emph{(resp.} quasi-final\emph{)} de la \hbox{$(n-i)$}\nbd-catégorie $\Par_{\Stnu(K)}(\cellule{x})$ des $i$\nbd-cellules parallèles à~$\cellule{x}$, l'unique $0$\nbd-cellule $h_0$\nbd-négligeable étant $c_0$.
\end{thm}

\begin{rem}\label{remcontraction}
Si le complexe dirigé augmenté $K$ est décent (cf.~\ref{ADCdesc}), se donner une contraction (resp. une contraction duale) de $K$ revient à se donner un morphisme constant de complexes dirigés augmentées $f:K\to K$ (cf.~\ref{ADCcons}) et une homotopie de morphismes de complexes dirigés augmentés de $f$ vers $1_K$ (resp. de $1_K$ vers $f$) (cf.~\ref{ADChmt}). En effet, se donner un morphisme constant de complexes dirigés augmentées $f:K\to K$ revient à se donner un élément $c_0$ de $K_0$ satisfaisant aux conditions $(1)$ et $(2)$ du paragraphe~\ref{defcontraction}, et une homotopie de morphismes de complexes dirigés augmentés de $f$ vers $1_K$ (resp. de $1_K$ vers $f$) est la donnée d'une famille $(h_j)_{j\geq0}$ de morphismes de groupes abéliens $h_j:K_j\to K_{j+1}$ satisfaisant exactement, pour $j\geq0$, aux conditions $(\cnb{*}_j)$, $(\cnb{**}_j)$ et $(\cnb{*\kern -2pt*\kern -2pt*}_j)$ de~\ref{defcontraction}.
\end{rem}

\section{Le nerf d'un complexe dirigé augmenté}\label{nerfADC}

\begin{paragr}\label{enssimpl}
On note $\ord$ la catégorie des ensembles ordonnés et applications croissantes, considérée comme sous-catégorie pleine de la catégorie $\cat$ des petites catégories. On rappelle que la \emph{catégorie des simplexes} $\cats$ est la sous-catégorie pleine de $\ord$ formée des ensembles
$$\smp{n}=\{0,1,\dots,n\}\ ,\qquad n\geq0\ ,$$
ordonnés par l'ordre naturel. On note 
$$\face{i}{n}:\smp{n-1}\toto\smp{n}\ ,\qquad n>0\ ,\quad0\leq i\leq n\ ,$$
l'unique injection croissante dont l'image ne contient pas $i$ et
$$\dgn{i}{n}:\smp{n+1}\toto\smp{n}\ ,\qquad n\geq0\ ,\quad0\leq i\leq n\ ,$$
l'unique surjection croissante qui prend deux fois la valeur $i$.
La catégorie $\cats$ est engendrée par composition par ces morphismes soumis aux \emph{relations simpliciales}:
$$\begin{aligned}
&\face{j}{n+1}\face{i}{n}=\face{i}{n+1}\face{j-1}{n},\qquad n>0\ ,\quad0\leq i<j\leq n+1\ ,\cr
\noalign{\vskip 3pt}
&\dgn{j}{n}\dgn{i}{n+1}=\dgn{i}{n}\dgn{j+1}{n+1}\,\ ,\qquad n\geq0\ ,\quad0\leq i\leq j\leq n\ ,\cr
\noalign{\vskip 3pt}
&\dgn{j}{n}\face{i}{n+1}=\left\{\begin{aligned}
&\face{i}{n}\dgn{j-1}{n-1}\ ,\qquad\kern-7pt n>0\ ,\quad 0\leq i<j\leq n\ ,\cr
&1_{\smp{n}}\ ,\qquad\kern 7pt n\geq0\ ,\quad0\leq j\leq n\ ,\quad i=j\,,j+1\ ,\cr
&\face{i-1}{n}\dgn{j}{n-1}\ ,\qquad\kern-16pt n>0\ ,\quad1\leq j+1<i\leq n+1\ .
\end{aligned}\right.
\end{aligned}$$
La catégorie des \emph{ensembles simpliciaux} est la catégorie $\simpl$ des préfaisceaux sur $\cats$. On identifiera, par le plongement de Yoneda, $\cats$ à une sous-catégorie pleine de $\simpl$. 
\end{paragr}

\begin{paragr}
Soit $X$ un ensemble simplicial. Comme de coutume, on pose
$$
\begin{aligned}
&X_n=X(\smp{n})\ ,\qquad n\geq0\ ,\cr
&d^i_n=X(\face{i}{n})\ ,\qquad n>0\ ,\quad0\leq i\leq n\ ,\cr
&s^i_n=X(\dgn{i}{n})\ ,\qquad n\geq0\ ,\quad0\leq i\leq n\ ,
\end{aligned}$$
et les éléments de $X_n$ sont appelés les \emph{$n$\nbd-simplexes} de $X$. On dit qu'un $n$\nbd-simplexe $x$ de $X$ est \emph{dégénéré} s'il existe $m$, $0\leq m<n$, une application croissante $\varphi:\smp{n}\to\smp{m}$, et un $m$\nbd-simplexe $y$ tels que $x=X(\varphi)(y)$; on dit qu'il est \emph{non dégénéré} s'il n'est pas dégénéré. On note $\sdeg{X}{n}$ (resp. $\sndeg{X}{n}$) l'ensemble des $n$\nbd-simplexes dégénérés (resp. non dégénérés) de $X$. Tout $0$\nbd-simplexe de $X$ est non dégénéré, et pour $n>0$, un $n$\nbd-simplexe de $X$ est dégénéré si et seulement si il existe $i$, $0\leq i<n$, et $y\in X_{n-1}$ tel que $x=s^i_{n-1}(y)$.
\end{paragr}

\begin{paragr}\label{complchaines}
\`A tout ensemble simplicial $X$, on associe un complexe dirigé augmenté $\Ch{X}$, appelé \emph{complexe dirigé augmenté des chaînes de $X$}. Le complexe sous-jacent est le complexe des chaînes associé à l'ensemble simplicial $X$, autrement dit, pour $n\geq0$, $(\Ch{X})_n$ est le groupe abélien libre engendré par $X_n$, et pour $n>0$ la différentielle $d_n:(\Ch{X})_n\to(\Ch{X})_{n-1}$ est l'application $\mathbb{Z}$\nbd-linéaire définie pour $x\in X_n$ par 
$$d^{}_n(x)=\textstyle\sum\limits_{0\leq i\leq n}(-1)^id^i_n(x)\ .$$
L'augmentation $\augm:(\Ch{X})_0\to\mathbb{Z}$ est l'application $\mathbb{Z}$\nbd-linéaire définie pour $x\in X_0$ par $\augm(x)=1$. Le sous-monoïde $(\Ch{X})^*_n$ de $(\Ch{X})_n$ est le sous-monoïde engendré par la base~$X_n$ de $(\Ch{X})_n$. On définit ainsi un foncteur $\Ch{}:\simpl\to\ADC$.
\end{paragr}

\begin{paragr}\label{souscompldeg}
Pour $X$ un ensemble simplicial et $n\geq0$, on note $(\Chdeg{X})_n$ le sous-groupe abélien libre de $(\Ch{X})_n$ engendré par l'ensemble $\sdeg{X}{n}$ des $n$\nbd-simplexes dégénérés de~$X$. Les relations simpliciales impliquent facilement qu'on a, pour $n>0$, l'inclusion $d_n((\Chdeg{X})_n)\subset(\Chdeg{X})_{n-1}$, d'où un sous-complexe $\Chdeg{X}$ de $\Ch{X}$. En munissant ce complexe de la structure de complexe dirigé augmenté induite par celle de $\Ch{X}$ (cf.~\ref{induitADC}), on obtient un complexe dirigé augmenté qu'on notera encore $\Chdeg{X}$, et qu'on appellera le \emph{complexe dirigé augmenté des chaînes dégénérées} de $X$. Comme un morphisme d'ensembles simpliciaux envoie les simplexes dégénérés sur des simplexes dégénérés, on en déduit un foncteur $\Chdeg{}:\simpl\to\ADC$.
\end{paragr}

\begin{paragr}\label{complnormchaines}
Comme pour tout ensemble simplicial $X$, l'ensemble des $0$\nbd-simplexes dégénérés est vide, on a $(\Chdeg{X})_0=0$, et en particulier l'augmentation de $\Chdeg{X}$ est nulle. On en déduit un complexe dirigé augmenté quotient de $\Ch{X}$ par $\Chdeg{X}$ (cf.~\ref{quotientADC}), qu'on notera $\Chnorm{X}$, et qu'on appellera le \emph{complexe dirigé augmenté normalisé des chaînes} de~$X$. Pour tout $n\geq0$, le groupe abélien $(\Chnorm{X})_n=(\Ch{X})_n/(\Chdeg{X})_n$ s'identifie au groupe abélien libre engendré par l'ensemble $\sndeg{X}{n}$ des $n$\nbd-simplexes non dégénérés de $X$, et le sous-monoïde $(\Chnorm{X})^*_n$ de $(\Chnorm{X})_n$ au sous-monoïde engendré par cet ensemble. Dans cette identification, pour $n>0$, la différentielle est donnée par la même formule que celle de $\Ch{X}$ (cf.~\ref{complchaines}), appliquée aux seuls simplexes non dégénérés et en omettant au membre de droite les termes correspondant à des $(n-1)$\nbd-simplexes dégénérés. Pour $n=0$, on a $(\Chnorm{X})_0=(\Ch{X})_0$, et l'augmentation de $\Chnorm{X}$ coïncide avec celle de $\Ch{X}$. On définit ainsi un foncteur $\Chnorm{}:\simpl\to\ADC$.
\end{paragr}

\begin{rem}
Pour tout ensemble simplicial $X$, les complexes dirigés augmentés $\Ch{X}$, $\Chdeg{X}$ et $\Chnorm{X}$ admettent une base (cf.~\ref{baseADC}), et sont décents (cf.~\ref{ADCdesc}).
\end{rem}

\begin{exemple}\label{exADCstandard}
Le complexe dirigé augmenté $\Chnorm{\smp{n}}$, $n\geq0$, se décrit comme suit. Pour $p\geq0$, un $p$\nbd-simplexe non dégénéré de l'ensemble simplicial représentable $\smp{n}$ est une application strictement croissante $\smp{p}\to\smp{n}$. Ainsi, pour $p\geq0$, $(\Chnorm{\smp{n}})_p$ (resp.~$(\Chnorm{\smp{n}})^*_p$) s'identifie au groupe (resp. au monoïde) commutatif libre engendré par la famille des \hbox{$(p+1)$}\nbd-uplets 
$$(i_0,i_1,\dots,i_p)\,,\quad 0\leq i_0<i_1<\cdots<i_p\leq n\ .$$
La différentielle est définie par 
$$d(i_0,i_1,\dots, i_p)=\textstyle\sum\limits_{0\leq k\leq p}(-1)^k(i_0,\dots, \widehat{i}_k,\dots, i_p)\,,\quad p>0\ ,$$
où $(i_0,\dots, \widehat{i}_k,\dots, i_p)=(i_0,\dots, i_{k-1},i_{k+1},\dots, i_p)$, et l'augmentation par $\augm(i_0)=1$. On remarque que pour $p>n$, on a $\ADCO{n}_p=0$. Le complexe dirigé augmenté $\Chnorm{\smp{n}}$ est à base unitaire fortement sans boucles~\cite[exemple~3.8]{Steiner}. Le théorème~\ref{csunitsansboucles} qu'on démontrera dans la section~\ref{complsimpl} constitue une généralisation de ce résultat.
\end{exemple}

\begin{prop}\label{complcommlimind}
Les foncteurs $\Ch{},\,\Chdeg{},\,\Chnorm{}:\simpl\to\ADC$ commutent aux petites limites inductives.
\end{prop}

\begin{proof}
Comme les foncteurs \og groupe abélien libre\fg{} et \og monoïde commutatif libre\fg{} commutent aux limites inductives, l'assertion concernant $\Ch{}$ résulte de la description des limites inductives dans $\ADC$ (cf.~\ref{limindADC}). L'assertion concernant le foncteur $\Chdeg$ se démontre de la même façon en remarquant que pour tout $n\geq0$, le foncteur associant à un ensemble simplicial $X$ l'ensemble $\sdeg{X}{n}$ de ses $n$\nbd-simplexes dégénérés commute aux limites inductives. En effet, il résulte~facilement du lemme d'Eilenberg-Zilber~\cite[8.3]{EZ} ou~\cite[chapitre II, 3.1]{GZ} que si l'on note $I$ la sous-catégorie pleine de la catégorie $\mc{\smp{n}}{\cats}$ des objets de $\cats$ sous $\smp{n}$ formée des couples $(\smp{i},\pi:\smp{n}\to\smp{i})$ tels que $\pi$ soit un épimorphisme de $\cats$ qui n'est pas une identité, alors on a une bijection fonctorielle
$$\sdeg{X}{n}\ \simeq\limind_{(\smp{i},\pi)\in I}\kern -5pt X_i\ ,$$
et l'assertion résulte de la commutativité des limites inductives entre elles. Enfin, la commutativité du foncteur $\Chnorm{}$ aux limites inductives résulte de celle de $\Ch{}$ et $\Chdeg{}$, et de la propriété universelle des quotients (cf.~\ref{quotientADC}).
\end{proof}

\begin{paragr}\label{defnerfADC}
On définit un foncteur \emph{nerf} $\NerfADC:\ADC\to\simpl$ par
$$K\longmapsto(\smp{n}\mapsto\Hom_{\ADC}(\Chnorm{}(\smp{n}),K))\ .$$
Le foncteur $\Chnorm{}$  commutant aux limites inductives en vertu de la proposition précédente, on obtient un couple de foncteurs adjoints
$$\Chnorm{}:\simpl\to\ADC\ ,\qquad \NerfADC:\ADC\to\simpl\ .$$
\end{paragr}

\begin{paragr}\label{eqfADC}
On note $\Wsimpl$ la classe des équivalences faibles simpliciales, flèches de $\simpl$ dont la réalisation topologique est une équivalence d'homotopie. On rappelle que la catégorie localisée $\Wsimpl^{-1}\simpl$ est équivalente à la catégorie homotopique $\Hot$ des CW\nbd-complexes~\cite{Mil}. On dit qu'un morphisme $f$ de $\ADC$ est une \emph{équivalence faible} si $\NerfADC(f)$ est une équivalence faible simpliciale. On note $\WADC=\NerfADC^{-1}(\Wsimpl)$ la classe des équivalences faibles de $\ADC$. Ainsi le foncteur nerf $\NerfADC$ induit un foncteur entre les catégories localisées
$$\loc{\NerfADC}:\WADC^{-1}\ADC\toto\Wsimpl^{-1}\simpl\simeq\Hot\ .$$
\end{paragr}

\begin{conj}\label{thinair}
Le foncteur $\loc{\NerfADC}$ est une équivalence de catégories.
\end{conj}

\section{Quelques propriétés des orientaux de Street}

\begin{paragr}\label{defor}
Soit $n\geq0$ un entier. Le \emph{$n$\nbd-ième oriental de Street} $\Or{n}$ est par définition la $\infty$\nbd-catégorie $\Stnu\ADCO{n}$ (cf.~\ref{defStlambda} et~\ref{exADCstandard}). Comme pour tout $p>n$, on a $\ADCO{n}_p=0$, l'oriental $\Or{n}$ est une $n$\nbd-catégorie. On vérifie facilement qu'on a des isomorphismes
$$\Or{0}\simeq\smp{0}\ ,\qquad\Or{1}\simeq\smp{1}\ ,$$
et que $\Or{2}$ est la $2$\nbd-catégorie
$$
\xymatrixrowsep{.31pc}
\xymatrixcolsep{2.3pc}
\xymatrix{
&&1\ar[rddddd]^{<(1,2)>}
\\
\\
&&
\\
\\
&&\ar@{=>}[uu]_{\alpha}
\\
&0\ar[rr]_{<(0,2)>}\ar[ruuuuu]^{<(0,1)>}
&&2
&\kern-20pt,\kern20pt
}$$
où
$$<(0,1)>\,=\begin{pmatrix}
(0)
&(0,1)
\cr
(1)
&(0,1)
\end{pmatrix},\ 
<(0,2)>=\begin{pmatrix}
(0)
&(0,2)
\cr
(2)
&(0,2)
\end{pmatrix},\ 
<(1,2)>=\begin{pmatrix}
(1)
&(1,2)
\cr
(2)
&(1,2)
\end{pmatrix} 
$$
et
$$
\alpha=<(0,1,2)>\,=\begin{pmatrix}
(0)
&(0,2)
&(0,1,2)
\cr
(2)
&(0,1)+(1,2)
&(0,1,2)
\end{pmatrix}\ .
$$
Pour une description explicite des $n$\nbd-catégories $\Or{n}$, pour $n\leq6$, voir~\cite{S1}.
\end{paragr}

\begin{paragr}
On note $u$ l'endomorphisme constant du complexe dirigé augmenté décent $\Chnorm{\smp{n}}$ défini par 
$$u_0:\ADCO{n}_0\toto\ADCO{n}_0\,,\qquad (i_0)\longmapsto(0)\,,\quad0\leq i_0\leq n\ .$$
Pour tout $p\geq0$, on définit un morphisme de groupes abéliens $$h_p:\ADCO{n}_p\to \ADCO{n}_{p+1}$$ par
$$h_p(i_0,i_1,\dots,i_p)=\left\{\begin{matrix}
(0,i_0,i_1,\dots,i_p)\,,
&&i_0>0\,,\cr
0\,,
&&i_0=0\,.
\end{matrix}\right.$$
\end{paragr}

\begin{prop}\label{hmtpsimpl}
La famille $(h_p)_{p\in\mathbb{N}}$ est une homotopie de carré nul de morphismes de complexes dirigés augmentés, de source l'endomorphisme constant $u$ de~$\Chnorm{\smp{n}}$ et de but l'endomorphisme identité, définissant ainsi une contraction du complexe dirigé augmenté $\Chnorm{\smp{n}}$ \emph{(cf.~remarque~\ref{remcontraction})}.
\end{prop}

\begin{proof}
Il est immédiat que pour tout $p\geq0$, on a $h_p(\ADCO{n}^*_p)\subset\ADCO{n}^*_{p+1}$ et $h_{p+1}h_p=0$. Il suffit donc de montrer:
\begin{itemize}
\item[a)] pour tout $x\in\ADCO{n}_0$, on a $d_1h_0(x)=x-\augm(x).(0)$;
\item[b)] pour tout $p$, $0<p\leq n$ et tout $x\in\ADCO{n}_p$, on a $(d_{p+1}h_p+h_{p-1}d_p)(x)=x$.
\end{itemize}
Pour prouver (\emph{a}) et (\emph{b}), il suffit de vérifier ces égalités pour $x$ un élément de la base de $\Chnorm{\smp{n}}$, ce qu'on va faire en omettant, conformément à l'usage, les indices de $h$ et de~$d$. Un élément $(i_0,\dots,i_p)$ de la base de $\Chnorm{\smp{n}}$ sera noté plus simplement $(i_0\dots i_p)$, ou en ne gardant que les virgules nécessaires pour éviter toute ambiguïté. 
\smallbreak

Pour $i_0>0$, on a les égalités $dh(i_0)=d(0,i_0)=(i_0)-(0)$ et pour $i_0=0$, les égalités $dh(i_0)=0=(0)-(0)$, ce qui prouve (\emph{a}).
\smallbreak

Pour $p>0$ et $0\leq i_0<i_1<\cdots<i_p\leq n$, si $i_0>0$, on a 
$$\begin{aligned}
(dh&+hd)(i_0\dots i_p)=d(0,i_0\dots i_p)+h\bigl(\textstyle\sum\limits_{0\leq k\leq p}(-1)^k(i_0\dots\widehat{i}_k\dots i_p)\bigr)\cr
&=(i_0\dots i_p)+\textstyle\sum\limits_{0\leq k\leq p}(-1)^{k+1}(0,i_0\dots\widehat{i}_k\dots i_p)+\sum\limits_{0\leq k\leq p}(-1)^k(0,i_0\dots\widehat{i}_k\dots i_p)\cr
&=(i_0\dots i_p)\ ,
\end{aligned}$$
et si $i_0=0$, on a 
$$\begin{aligned}
(dh+hd)(i_0\dots i_p)&=h\bigl(\textstyle\sum\limits_{0\leq k\leq p}(-1)^k(i_0\dots\widehat{i}_k\dots i_p)\bigr)\cr
&=(0,i_1\dots i_p)=(i_0\dots i_p)\ ,
\end{aligned}$$
ce qui prouve (\emph{b}).
\end{proof}

\begin{paragr}
On note $v$ l'endomorphisme constant du complexe dirigé augmenté décent $\Chnorm{\smp{n}}$ défini par 
$$v_0:\ADCO{n}_0\toto\ADCO{n}_0\,,\qquad (i_0)\longmapsto(n)\,,\quad0\leq i_0\leq n\ .$$
Pour tout $p\geq0$, on définit un morphisme de groupes abéliens $$h'_p:\ADCO{n}_p\to \ADCO{n}_{p+1}$$ par
$$
\begin{aligned}
&h'_0(i_0)=\textstyle\sum\limits_{i_0< k\leq n}(k-1,k)\ ,\cr
&h'_p(i_0i_1\dots i_p)=\textstyle\sum\limits_{i_0<k<i_1}(k-1,k,i_1\dots i_p)\ ,\quad p>0\ .
\end{aligned}$$
\end{paragr}

\begin{prop}\label{hmtpcompl}
La famille $(h'_p)_{p\in\mathbb{N}}$ définit une homotopie de carré nul de morphismes de complexes dirigés augmentés, de source l'endomorphisme identité de $\Chnorm{\smp{n}}$ et de but l'endomorphisme constant $v$, définissant ainsi une contraction duale du complexe dirigé augmenté $\Chnorm{\smp{n}}$ \emph{(cf.~remarque~\ref{remcontraction})}.
\end{prop}

\begin{proof}
Il est immédiat que pour tout $p\geq0$, on a $h'_p(\ADCO{n}^*_p)\subset\ADCO{n}^*_{p+1}$ et $h'_{p+1}h'_p=0$. Il suffit donc de montrer:
\begin{itemize}
\item[a)] pour tout $x\in\ADCO{n}_0$, on a $d_1h'_0(x)=\augm(x).(n)-x$;
\item[b)] pour tout $p$, $0<p\leq n$ et tout $x\in\ADCO{n}_p$, on a $(d_{p+1}h'_p+h'_{p-1}d_p)(x)=-x$.
\end{itemize}
Pour prouver (\emph{a}) et (\emph{b}), il suffit de vérifier ces égalités pour $x$ un élément de la base de $\Chnorm{\smp{n}}$, ce qu'on va faire en omettant, comme dans la preuve précédente, les indices de $h'$ et de~$d$. Un élément $(i_0,\dots,i_p)$ de la base de $\Chnorm{\smp{n}}$ sera également noté $(i_0\dots i_p)$, ou en ne gardant que les virgules nécessaires pour éviter toute ambiguïté.
\smallbreak
Pour $0\leq i_0\leq n$, on a
$$
\begin{aligned}
dh'(i_0)&=d\bigl(\textstyle\sum\limits_{i_0< k\leq n}(k-1,k)\bigr)\cr
&=\textstyle\sum\limits_{i_0< k\leq n}(k)-\textstyle\sum\limits_{i_0< k\leq n}(k-1)=(n)-(i_0)\ ,
\end{aligned}$$
ce qui prouve (\emph{a}).
\smallbreak

Pour $0\leq i_0<i_1\leq n$, on a
$$
\begin{aligned}
(dh'+h'd)(i_0i_1)&=d\bigl(\textstyle\sum\limits_{i_0<k<i_1}(k-1,k,i_1)\bigr)+h'(i_1-i_0)\cr
\noalign{\vskip 5pt}
&=\textstyle\sum\limits_{i_0<k<i_1}(k,i_1)-\textstyle\sum\limits_{i_0<k<i_1}(k-1,i_1)+\textstyle\sum\limits_{i_0<k<i_1}(k-1,k)\cr
&{}\kern 72pt+\textstyle\sum\limits_{i_1< k\leq n}(k-1,k)-\textstyle\sum\limits_{i_0< k\leq n}(k-1,k)\cr
\noalign{\vskip 5pt}
&=(i_1-1,i_1)-(i_0,i_1)-(i_1-1,i_1)=-(i_0,i_1)\ ,
\end{aligned}$$
ce qui prouve (\emph{b}), dans le cas $p=1$.
\smallbreak

Pour $p>1$ et $0\leq i_0<i_1<\cdots<i_p\leq n$, on a
$$
\begin{aligned}
d&h'(i_0i_1\dots i_p)=d\bigl(\textstyle\sum\limits_{i_0<k<i_1}(k-1,k,i_1\dots i_p)\bigr)\cr
&=\kern -8pt\textstyle\sum\limits_{i_0<k<i_1}\kern -6pt(k,i_1\dots i_p)-\kern -9pt\textstyle\sum\limits_{i_0<k<i_1}\kern -6pt(k-1,i_1\dots i_p)+\kern -8pt\textstyle\sum\limits_{i_0<k<i_1}\sum\limits_{1\leq l\leq p}\kern -4pt(-1)^{l+1}(k-1,k,i_1\dots\widehat{i}_l\dots i_p)\cr
&=(i_1-1,i_1\dots i_p)-(i_0,i_1\dots i_p)+\textstyle\sum\limits_{1\leq l\leq p}(-1)^{l+1}\kern -4pt\sum\limits_{i_0<k<i_1}(k-1,k,i_1\dots\widehat{i}_l\dots i_p)
\end{aligned}$$
et
$$
\begin{aligned}
h'd(i_0i_1\dots i_p)&=h'\bigl(\textstyle\sum\limits_{0\leq l\leq p}(-1)^l(i_0i_1\dots\widehat{i}_l\dots i_p)\bigr)\cr
\noalign{\vskip 5pt}
&=\textstyle\sum\limits_{i_1<k<i_2}(k-1,k,i_2\dots i_p)-\textstyle\sum\limits_{i_0<k<i_2}(k-1,k,i_2\dots i_p)\cr
&{}\kern 72pt+\textstyle\sum\limits_{2\leq l\leq p}(-1)^{l}\kern -4pt\sum\limits_{i_0<k<i_1}(k-1,k,i_1\dots\widehat{i}_l\dots i_p)\cr
\noalign{\vskip 5pt}
&=-\kern -4pt\textstyle\sum\limits_{i_0<k\leq i_1}(k-1,k,i_2\dots i_p)+\kern -4pt\textstyle\sum\limits_{2\leq l\leq p}(-1)^{l}\kern -4pt\sum\limits_{i_0<k<i_1}(k-1,k,i_1\dots\widehat{i}_l\dots i_p)\cr
\noalign{\vskip 5pt}
&=-(i_1-1,i_1,i_2\dots i_p)+\textstyle\sum\limits_{1\leq l\leq p}(-1)^{l}\kern -4pt\sum\limits_{i_0<k<i_1}(k-1,k,i_1\dots\widehat{i}_l\dots i_p)
\end{aligned}$$
et par suite
$$(dh'+h'd)(i_0i_1\dots i_p)=-(i_0i_1\dots i_p)\ ,$$
ce qui prouve (\emph{b}), dans le cas $p>1$.
\end{proof}
\smallbreak

\begin{thm}\label{bigchief}
L'oriental de Street $\Or{n}$ satisfait aux propriétés suivantes :
\begin{itemize}
\item[\emph{a)}] L'objet $(0)$ de $\Or{n}$ est un objet quasi-initial et $(n)$ un objet quasi-final.
\item[\emph{b)}] Si $n\geq1$, les $1$\nbd-flèches
$$\begin{pmatrix}
(0)
&(0,n)
\cr
(n)
&(0,n)
\end{pmatrix}
\qquad\hbox{et}\qquad
\begin{pmatrix}
(0)
&(0,1)+(1,2)+\cdots+(n-1,n)
\cr
(n)
&(0,1)+(1,2)+\cdots+(n-1,n)
\end{pmatrix}
$$
de $\Or{n}$ sont respectivement un objet quasi-initial et un objet quasi-final de la $(n-1)$\nbd-catégorie $\sHom_{\Or{n}}((0),(n))$.
\item[\emph{c)}] Si $n\geq2$, les $2$\nbd-flèches
$$\begin{pmatrix}
(0)
&(0,n)
&(0,1,2)+(0,2,3)+\cdots+(0,n-1,n)
\cr
(n)
&(0,1)+(1,2)+\cdots+(n-1,n)
&(0,1,2)+(0,2,3)+\cdots+(0,n-1,n)
\end{pmatrix}
$$
et
$$\begin{pmatrix}
(0)
&(0,n)
&(0,1,n)+(1,2,n)+\cdots+(n-2,n-1,n)
\cr
(n)
&(0,1)+(1,2)+\cdots+(n-1,n)
&(0,1,n)+(1,2,n)+\cdots+(n-2,n-1,n)
\end{pmatrix}
$$
de $\Or{n}$ sont respectivement un objet quasi-initial et un objet quasi-final de la $(n-2)$\nbd-catégorie 
$$\sHom_{\Or{n}}\Biggl(
\begin{pmatrix}
(0)
&(0,n)
\cr
(n)
&(0,n)
\end{pmatrix}
\,,\,
\begin{pmatrix}
(0)
&(0,1)+(1,2)+\cdots+(n-1,n)
\cr
(n)
&(0,1)+(1,2)+\cdots+(n-1,n)
\end{pmatrix}
\Biggr)\ .$$
\end{itemize}
\end{thm}

\begin{proof}
En vertu de la remarque~\ref{remcontraction}, le théorème est conséquence immédiate du théorème~\ref{thmcontraction} et des propositions~\ref{hmtpsimpl} et \ref{hmtpcompl}, vu que
$$h(0)=0\,,\quad h(0,n)=0\,,\quad h\bigl((0,1,2)+(0,2,3)+\dots+(0,n-1,n)\bigr)=0$$
et
$$\begin{aligned}
&h'(n)=0\,,\quad h'\bigl((0,1)+(1,2)+\cdots+(n-1,n)\bigr)=0\,,\cr
&h'\bigl((0,1,n)+(1,2,n)+\cdots+(n-2,n-1,n)\bigr)=0\,,
\end{aligned}$$
où $h$ et $h'$ désignent les homotopies des propositions~\ref{hmtpsimpl} et \ref{hmtpcompl}.
\end{proof}

\begin{rem}\label{remthin}
Si $0\leq i<j\leq n$ et si l'on pose $m=j-i$, l'application $k\mapsto k+i$ induit une inclusion de $\Or{m}\to\Or{n}$ telle que pour tout couple $x,y$ de $p$\nbd-cellules parallèles de $\Or{m}$, $p\geq0$, l'inclusion $\sHom_{\Or{m}}(x,y)\to\sHom_{\Or{n}}(x,y)$ soit un isomorphisme. En appliquant le théorème précédent à $\Or{m}$, on en déduit que la $(n-1)$\nbd-catégorie $\sHom_{\Or{n}}((i),(j))$ admet à la fois un objet quasi-initial et un objet quasi-final, $(i)$ étant un objet quasi-initial et $(j)$ un objet quasi-final. 
\end{rem}

\begin{conj}
Pour tout couple parallèle de $i$\nbd-cellules $x,y$ de $\Or{n}$, $0\leq i<n$, si la \hbox{$(n-i-1)$}\nbd-catégorie $\sHom_{\Or{n}}(x,y)$ est non vide, alors elle admet un objet quasi-initial et un objet quasi-final.
\end{conj}

\begin{rem}
Pour $i=1$ cette conjecture résulte du théorème~\ref{thconjor1}.
\end{rem}

\section{Le nerf de Street et les équivalences faibles \pdfinfty-catégoriques}\label{nerfStreet}

\begin{paragr}\label{defnerfStreet}
La restriction à $\cats$ du foncteur composé
$$\xymatrix{
\simpl\ar[r]^{\Chnorm{}}
&\ADC\ar[r]^{\Stnu}
&\nCat{\infty}
}
\ ,\quad \smp{n}\longmapsto\Or{n}$$
définit un objet cosimplicial de $\nCat{\infty}$. Par le procédé de Kan, on en déduit un couple de foncteurs adjoints
$$\cStr{\infty}:\simpl\toto\nCat{\infty}\ ,\quad\NStr{\infty}:\nCat{\infty}\toto\simpl\ ,$$
$\cStreet$ étant l'unique foncteur, à isomorphisme unique près, commutant aux limites inductives et dont la restriction à $\cats$ coïncide avec celle de $\Stnu\Chnorm{}$, et $\NStreet$ étant défini par
$$C\,\longmapsto\,\bigl(\smp{n}\mapsto\Hom_{\ooCat}(\Or{n},C)\bigr)\ .$$
\end{paragr}

\begin{paragr}\label{objnerfadj}
Vu la description des orientaux $\Or{n}$ pour $0\leq n\leq2$ (cf.~\ref{defor}), pour toute $\infty$\nbd-catégorie $C$, les $0$\nbd-simplexes de l'ensemble simplicial $\NStr{\infty}(C)$ s'identifient aux objets de $C$, ses $1$\nbd-simplexes aux $1$\nbd-flèches de $C$, et les $2$\nbd-simplexes aux diagrammes dans $C$ de la forme
$$
\xymatrixrowsep{.31pc}
\xymatrixcolsep{2.3pc}
\xymatrix{
&&x^{}_1\ar[rddddd]^{x^{}_{21}}
\\
\\
&&
\\
\\
&&\ar@{=>}[uu]_{x^{}_{210}}
\\
&x^{}_0\ar[rr]_{x^{}_{20}}\ar[ruuuuu]^{x^{}_{10}}
&&x^{}_2
&\kern-20pt.\kern20pt
}$$
D'autre part, vu que les foncteurs $\cStr{\infty}$,
$$\ob:\ooCat\toto\ens\ ,\qquad C\longmapsto\ob\,C$$
et
$$(?)_{\,0}:\simpl\toto\ens\ ,\qquad X\longmapsto X_0$$
commutent aux limites inductives, pour tout ensemble simplicial $X$, on a des bijections fonctorielles
$$\begin{aligned}
\ob\,\cStr{\infty}\,X&\simeq\ob\,\cStr{\infty}\!\limind_{\smp{m}\rightarrow X}\smp{m}\simeq\limind_{\smp{m}\rightarrow X}\ob\,\cStr{\infty}\,\smp{m}\cr
\noalign{\vskip 3pt}
&=\limind_{\smp{m}\rightarrow X}\ob\,\Or{n}\simeq\limind_{\smp{m}\rightarrow X}(\smp{m})_{\,0}\simeq\Bigl(\limind_{\smp{m}\rightarrow X}\smp{m}\Bigr)_0\simeq X_0\ ,
\end{aligned}$$
identifiant les objets de $\cStr{\infty}(X)$ aux $0$\nbd-simplexes de $X$. De plus, la fonctorialité du morphisme d'adjonction $\cStr{\infty}\NStr{\infty}\to1_{\ooCat}$, appliquée au $\infty$\nbd-foncteur $\Or{0}\simeq\smp{0}\to C$, pour $C$ une $\infty$\nbd-catégorie, implique aussitôt que le $\infty$\nbd-foncteur $\cStr{\infty}\NStr{\infty}C\to C$ induit l'identité sur les objets.
\end{paragr}

\begin{paragr}\label{eqfooCat}
On dit qu'un morphisme $F$ de $\ooCat$ est une \emph{équivalence faible} si $\NStreet(F)$ est une équivalence faible simpliciale. On note $\WooCat=\NStreet^{-1}(\Wsimpl)$ la classe des équivalences faibles de $\ooCat$. Ainsi le foncteur nerf $\NStreet$ induit un foncteur entre les catégories localisées
$$\loc{\NStreet}:\WooCat^{-1}\ooCat\toto\Wsimpl^{-1}\simpl\simeq\Hot\ .$$
On dit qu'une $\infty$\nbd-catégorie $C$ est \emph{asphérique} (ou \emph{faiblement contractile}) si le morphisme de $C$ vers l'objet final $e$ de $\ooCat$ ($\infty$\nbd-catégorie ayant un seul objet et les identités itérées de cet objet comme seules cellules) est une équivalence faible, autrement dit si l'ensemble simplicial $\NStreet(C)$ est faiblement contractile.
\end{paragr}

\begin{conj}\label{conjfond}
Le foncteur $\loc{\NStreet}$ est une équivalence de catégories.
\end{conj}

\begin{prop}\label{commnerf}
Les triangles de foncteurs
$$
\xymatrixcolsep{1pc}
\xymatrix{
\ADC\ar[rr]^{\Stnu}\ar[rd]_{\NerfADC}
&&\ooCat\ar[ld]^{\NStreet}
\\
&\simpl
}
\kern 60pt
\xymatrix{
&\simpl\ar[rd]^{\Chnorm{}}\ar[ld]_{\cStreet}
\\
\ooCat\ar[rr]_{\Stlambda}
&&\ADC
}
$$
sont commutatifs (à isomorphisme canonique près).
\end{prop}

\begin{proof}
Par adjonction, il suffit de montrer la commutativité de celui de gauche. Or, pour $n\geq0$ et $K$ un complexe dirigé augmenté, on a des bijections fonctorielles
$$\begin{aligned}
(\NStreet\Stnu K)_n=\Hom_{\ooCat}(\Stnu\Chnorm{\smp{n}},\Stnu K)&\simeq\Hom_{\ADC}(\Stlambda\Stnu\Chnorm{\smp{n}},K)\cr
&\simeq\Hom_{\ADC}(\Chnorm{\smp{n}},K)=(\NerfADC K)_n
\end{aligned}$$
la première par adjonction (cf. proposition~\ref{adjSteiner}) et la seconde en vertu du théorème~\ref{isoadjunitsansboucles}, de la proposition~\ref{propfortsansboucles} et de l'exemple~\ref{exADCstandard}, ce qui prouve l'assertion.
\end{proof}

\begin{cor}
On a l'égalité $\WADC=\Stnu^{-1}(\WooCat)$ et en particulier le foncteur~$\Stnu$ induit un foncteur $\loc{\Stnu}:\WADC^{-1}\ADC\to\WooCat^{-1}\ooCat$, et le triangle de foncteurs
$$
\xymatrixcolsep{1pc}
\xymatrix{
\WADC^{-1}\ADC\ar[rr]^{\loc{\Stnu}}\ar[rd]_{\loc{\NerfADC}}
&&\WooCat^{-1}\ooCat\ar[ld]^{\loc{\NStreet}}
\\
&\Hot
}$$
est commutatif.
\end{cor}

\begin{proof}
Le corollaire est conséquence immédiate de la proposition~\ref{commnerf}.
\end{proof}

\begin{rem}
En vertu du corollaire précédent, les conjectures~\ref{thinair} et~\ref{conjfond} impliquent que le foncteur $\loc{\Stnu}$ est une équivalence de catégories.
\end{rem}

\begin{thm}\label{consthA}
Soit $C$ une $\infty$\nbd-catégorie. Si $C$ admet un objet $x$ tel que pour tout objet $y$ de $C$ la $\infty$\nbd-catégorie $\sHom_C(y,x)$ \emph{(resp.}~$\sHom_C(x,y)$\emph{)} soit asphérique, alors $C$ est asphérique.
\end{thm}

\begin{proof}
Ce résultat sera obtenu dans~\cite{DG3} comme corollaire d'un théorème~A $\infty$\nbd-catégorique.
\end{proof}

\begin{cor}\label{pseudoasph}
Soient $n\geq0$ un entier, et $C$ une $n$\nbd-catégorie. Si $C$ admet un objet quasi-final \emph{(resp.}~quasi-initial\emph{)}, alors $C$ est asphérique.
\end{cor}

\begin{proof}
Si $n=0$, $C$ considérée comme $\infty$-catégorie est un objet final de $\ooCat$ et par suite, l'ensemble simplicial $\NStreet(C)$ est isomorphe à $\smp{0}$, ce qui prouve l'assertion dans ce cas. Si $n>0$, par hypothèse, il existe un objet $x$ de $C$ tel que pour tout objet $y$ de $C$ la $(n-1)$\nbd-catégorie $\sHom_C(y,x)$ (resp. $\sHom_C(x,y)$) admette un objet quasi-final (resp. quasi-initial). Par l'hypothèse de récurrence la \hbox{$(n-1)$}\nbd-catégorie $\sHom_C(y,x)$ (resp. $\sHom_C(x,y)$) est alors asphérique. L'assertion résulte donc du théorème précédent.
\end{proof}

\begin{thm}\label{eqfDwyerKaneqf}
Soit $F:C\to D$ un morphisme de $\ooCat$. Si $F$ induit une bijection des ensembles des objets et si pour tous objets $x,y$ de $C$ le morphisme $\sHom_C(x,y)\to\sHom_D(F(x),F(y))$ induit par $F$~est une équivalence faible, alors $F$~est une équivalence faible.
\end{thm}

\begin{proof}
Ce résultat sera obtenu dans~\cite{DG2} comme conséquence de la comparaison du nerf de Street avec une variante bisimpliciale de ce nerf.
\end{proof}

\begin{paragr}\label{nnerf}
Soit $n\geq0$. On note $\NStr{n}$ la restriction du nerf de Street $\NStreet$ à $\nCat{n}$, composé de l'inclusion $\incl{n}:\nCat{n}\to\ooCat$ et de $\NStreet$. Le foncteur $\NStr{n}$ admet comme adjoint à gauche le composé $\ti{n}\cStreet$ de l'adjoint à gauche de $\NStreet$ et du foncteur de troncation intelligente (cf.~\ref{tronq}). Explicitement, si $C$ est une $n$\nbd-catégorie et $p\geq0$, on a
$$(\NStr{n}C)_p=\Hom_{\nCat{n}}(\ti{n}\Or{p},C)\ .$$
On vérifie facilement que le $1$\nbd-tronqué $\ti{1}\Or{p}$ n'est autre que l'ensemble ordonné $\smp{p}$, et par suite le foncteur $\NStr{1}$ coïncide avec le foncteur nerf usuel $\cat\to\simpl$. 
\end{paragr}

\begin{paragr}\label{eqfnCat}
Une \emph{équivalence faible} de $\nCat{n}$ est une flèche $F$ de $\nCat{n}$ telle que $\NStr{n}(F)$ soit une équivalence faible simpliciale. On note
$$\WnCat{n}=\NStr{n}^{-1}(\Wsimpl)=i_n^{-1}(\WooCat)=\fl{}\nCat{n}\cap\WooCat$$
la classe des équivalences faibles de $\nCat{n}$.
\end{paragr}

\section{Le complexe dirigé augmenté associé à un complexe simplicial}\label{complsimpl}

\begin{paragr}\label{defcs}
Pour un ensemble ordonné $E$, on note $\xi E$ l'ensemble de ses parties finies non vides totalement ordonnées. On rappelle qu'un \emph{complexe simplicial}
est un couple~$(E,\varPhi)$, où $E$ est un ensemble ordonné, et où $\varPhi$ est un sous-ensemble de $\xi E$ satisfaisant aux deux conditions suivantes:
\begin{itemize}
\item[a)] pour tout élément $x$ de $E$, le singleton $\{x\}$ appartient à $\varPhi$;
\item[b)] pour tous $S \in \varPhi$, $S' \subset S$, si $S' \neq \varnothing$, alors $S' \in \varPhi$.
\end{itemize}
Un morphisme de complexes simpliciaux
$f:(E_0,\varPhi_0)\to(E_1,\varPhi_1)$ est une application croissante $f:E_0\to E_1$ telle que pour tout $S\in\varPhi_0$, on a $f(S)\in\varPhi_1$.  On note $\cs$ la catégorie des complexes simpliciaux.
\end{paragr}

\begin{paragr}\label{limindprojcs}
La catégorie $\cs$ est complète et cocomplète. Si $(E_i,\varPhi_i)$ est un système projectif de complexes simpliciaux indexé par une petite catégorie $I$, sa limite projective est le complexe simplicial $(E,\varPhi)$, où $E=\limproj_IE_i$ est la limite projective dans la catégorie des ensembles ordonnés, et
$$\varPhi=\{S\in\xi E\,|\,\hbox{pour tout $i\in \ob I$\,, $\pi_i(S)\in\varPhi_i$}\}\ ,$$
où $\pi_i:E\to E_i$ désigne le morphisme canonique. Si $(E_i,\varPhi_i)$ est un système inductif de complexes simpliciaux indexé par une petite catégorie $I$, sa limite inductive est le complexe simplicial $(E,\varPhi)$, où $E=\limind_IE_i$ est la limite inductive dans la catégorie des ensembles ordonnés, et
$$\varPhi=\{S\in\xi E\,|\,\hbox{il existe $i\in \ob I$ et $S_i\in\varPhi_i$ tel que}\  S=\e_i(S_i)\}\ ,$$
où $\e_i:E_i\to E$ désigne le morphisme canonique.
\end{paragr}

\begin{paragr}\label{kappacs}
Il résulte du paragraphe précédent que le foncteur \og ensemble ordonné sous-jacent à un complexe simplicial\fg{} commute à la fois aux limites inductives et projectives. En fait, il admet un adjoint à gauche et un adjoint à droite. L'adjoint à gauche associe à un ensemble ordonné $E$ le complexe simplicial $(E,\varPhi)$, où $\varPhi=\{\{x\}\,|\,x\in E\}$. L'adjoint à droite, noté $\kappa:\ord\to\cs$, associe à un ensemble ordonné $E$ le complexe simplicial $(E,\xi E)$. Aussi bien l'adjoint à gauche que l'adjoint à droite sont des foncteurs pleinement fidèles. On identifiera parfois $\ord$ à une sous-catégorie pleine de $\cs$ par le foncteur $\kappa$.
\end{paragr}

\begin{paragr}\label{kappastarcs}
On note également $\kappa:\cats\to\cs$ la restriction du foncteur $\kappa$ du paragraphe précédent à la catégorie des simplexes. On en déduit, par le procédé de Kan, un couple de
foncteurs adjoints
$$\kappa^{}_!:\simpl\to\cs\ ,\qquad \kappa^*:\cs\to\simpl\ ,$$
où $\kappa^{}_!$ est l'unique foncteur, à isomorphisme unique près, commutant aux
petites limites inductives et prolongeant le foncteur $\kappa$, et $\kappa^*$ est
défini par
$$
(E,\varPhi)\longmapsto\bigl(\smp{m}\mapsto\Hom_{\cs}(\kappa(\smp{m}),(E,\varPhi))=\Hom_{\cs}((\smp{m},\xi\smp{m}),(E,\varPhi))\bigr)\ .
$$
Ainsi, pour un complexe simplicial $(E,\varPhi)$, l'ensemble des $m$\nbd-simplexes de $\kappa^*(E,\varPhi)$ est formé des applications croissantes $f:\smp{m}\to E$ telles que $f(\smp{m})\in\varPhi$. On en déduit que le composé
$$\xymatrix{
\ord\ar[r]^{\kappa}
&\cs\ar[r]^{\kappa^*}
&\simpl
}$$
associe à un ensemble ordonné son nerf. En particulier, pour tout $n\geq0$, on a~$\kappa^*(\smp{n},\xi\smp{n})=\smp{n}$.
\end{paragr}

\begin{paragr}
Dans la suite, on s'intéresse au foncteur composé $\Chcs=\Chnorm{}\kappa^*$
$$\xymatrix{
\cs\ar[r]^{\kappa^*}
&\simpl\ar[r]^{\Chnorm{}}
&\ADC
}\ .$$
Soit $(E,\varPhi)$ un complexe simplicial. Le complexe dirigé augmenté $\Chcs(E,\varPhi)$ se décrit comme suit. Pour $p\geq0$, un $p$\nbd-simplexe non dégénéré de l'ensemble simplicial $\kappa^*(E,\varPhi)$ est une application strictement croissante $\smp{p}\to E$ dont l'image est dans $\varPhi$. Ainsi, pour $p\geq0$, $(\Chcs(E,\varPhi))_p$ (resp.~$(\Chcs(E,\varPhi))^*_p$) s'identifie au groupe (resp. au monoïde) commutatif libre engendré par la famille des \hbox{$(p+1)$}\nbd-uplets 
$$(i_0,i_1,\dots,i_p)\,,\quad \{i_0,i_1,\dots,i_p\}\in\varPhi\ ,\quad i_0<i_1<\cdots<i_p\ .$$
La différentielle est définie par 
$$d(i_0,i_1,\dots, i_p)=\textstyle\sum\limits_{0\leq k\leq p}(-1)^k(i_0,\dots, \widehat{i}_k,\dots, i_p)\,,\quad p>0\ ,$$
où $(i_0,\dots, \widehat{i}_k,\dots, i_p)=(i_0,\dots, i_{k-1},i_{k+1},\dots, i_p)$, et l'augmentation par $\augm(i_0)=1$. 
\end{paragr}

\begin{thm}\label{csunitsansboucles}
Pour tout complexe simplicial $(E,\varPhi)$, l'ensemble
$$B=\{(i_0,i_1,\dots,i_p)\,|\,p\geq0\,,\ \{i_0,i_1,\dots,i_p\}\in\varPhi\,,\ i_0<i_1<\cdots<i_p\}$$
est une base unitaire fortement sans boucles du complexe dirigé augmenté $\Chcs(E,\varPhi)$.
\end{thm}

\begin{proof}
En vertu de la description de $\Chcs(E,\varPhi)$ donnée au paragraphe précédent, l'ensemble $B$ est bien une base du complexe dirigé augmenté $\Chcs(E,\varPhi)$. On abrégera la notation des éléments de la base $B$ en posant
$$i_0i_1\cdots i_p = (i_0,i_1,\dots,i_p)\ .$$
Pour montrer que cette base est fortement sans boucles, il suffit de définir une relation d'ordre $\preccurlyeq$ sur $B$ moins fine que la relation de préordre $\leqbb$, autrement dit, vue la définition de la relation $\leqbb$ et celle de la différentielle de $\Chcs(E,\varPhi)$, telle que 
$$\begin{aligned}
&i_0i_1\cdots\widehat{i}_k\cdots i_p\preccurlyeq i_0i_1\cdots i_p\ ,\qquad\hbox{pour $k$ impair,}\cr
\noalign{\vskip 3pt}
&i_0i_1\cdots i_p\preccurlyeq i_0i_1\cdots\widehat{i}_k\cdots i_p\ ,\qquad\hbox{pour $k$ pair,}
\end{aligned}\leqno(*)$$
pour tous $p\geq1$ et $i_0i_1\cdots i_p\in B$. On définit une telle relation 
$$i_0i_1\cdots i_p\preccurlyeq j_0j_1\cdots j_q$$
sur $B$ par récurrence sur $p$ comme suit. Si $p=0$,
$$i_0\preccurlyeq j_0j_1\cdots j_q\ \Longleftrightarrow\ i_0\leq j_0\ .$$
Pour $p>0$,
$$i_0i_1\cdots i_p\preccurlyeq j_0j_1\cdots j_q\ \Longleftrightarrow\ \left\{\begin{aligned}
&\kern 20pt i_0<j_0\cr
\noalign{\vskip -7pt}
&\hbox{ou}\cr
\noalign{\vskip -7pt}
&\kern 20pt i_0=j_0\ ,\quad q>0\ ,\quad j_1\cdots j_q\preccurlyeq i_1\cdots i_p\ .
\end{aligned}\right.$$
Vérifions que la relation ainsi définie est une relation d'ordre. La réflexivité de $\preccurlyeq$ est conséquence immédiate de la réflexivité de la relation d'ordre $\leq$ sur $E$. Pour montrer l'antisymétrie supposons que l'on ait
$$i_0i_1\cdots i_p\preccurlyeq j_0j_1\cdots j_q\qquad\hbox{et}\qquad j_0j_1\cdots j_q\preccurlyeq i_0i_1\cdots i_p\ .$$
La situation étant symétrique, on peut supposer que $q\leq p$. On raisonne par récurrence sur $p$. L'antisymétrie de la relation $\leq$ implique aussitôt que $i_0=j_0$, ce qui prouve en particulier le cas $p=0$. Si $p>0$, comme $i_0=j_0$, la première relation implique que $q>0$ et $j_1\cdots j_q\preccurlyeq i_1\cdots i_p$ et la deuxième que $i_1\cdots i_p\preccurlyeq j_1\cdots j_q$. Il résulte donc de l'hypothèse de récurrence que $j_1\cdots j_q=i_1\cdots i_p$, ce qui prouve l'antisymétrie. Pour montrer la transitivité, supposons que l'on ait
$$i_0i_1\cdots i_p\preccurlyeq j_0j_1\cdots j_q\qquad\hbox{et}\qquad j_0j_1\cdots j_q\preccurlyeq k_0k_1\cdots k_r\ $$
et prouvons que 
$$i_0i_1\cdots i_p\preccurlyeq k_0k_1\cdots k_r\ .$$
On raisonne par récurrence sur $l=\max\{p,r\}$. 
La transitivité de la relation $\leq$ implique aussitôt que $i_0\leq k_0$, ce qui prouve en particulier le cas $p=0$ et à plus forte raison le cas $l=0$. Dans le cas général, si $i_0<j_0$ ou $j_0<k_0$, on a $i_0<k_0$ et l'assertion est évidente. Reste le cas où $i_0=j_0=k_0$ et $p>0$. Alors la première relation implique que $q>0$ et $j_1\cdots j_q\preccurlyeq i_1\cdots i_p$ et la deuxième que $r>0$ et $k_1\cdots k_r\preccurlyeq j_1\cdots j_q$. L'hypothèse de récurrence implique alors que $k_1\cdots k_r\preccurlyeq i_1\cdots i_p$, ce qui prouve l'assertion. Il reste à démontrer les relations $(*)$. On raisonne par récurrence sur $p\geq1$. Pour $p=1$, si $i_0i_1\in B$, on a 
$$i_0\preccurlyeq i_0i_1\preccurlyeq i_1\ ,$$
puisque $i_0\leq i_0<i_1$. Supposons donc que $p>1$, et soit $i_0i_1\cdots i_p\in B$. Si $k=0$, on a 
$$i_0i_1\cdots i_p\preccurlyeq i_1\cdots i_p\ ,$$
puisque $i_0<i_1$. Si $k>0$ et $k$ est pair (resp. impair), comme $k-1$ est impair (resp. pair), on a par hypothèse de récurrence
$$i_1\cdots\widehat{i}_k\cdots i_p\preccurlyeq i_1\cdots i_p\qquad\hbox{(resp. \ }i_1\cdots i_p\preccurlyeq i_1\cdots\widehat{i}_k\cdots i_p\ )\ ,$$
d'où
$$i_0i_1\cdots i_p\preccurlyeq i_0i_1\cdots\widehat{i}_k\cdots i_p\qquad\hbox{(resp. \ }i_0i_1\cdots\widehat{i}_k\cdots i_p\preccurlyeq i_0i_1\cdots i_p\ )\ .$$
Le fait que la base $B$ est unitaire résulte 
du lemme suivant.
\end{proof}

\begin{lemme}\label{lemmeaprouver}
En gardant les notations ci-dessus, soient $p\geq0$ et $i_0i_1\cdots i_p\in B$. Alors pour tout $q$, $0\leq q\leq p$ et $\e\in\{0,1\}$, on a
$$\atom{i_0i_1\cdots i_p}^\e_q=\sum d^{k_{1}}_{q+1} d^{k_{2}}_{q+2}\cdots d^{k_{p-q}}_p(i_0i_1\cdots i_p)\ ,$$
la somme portant sur les suites d'entiers $(k_1,k_2\dots,k_{p-q})$ de parités alternées tels que $0\leq k_1<k_2<\cdots <k_{p-q}\leq p$, l'entier $k_{1}$ étant pair si $\e=1$ et impair si $\e=0$. En particulier,
$$\atom{i_0i_1\cdots i_p}^0_1=i_0i_p\ ,\quad\atom{i_0i_1\cdots i_p}^1_1=i_0i_1+i_1i_2+\cdots i_{p-1}i_p\ $$
et
$$\atom{i_0i_1\cdots i_p}^0_0=i_0\ ,\quad\atom{i_0i_1\cdots i_p}^1_1=i_p\ .$$
\end{lemme}

\begin{proof}
Fixons un entier $p\geq0$. Pour tout $q$, $0\leq q\leq p$, on pose
$$I_q=\{(k_1,k_2,\dots,k_{p-q})\,|\,0\leq k_1<k_2<\cdots<k_{p-q}\leq p\}\ ,$$
et pour tout $K=(k_1,k_2,\dots,k_{p-q})\in I_q$, on note $d^K_p$ l'opérateur simplicial
$$d^K_p=d^{k_{1}}_{q+1} d^{k_{2}}_{q+2}\cdots d^{k_{p-q}}_p\ ,$$
étant entendu que si $q=p$, de sorte que la suite $K$ soit vide, l'opérateur $d^K_p$ est l'identité.
On pose, pour $\e\in\{0,1\}$,
$$J^\e_q=\{(k_1,k_2,\dots,k_{p-q})\in I_q\,|\,(-1)^{k_i}=(-1)^{i+\e},\ 1\leq i\leq p-q\}\ .$$
Il s'agit donc de montrer que pour tout $i_0i_1\cdots i_p\in B$, on a 
$$\atom{i_0i_1\cdots i_p}^\e_q=\sum_{K\in J^\e_q}d^K_p(i_0i_1\cdots i_p)\ .$$
On procède par récurrence descendante sur $q$ (de $q=p$ à $q=0$). La formule est évidente pour $q=p$ et $q=p-1$. Supposons-la au rang $q$ (avec $0<q<p$) et montrons-la au rang $q-1$. On considère donc
$$d_q\atom{i_0i_1\cdots i_p}^1_q=\sum_{K\in J^1_q}d_qd^K_p(i_0i_1\cdots i_p)=\sum_{K\in J^1_q\,}\sum_{j=0}^q(-1)^jd^j_qd^K_p(i_0i_1\cdots i_p)\ .$$
En tenant compte de la relation simpliciale
$$d^n_rd^m_{r+1}=d^m_{r}d^{n+1}_{r+1}\ ,\qquad r>0\ ,\quad0\leq m\leq n\leq r\ ,$$
on en déduit que
$$d_q\atom{i_0i_1\cdots i_p}^1_q=\sum_{K\in J^1_q\,}\sum_{j=0}^q(-1)^jd^{\varphi(j,K)}_{p}(i_0i_1\cdots i_p)\ ,\leqno(*)$$
où la fonction
$$\varphi:\{0,1,\dots,q\}\times J^1_q\toto I_{q-1}$$
est définie comme suit. Pour $0\leq j\leq q$ et $K=(k_1,\dots,k_{p-q})\in J^1_q$, il existe (en posant $k_0=-1$ et $k_{p-q+1}=p+1$) un unique entier $m(j,K)$ tel que 
$$0\leq m(j,K)\leq p-q\qquad\hbox{et}\qquad k_{m(j,K)}<j+m(j,K)<k_{m(j,K)+1}\ ,$$
et par définition, 
$$\varphi(j,K)=(k_1,\dots,k_{m(j,K)},j+m(j,K),k_{m(j,K)+1},\dots,k_q)\ .$$
L'image de l'application $\varphi$ est la réunion (disjointe) des ensembles $J^1_{q+1}$, $J^0_{q+1}$, et de l'ensemble $J'^1_{q+1}$ des suites appartenant à $I_{q+1}$ ayant exactement un couple formé de deux termes consécutifs de même parité et dont le premier terme est pair. Il est évident que toute suite appartenant à l'image de $\varphi$ appartient à cette réunion. Réciproquement, il y a trois cas à examiner. 
\smallbreak

a) La suite $(l_1,\dots,l_{p-q+1})$ appartient à $J'^1_{q+1}$, autrement dit, comporte exactement un couple de termes consécutifs de même parité $(l_r,l_{r+1})$, $0\leq r\leq p-q$, et $l_1$ est pair. Alors la suite $(l_1,\dots,l_{p-q+1})$ admet exactement deux antécédents dans \hbox{$\{0,1,\dots,q\}\times J^1_q$}, à savoir $(j_1,K_1)$ et $(j_2,K_2)$, où
$$\begin{aligned}
&j_1=l_r-r+1\ ,\quad K_1=(l_1,\dots,\widehat{l}_r,\dots,l_{p-q+1})\ ,\cr
&j_2=l_{r+1}-r\ ,\quad K_2=(l_1,\dots,\widehat{l}_{r+1},\dots,l_{p-q+1})\ .
\end{aligned}$$
En observant que $j_1$ et $j_2$ sont de parité opposée, on en déduit que les deux termes correspondants dans la somme de la formule $(*)$ s'éliminent mutuellement.
\smallbreak

b) La suite $(l_1,\dots,l_{p-q+1})$ appartient à $J^0_{q+1}$. Alors cette suite admet exactement un antécédent, à savoir $(j,K)$, où $j=l_1$ et $K=(l_2,\dots,l_{p-q+1})$. Comme $j$ est alors impair, le signe du terme correspondant dans la somme de la formule $(*)$ est $-1$.
\smallbreak

c) La suite $(l_1,\dots,l_{p-q+1})$ appartient à $J^1_{q+1}$. Alors cette suite admet exactement un antécédent, à savoir $(j,K)$, où $j=l_{p-q+1}-(p-q)$ et $K=(l_1,\dots,l_{p-q})$. Comme les parités dans la suite $(l_1,\dots,l_{p-q+1})$ sont alternées, et comme $l_1$ est pair, on a $(-1)^{l_{p-q+1}}=(-1)^{p-q}$, et par suite $j$ est pair. On en déduit que le signe du terme correspondant dans la somme de la formule $(*)$ est $+1$.
\smallbreak

En vertu de ces considérations, on a donc l'égalité
$$d_q\atom{i_0i_1\cdots i_p}^1_q=\sum_{\kern 8pt L\in J^1_{q+1}}\!\!d^L_p(i_0i_1\cdots i_p)\,-\kern -7pt\sum_{\kern 8pt L\in J^0_{q+1}}\!\!d^L_p(i_0i_1\cdots i_p)\ ,$$
ce qui prouve le lemme.
\end{proof}

\begin{rem}
La démonstration du théorème~\ref{csunitsansboucles} est directement inspirée de la preuve du cas particulier où $E=\smp{n}$ et $\varPhi=\xi\smp{n}$ esquissée par Steiner~\cite[exemple~3.8]{Steiner}. Une preuve détaillée du lemme~\ref{lemmeaprouver}, pour ce cas particulier, figure dans un texte non publié de Burroni et Penon~\cite{B-P}.
\end{rem}

\begin{rem}
En gardant les notations de la preuve du théorème~\ref{csunitsansboucles}, il faut se garder de croire que pour un complexe simplicial général, la relation d'ordre $\preccurlyeq$ sur la base $B$ coïncide avec la relation $\leqbb$ définie dans le paragraphe~\ref{deffortsansboucles}. En effet, la relation $\leqbb$ est obtenue par clôture par transitivité et réflexivité à partir des relations \og élémentaires\fg{}
$$i_0i_1\cdots\widehat{i}_k\cdots i_p\leqbb i_0i_1\cdots i_p\ ,\qquad\hbox{pour $k$ impair,}\leqno(1)$$
$$i_0i_1\cdots i_p\leqbb i_0i_1\cdots\widehat{i}_k\cdots i_p\ ,\qquad\hbox{pour $k$ pair,}\phantom{\hbox{im}}\leqno(2)$$
pour $p\geq1$ et $i_0i_1\cdots i_p\in B$. Ainsi, si par exemple 
$$E=\smp{1}=\{0<1\}\quad \hbox{et}\quad \varPhi=\{\{0\},\{1\}\}\ ,$$ 
l'ensemble des relations \og élémentaires\fg{} ci-dessus est vide et la relation $\leqbb$ est l'égalité sur $B$, tandis que par définition de $\preccurlyeq$ on a $0\preccurlyeq1$.
\smallbreak

En revanche, si $E$ est totalement ordonné et $\varPhi=\xi E$, alors la relation d'ordre $\leqbb$ coïncide avec la relation $\preccurlyeq$ et est une relation d'ordre total. Pour le voir, il suffit de démontrer cette dernière assertion. En effet, dans la preuve du théorème~\ref{csunitsansboucles} on a montré que pour tous $p,q\geq0$ et $i_0i_1\cdots i_p,\,j_0j_1\cdots j_q\in B$, on a
$$i_0i_1\cdots i_p\,\leqbb\, j_0j_1\cdots j_q\ \Longrightarrow\ i_0i_1\cdots i_p\,\preccurlyeq\, j_0j_1\cdots j_q\ .$$
Réciproquement, si $i_0i_1\cdots i_p\,\preccurlyeq\, j_0j_1\cdots j_q$ 
il en résulte qu'on n'a pas l'inégalité stricte
$j_0j_1\cdots j_q\leqbbs i_0i_1\cdots i_p$, et si la relation d'ordre $\leqbb$ est total, on en déduit que $i_0i_1\cdots i_p\,\leqbb\, j_0j_1\cdots j_q$.
\smallbreak

Pour montrer que si $E$ est totalement ordonné et $\varPhi=\xi E$, alors la relation d'ordre $\leqbb$ est un ordre total, on observe d'abord (sans aucune hypothèse sur $(E,\varPhi)$) que pour tous $p\geq0$ et $i_0i_1\cdots i_p\in B$, et tout entier $k$, $0\leq k\leq p$, on a par une application répétée de (1) si $k$ est pair
$$i_0i_1\cdots i_k\leqbb i_0i_1\cdots i_ki_p\leqbb i_0i_1\cdots i_ki_{p-1}i_p\leqbb\cdots\leqbb i_0i_1\cdots\widehat{i}_{k+1}\cdots i_p\leqbb i_0i_1\cdots i_p$$
et par une application répétée de (2) si $k$ est impair
$$i_0i_1\cdots i_p\leqbb i_0i_1\cdots\widehat{i}_{k+1}\cdots i_p\leqbb\cdots\leqbb i_0i_1\cdots i_ki_{p-1}i_p\leqbb i_0i_1\cdots i_ki_p\leqbb i_0i_1\cdots i_k\,.$$
On a donc
$$i_0i_1\cdots i_k\leqbb i_0i_1\cdots i_p\ ,\qquad\hbox{pour $k$ pair,\phantom{im}}\leqno(3)$$
$$i_0i_1\cdots i_p\leqbb i_0i_1\cdots i_k\ ,\qquad\hbox{pour $k$ impair.}\leqno(4)$$
\smallbreak

Soient maintenant $p,q\geq0$ et $i_0i_1\cdots i_p,\,j_0j_1\cdots j_q$ deux éléments de la base $B$. Il s'agit de montrer (sous l'hypothèse $E$ totalement ordonné et $\varPhi=\xi E$) que
$$i_0i_1\cdots i_p\leqbb j_0j_1\cdots j_q\qquad\hbox{ou}\qquad j_0j_1\cdots j_q\leqbb i_0i_1\cdots i_p\ .$$
Si l'un d'eux est une section commençante de l'autre, cela résulte des relations $(3)$ et~$(4)$. On peut donc supposer qu'il existe un entier $k$, $0\leq k\leq\min\{p,q\}$, tel que 
$$i_{k'}=j_{k'}\,,\quad0\leq k'<k\ ,\qquad\hbox{et}\qquad i_k\neq j_k\ .$$
Comme $E$ est totalement ordonné, on a $i_k<j_k$ ou $j_k<i_k$, et par symétrie, on peut supposer que $i_k<j_k$. Supposons d'abord que $k$ soit pair. On distingue plusieurs cas.
\begin{itemize}
\item $k=p$. Alors on a
$$\begin{aligned}
i_0\cdots i_p=i_0\cdots i_{k-1}i_k\,\leqbb\,i_0\cdots i_{k-1}i_kj_k\,\leqbb\,i_0\cdots i_{k-1}j_k&\cr
=j_0\cdots j_{k-1}j_k&\,\leqbb\,j_0\cdots j_q\ ,
\end{aligned}$$
la première inégalité résultant de $(1)$ et du fait que comme $\varPhi=\xi E$ on a $\{i_0,\dots,i_{k-1},i_k,j_k\}\in\varPhi$, la deuxième résultant de $(2)$ et la troisième de $(3)$.
\item $k<p$. Alors on a
$$i_0\cdots i_p\,\leqbb\,i_0\cdots i_{k-1}i_ki_{k+1}\,\leqbb\,i_0\cdots i_{k-1}i_{k+1}$$
la première inégalité résultant de $(4)$ et la deuxième de $(2)$. Comme $E$ est totalement ordonné, on distingue trois cas.
\begin{itemize}
\item $i_{k+1}=j_k$. Alors en vertu de $(3)$, on a
$$i_0\cdots i_{k-1}i_{k+1}=j_0\cdots j_{k-1}j_k\,\leqbb\,j_0\cdots j_q\ .$$
\item $i_{k+1}<j_k$. Alors en vertu du cas $k=p$, on a
$$i_0\cdots i_{k-1}i_{k+1}\,\leqbb\,j_0\cdots j_q\ .$$
\item $j_k<i_{k+1}$. Alors on a
$$\begin{aligned}
i_0\cdots i_{k-1}i_ki_{k+1}\,\leqbb\,{}&i_0\cdots i_{k-1}i_kj_ki_{k+1}\,\leqbb\,i_0\cdots i_{k-1}i_kj_k\,\cr
\leqbb\,{}&i_0\cdots i_{k-1}j_k=j_0\cdots j_{k-1}j_k\,\leqbb\,j_0\cdots j_q\ ,
\end{aligned}$$
la première inégalité résultant de $(1)$ et du fait que comme $\varPhi=\xi E$ on a $\{i_0,\dots,i_{k-1},i_k,j_k,i_{k+1}\}\in\varPhi$, la deuxième et la troisième résultant de $(2)$ et la quatrième de $(3)$.
\end{itemize}
\end{itemize}
Dans le cas où $k$ est impair, on procède exactement de la même façon, toutes les inégalités étant inversées et les applications de $(1)$ et $(2)$ et de $(3)$ et $(4)$ étant échangées.
\end{rem}

\begin{cor}\label{oocatcspolygr}
Pour tout complexe simplicial $(E,\varPhi)$, la $\infty$\nbd-catégorie $\Stnu\Chcs(E,\varPhi)$ est librement engendrée au sens des polygraphes par les atomes
$$\atom{i_0i_1\cdots i_p}\ ,\qquad p\geq 0\ ,\quad i_0<i_1<\cdots<i_p\ ,\quad\{i_0,i_1,\dots,i_p\}\in\varPhi\ .$$
\end{cor}

\begin{proof}
Le corollaire est conséquence immédiate des théorèmes~\ref{unitsansbouclespolygr} et~\ref{csunitsansboucles}, et de la proposition~\ref{propfortsansboucles}.
\end{proof}

\section{La \pdfinfty-catégorie associée à un complexe simplicial}

\begin{paragr}\label{foncteurO}
On dispose de deux foncteurs de source la catégorie des complexes simpliciaux et but celle des $\infty$\nbd-catégories, à savoir le composé $\OrCs=\Stnu\Chcs$
$$\xymatrix{
\cs\ar[r]^{\Chcs}
&\ADC\ar[r]^-{\Stnu}
&\ooCat
}\ ,$$
et le composé
$$\xymatrix{
\cs\ar[r]^-{\kappa^*}
&\simpl\ar[r]^-{\cStreet}
&\ooCat
}\ .$$
On va montrer que ces deux composés sont canoniquement isomorphes. Plus précisément, on a un morphisme d'adjonction $1_{\ooCat}\to\Stnu\Stlambda$ (cf. proposition~\ref{adjSteiner}), d'où en vertu de la proposition~\ref{commnerf}, un morphisme de foncteurs
$$\cStreet\kappa^*\toto\Stnu\Stlambda\cStreet\kappa^*\simeq\Stnu\Chnorm{}\kappa^*=\Stnu\Chcs=\OrCs\ .$$
On va montrer que ce dernier est un isomorphisme. La $\infty$\nbd-catégorie associée ainsi à un complexe simplicial $(E,\varPhi)$ sera parfois appelée \emph{l'oriental} de $(E,\varPhi)$.
\end{paragr}

\begin{paragr}
Soit $(E,\varPhi)$ un complexe simplicial. Pour tout $S\in\varPhi$, on note $i_S$ l'inclusion \hbox{$i_S:S\to E$} et pour tout $S'\subset S$, on note $i_{S,S'}$ l'inclusion \hbox{$i_{S,S'}:S'\to S$.} En considérant l'ensemble $\varPhi$ ordonné par inclusion comme une catégorie, on définit un foncteur $\varPhi\to\cs$ en associant à $S\in\varPhi$ le complexe simplicial $(S,\xi S)$ et à une inclusion $S'\subset S$ dans $\varPhi$, le morphisme de complexes simpliciaux \hbox{$i_{S,S'}:(S',\xi S')\to(S,\xi S)$.} 
Comme $i_Si_{S,S'}=i_{S'}$, les morphismes $i_S:(S,\xi S)\to(E,\varPhi)$, $S\in\varPhi$, définissent un cône inductif, d'où un morphisme canonique de complexes simpliciaux 
$$\limind_{S\in\varPhi}(S,\xi S)\to(E,\varPhi)\ .$$
Bien que ce morphisme induise une bijection $\limind_{S\in\varPhi}\xi S\to\varPhi$, il faut se garder de croire qu'il est toujours un isomorphisme. Par exemple, ce n'est pas le cas si \hbox{$\varPhi=\{\{x\}\,|\,x\in E\}$} et si la relation d'ordre sur $E$ n'est pas l'égalité. En revanche, il est facile de vérifier qu'il est bien un isomorphisme si la relation d'ordre sur $E$ est engendrée par les couples $x,y\in E$ tels que $x<y$ et $\{x,y\}\in \varPhi$ (et à plus forte raison si $\varPhi=\xi E$), mais on n'aura pas besoin de ce résultat.
\end{paragr}

\begin{lemme}\label{limindkappa}
Le morphisme canonique $\limind_{S\in\varPhi}\kappa^*(S,\xi S)\to\kappa^*(E,\varPhi)$ est un isomorphisme d'ensembles simpliciaux. En particulier, 
l'ensemble sous-jacent à $E$ s'identifie à la limite inductive des ensembles sous-jacents aux $S\in\varPhi$.
\end{lemme}

\begin{proof}
En gardant les notations du paragraphe précédent, on va montrer que pour tout $n\geq0$, l'application canonique $\limind_{S\in\varPhi}\kappa^*(S,\xi S)_n\to\kappa^*(E,\varPhi)_n$ est bijective, ce qui prouvera la première assertion. Un $n$\nbd-simplexe de $\kappa^*(E,\varPhi)$ est une application croissante $f:\smp{n}\to E$ dont l'image $S$ est dans $\varPhi$. L'application $f$ induit une application croissante $f':\smp{n}\to S$ qui est un $n$\nbd-simplexe de $(S,\xi S)$ dont l'image par~$i_S$ est $f$, ce qui prouve la surjectivité. Pour montrer l'injectivité, soient \hbox{$f_1:\smp{n}\to S_1$} et~$f_2:\smp{n}\to S_2$ des $n$\nbd-simplexes de $\kappa^*(S_1,\xi S_1)$ et $\kappa^*(S_2,\xi S_2)$ respectivement ayant même image $f$ dans $\kappa^*(E,\varPhi)_n$. Alors $f=i_{S_1}f_1=i_{S_2}f_2$, et si $S$ est l'image de $f$, et $f':\smp{n}\to S$ l'application induite par $f$, on a $S\subset S_1\cap S_2$, et $f_1$ et~$f_2$ sont l'image de $f'$ par $i_{S_1,S}$ et $i_{S_2,S}$ respectivement, ce qui prouve qu'ils ont même image dans la limite inductive.
Enfin, en remarquant que pour tout complexe simplicial $(E,\varPhi)$, l'ensemble $(\kappa^*(E,\varPhi))_0$ des $0$\nbd-simplexes de $\kappa^*(E,\varPhi)$ s'identifie à l'ensemble sous-jacent à l'ensemble ordonné $E$, la dernière assertion en résulte.
\end{proof}

\begin{prop}\label{limindOr}
Le morphisme canonique $\limind_{S\in\varPhi}\OrCs(S,\xi S)\to\OrCs(E,\varPhi)$ est un isomorphisme de $\infty$\nbd-catégories.
\end{prop}

\begin{proof}
En vertu du corollaire~\ref{oocatcspolygr}, la $\infty$\nbd-catégorie $\OrCs(E,\varPhi)$ est librement engendrée au sens des polygraphes par l'ensemble des atomes
$$\atom{i_0i_1\cdots i_p}\ ,\qquad p\geq 0\ ,\quad i_0<i_1<\cdots<i_p\ ,\quad\{i_0,i_1,\dots,i_p\}\in\varPhi\ ,$$
qui s'identifie à l'ensemble $\varPhi$. De même, pour tout $S\in\varPhi$, la $\infty$\nbd-catégorie $\OrCs(S,\xi S)$ est librement engendrée au sens des polygraphes par l'ensemble des atomes
$$\atom{i_0i_1\cdots i_p}\ ,\qquad p\geq 0\ ,\quad i_0<i_1<\cdots<i_p\ ,\quad\{i_0,i_1,\dots,i_p\}\subset S\ ,$$
qui s'identifie à l'ensemble $\xi S$. Comme l'application $\limind_{S\in\varPhi}\xi S\to\varPhi$ est bijective, la proposition résulte donc de la proposition~\ref{englibrclef}.
\end{proof}

\begin{thm}\label{OrisoNkappa}
Le morphisme de foncteurs $\cStreet\kappa^*\to\OrCs$ \emph{(cf.~\ref{foncteurO})} est un isomorphisme. 
\end{thm}

\begin{proof}
Pour tout complexe simplicial $(E,\varPhi)$, on a un carré commutatif de $\ooCat$
$$
\xymatrixcolsep{3pc}
\xymatrix{
\limind_{S\in\varPhi}\cStreet\kappa^*(S,\xi S)\ar[r]\ar[d]
&\limind_{S\in\varPhi}\OrCs(S,\xi S)\ar[d]
\\
\cStreet\kappa^*(E,\varPhi)\ar[r]
&\OrCs(E,\varPhi)
}$$
dont les flèches verticales sont des isomorphismes, celle de droite en vertu de la proposition~\ref{limindOr}, et celle de gauche en vertu de la proposition~\ref{limindkappa} et du fait que~$\cStreet$, étant un adjoint à gauche, commute aux limites inductives. Pour montrer que $\cStreet\kappa^*(E,\varPhi)\to\OrCs(E,\varPhi)$ est un isomorphisme, il suffit donc de le montrer pour $(E,\varPhi)=(\smp{m},\xi\smp{m})$, $m\geq0$ (puisque pour tout $S\in\varPhi$, le complexe simplicial $(S,\xi S)$ est isomorphe à $(\smp{m},\xi\smp{m})$, pour $m=\card\, S-1$). Or, $\cStreet\kappa^*(\smp{m},\xi\smp{m})=\cStreet(\smp{m})$ est par définition égal à $\Stnu\Chnorm{(\smp{m})}=\Stnu\Chnorm\kappa^*(\smp{m},\xi\smp{m})=\OrCs(\smp{m},\xi\smp{m})$, ce qui prouve l'assertion.
\end{proof}

\begin{prop}\label{Orespmono}
Le foncteur $\OrCs$ respecte les monomorphismes.
\end{prop}

\begin{proof}
Par définition, $\OrCs=\Stnu\Chcs=\Stnu\Chnorm{}\kappa^*$. Les foncteurs $\kappa^*$ et $\Stnu$ étant des adjoints à droite ils respectent les monomorphismes. Comme l'image par un monomorphisme d'ensembles simpliciaux d'un simplexe non dégénéré est un simplexe non dégénéré, le foncteur $\Chnorm{}$ respecte aussi les monomorphismes, ce qui prouve la proposition.
\end{proof}

\section{Le type d'homotopie de la \pdfinfty-catégorie associée\\ à un ensemble ordonné}

\begin{paragr}
Dans cette section on s'intéresse à la restriction du foncteur $\OrCs$ à la catégorie $\ord$, considérée comme sous catégorie pleine de $\cs$ \emph{via} le foncteur $\kappa$ (cf.~\ref{kappacs}). Cette restriction sera notée aussi $\OrCs:\ord\to\ooCat$.
\end{paragr}

\begin{lemme}\label{Ormono}
Le foncteur $\OrCs:\ord\to\ooCat$ respecte les monomorphismes.
\end{lemme}

\begin{proof}
Comme le foncteur $\kappa$ est un adjoint à droite (cf.~\ref{kappacs}), il respecte les monomorphismes, et l'assertion résulte de la proposition~\ref{Orespmono}.
\end{proof}
\goodbreak

\emph{Dans la suite de cette section on se fixe un ensemble ordonné $E$.}

\begin{paragr}\label{untronqueor}
En vertu du corollaire~\ref{oocatcspolygr}, $\OrCS{E}$ est la $\infty$\nbd-catégorie librement engendrée au sens des polygraphes par les atomes
$$\atom{i_0i_1\cdots i_k}\ ,\qquad k\geq 0\ ,\quad i_0<i_1<\cdots<i_k\ ,\quad\{i_0,i_1,\dots,i_k\}\in\xi E\ .$$
En particulier, le $1$\nbd-tronqué bête de $\OrCS{E}$ est la catégorie librement engendrée par le graphe dont les sommets sont les éléments de $E$ et dont les arêtes sont les couples $(i,j)$, $i<j$. Les $1$\nbd-flèches de $\OrCS{E}$ s'identifient donc aux éléments $S$ de $\xi E$, la source et le but de la flèche $S$ étant définis respectivement par $s(S)=\min(S)$ et $t(S)=\max(S)$, la composition de deux flèches composables étant leur réunion, et l'identité d'un objet $i\in E$ étant le singleton $\{i\}$. Par cette identification, si $S=\{i_0,\dots,i_n\}$, $i_0<\cdots<i_n$, alors $S$ représente la $1$\nbd-flèche
$$
\begin{pmatrix}
(i_0)
&(i_0,i_1)+(i_1,i_2)+\cdots+(i_{n-1},i_n)
\cr
(i_n)
&(i_0,i_1)+(i_1,i_2)+\cdots+(i_{n-1},i_n)
\end{pmatrix}
$$
de $\OrCS{E}=\Stnu\Chcs(E)$.
\end{paragr}

\begin{lemme}\label{cnsexfl}
Soient $S,S'\in\xi E$ deux $1$\nbd-flèches parallèles de $\OrCS{E}$. Pour qu'il existe une $2$\nbd-flèche de $S$ vers $S'$, il faut et il suffit que $S\subset S'$.
\end{lemme}

\begin{proof}
Comme les $1$\nbd-flèches $S$ et $S'$ sont parallèles, on a \hbox{$\min(S)=\min(S')$} et $\max(S)=\max(S')$. Pour montrer que si $S'\subset S$, alors il existe une $2$\nbd-flèche de $S$ vers $S'$, il suffit (par récurrence sur $\card(S)-\card(S')$) de le montrer quand \hbox{$\card(S)=\card(S')+1$}. Soit $j$ l'unique élément de $S\sauf S'$. On a \hbox{$\min(S')<j<\max(S')$}, et si l'on pose
$$S_1=\{l\in S\,|\,l<j\}\ ,\quad S_2=\{l\in S\,|\,l>j\}\ ,\quad i=\max(S_1)\ ,\quad k=\min(S_2)\ ,$$
on a une $2$\nbd-flèche
$$S_2\comp{0}\atom{ijk}\comp{0}S_1:S'\toto S\ .$$
La réciproque résulte du lemme plus précis suivant.
\end{proof}

\begin{lemme}\label{surjcel}
Soit $\alpha$ une $n$\nbd-flèche de $\OrCS{E}$, $n\geq2$, de $1$-cellule but itéré $S$ et de $1$-cellule source itérée $S'$.
Alors $S'\subset S$, et si l'on note $i_S:S\to E$ l'inclusion d'ensembles ordonnés de $S$ dans $E$, $\alpha$ est l'image d'une $n$\nbd-flèche de $\OrCS{S}$ par le $\infty$\nbd-foncteur $\OrCS{S}\to\OrCS{E}$ induit par $i_S$.
\end{lemme}

\begin{proof}
On raisonne par récurrence sur le nombre minimum $l_\alpha$ d'occurrences de générateurs dans une expression de $\alpha$ comme composé de générateurs (ou d'identités itérées de générateurs). Si $l_\alpha=1$, alors $\alpha$ est un générateur ou une identité itérée d'un générateur et si ce générateur est $\atom{i_0i_1\cdots i_p}$, alors il résulte du lemme~\ref{lemmeaprouver} et de la description du foncteur $\Stnu$, donnée dans le paragraphe~\ref{defStnu}, que $S=\{i_0,i_1,\dots,i_p\}$ et $S'=\{i_0,i_p\}$. On a donc $S'\subset S$ et $\atom{i_0i_1\cdots i_p}$ est un atome de~$\OrCS{S}$, et par suite $\alpha$ est l'image d'une $n$\nbd-flèche de $\OrCS{S}$, ce qui prouve dans ce cas les deux assertions. Si $l_\alpha>1$, alors $\alpha$ se décompose en \smash{$\alpha=\alpha_2\comp{j}\alpha_1$}, $0\leq j<n$, où $\alpha_1,\alpha_2$ sont des $n$\nbd-flèches $j$\nbd-composables de $\OrCS{E}$ telles que $l_{\alpha_1},l_{\alpha_2}<l_\alpha$. On distingue plusieurs cas.
\smallbreak

1) $j\geq2$. Alors on a $S=t_1(\alpha_1)=t_1(\alpha_2)$, $S'=s_1(\alpha_1)=s_1(\alpha_2)$, et par l'hypothèse de récurrence on a $S'\subset S$, et $\alpha_1$ et $\alpha_2$ sont des images de $n$\nbd-flèches de $\OrCS{S}$ qui sont, en vertu du lemme~\ref{Ormono}, $j$\nbd-composables. La $n$\nbd-flèche $\alpha$ est l'image de leur composé.
\smallbreak

2) $j=1$. Alors on a $S'=s_1(\alpha_1)$, $S=t_1(\alpha_2)$ et on pose $S''=s_1(\alpha_2)=t_1(\alpha_1)$. Par l'hypothèse de récurrence, on a $S'\subset S''\subset S$ et $\alpha_2$ (resp. $\alpha_1$) est l'image, par le $\infty$\nbd-foncteur induit par l'inclusion $S\subset E$ (resp. $S''\subset E$), d'une $n$\nbd-flèche de $\OrCS{S}$ (resp. de $\OrCS{S''}$). Comme on a un triangle commutatif de $\infty$\nbd-foncteurs
$$
\xymatrixrowsep{.5pc}
\xymatrix{
\OrCS{S''}\ar[dd]\ar[dr]
\\
&\OrCS{E}
\\
\OrCS{S}\ar[ur]
}$$
induit par les inclusions, $\alpha_1$ est aussi l'image d'une $n$\nbd-flèche de $\OrCS{S}$. En vertu du lemme~\ref{Ormono}, ces deux $n$\nbd-flèches de $\OrCS{S}$ sont $1$\nbd-composables, et par suite $\alpha$ est l'image de leur composé.
\smallbreak

3) $j=0$. Alors on a \smash{$S'=s_1(\alpha_2)\comp{0}s_1(\alpha_1)$ et $S=t_1(\alpha_2)\comp{0}t_1(\alpha_1)$}, autrement dit, en posant 
$$S'_1=s_1(\alpha_1)\ ,\quad S'_2=s_1(\alpha_2)\ ,\quad S_1=t_1(\alpha_1)\ ,\quad S_2=t_1(\alpha_2)\ ,$$
on a $S'=S'_2\cup S'_1$ et $S=S_2\cup S_1$. Par l'hypothèse de récurrence, on a $S'_1\subset S_1$ et $S'_2\subset S_2$, d'où $S'\subset S$, et $\alpha_2$ (resp. $\alpha_1$) est l'image, par le $\infty$\nbd-foncteur induit par l'inclusion $S_2\subset E$ (resp. $S_1\subset E$), d'une $n$\nbd-flèche de $\OrCS{S_2}$ (resp. de $\OrCS{S_1}$). Comme on a des triangles commutatifs de $\infty$\nbd-foncteurs
$$
\xymatrixrowsep{.5pc}
\xymatrix{
\OrCS{S_2}\ar[dd]\ar[dr]
\\
&\OrCS{E}
&\kern -20pt,\kern 20pt
\\
\OrCS{S}\ar[ur]
}\qquad\quad
\xymatrixrowsep{.5pc}
\xymatrix{
\OrCS{S_1}\ar[dd]\ar[dr]
\\
&\OrCS{E}
\\
\OrCS{S}\ar[ur]
}$$
induits par les inclusions, $\alpha_1$ et $\alpha_2$ sont aussi des images de $n$\nbd-flèches de $\OrCS{S}$. En vertu du lemme~\ref{Ormono}, ces deux $n$\nbd-flèches de $\OrCS{S}$ sont $0$\nbd-composables, et par suite $\alpha$ est l'image de leur composé.
\end{proof}

\begin{prop}\label{incliso}
Soient $E$ un ensemble ordonné, $S,S'\in\xi E$ deux $1$\nbd-flèches parallèles de $\OrCS{E}$ et $i_S:S\to E$ l'inclusion. Alors le $\infty$\nbd-foncteur
$$\sHom_{\OrCS{S}}(S',S)\toto\sHom_{\OrCS{E}}(S',S)\ ,$$
induit par $i_S$, est un isomorphisme.
\end{prop}

\begin{proof}
La proposition est conséquence des lemmes~\ref{Ormono} et~\ref{surjcel}.
\end{proof}

\begin{cor}\label{locpseudoin}
\emph{a)} Pour tout objet $i_0$ de $\OrCS{E}$, la $\infty$\nbd-catégorie $\sHom_{\OrCS{E}}(i_0,i_0)$ est un objet final de $\ooCat$.
\smallbreak
\emph{b)} Pour tous $i_0,i_1\in E$ tels que $i_0<i_1$, et toute $1$\nbd-flèche $S\in\xi E$ de $\OrCS{E}$ de source~$i_0$ et but $i_1$, si l'on pose $S_0=\{i_0,i_1\}$ et $n=\card(S)-3$, alors 

\indent\indent \kern 2.8pt\emph{i)} si $S=S_0$, $\sHom_{\OrCS{E}}(S_0,S)$ est un objet final de $\ooCat$;

\indent\indent \emph{ii)} si $S\neq S_0$, $\sHom_{\OrCS{E}}(S_0,S)$ est une $n$\nbd-catégorie admettant à la fois un objet \phantom{a\kern 29pt} quasi-initial et un objet quasi-final.
\end{cor}

\begin{proof}
Pour montrer (a), on remarque que le seul objet de $\sHom_{\OrCS{E}}(i_0,i_0)$ est la $1$\nbd-flèche $S=\{i_0\}$ de $\OrCS{E}$, identité de l'objet $i_0$. Or, en vertu de la proposition précédente, la $\infty$\nbd-catégorie $\sHom_{\OrCS{E}}(S,S)$ est isomorphe à $\sHom_{\OrCS{S}}(S,S)$. Comme $\OrCS{S}$ est isomorphe à l'oriental $\Or{0}\simeq\smp{0}$, cela prouve l'assertion. De même, sous les hypothèses et notations de (b), il résulte de la proposition précédente que la $\infty$\nbd-catégorie $\sHom_{\OrCS{E}}(S_0,S)$ est isomorphe à $\sHom_{\OrCS{S}}(S_0,S)$. Le cas (i) est trivial puisque si $S=S_0$, alors $\OrCS{S}$ est isomorphe à $\Or{1}\simeq\smp{1}$.  Si $S\neq S_0$, on a $n\geq0$ et $\sHom_{\OrCS{S}}(S_0,S)$ est isomorphe à la $n$\nbd-catégorie
$$\sHom_{\Or{n+2}}\Biggl(
\begin{pmatrix}
(0)
&(0,n+2)
\cr
(n+2)
&(0,n+2)
\end{pmatrix}
\,,\,
\begin{pmatrix}
(0)
&(0,1)+(1,2)+\cdots+(n+1,n+2)
\cr
(n+2)
&(0,1)+(1,2)+\cdots+(n+1,n+2)
\end{pmatrix}
\Biggr)\ $$
qui admet bien, en vertu du théorème~\ref{bigchief}, à la fois un objet quasi-initial et un objet quasi-final, ce qui prouve le cas (ii).
\end{proof}

\begin{lemme}
Le $1$\nbd-tronqué intelligent de $\OrCS{E}$ s'identifie à $E$, et par cette identification le morphisme d'adjonction
$$\OrCS{E}\toto i^{}_1\ti{1}(\OrCS{E})\simeq E$$
s'identifie à l'unique $\infty$\nbd-foncteur $\OrCS{E}\to E$ induisant l'identité sur les ensembles des objets.
\end{lemme}

\begin{proof}
Vu la description du $1$\nbd-tronqué bête de $\OrCS{E}$ présentée au paragraphe~\ref{untronqueor}, pour montrer la première assertion, il suffit de montrer que si $i_0,i_1\in E$ sont deux objets de $\OrCS{E}$ et $S_1,S_2\in\xi E$  deux $1$\nbd-flèches de $\OrCS{E}$ de source $i_0$ et de but $i_1$, autrement dit telles que $\min(S_1)=\min(S_2)=i_0$ et $\max(S_1)=\max(S_2)=i_1$, alors $S_1$ et $S_2$ sont reliées par un zigzag de $2$\nbd-flèches. Or, $S=\{i_0,i_1\}$ est une $1$\nbd-flèche de $\OrCS{E}$ de source $i_0$ et de but $i_1$, et $S\subset S_1$ et $S\subset S_2$. En vertu du lemme~\ref{cnsexfl}, il existe donc deux $2$\nbd-flèches de $\OrCS{E}$ de source $S$ et de but respectivement $S_1$ et $S_2$, ce qui prouve l'assertion. La deuxième assertion résulte du fait que les morphismes d'adjonction de la 1-troncation intelligente induisent l'identité sur les objets et du fait qu'un $\infty$\nbd-foncteur de but un ensemble ordonné est déterminé par sa restriction aux objets.
\end{proof}

\begin{thm}
Pour tout ensemble ordonné $E$, l'unique $\infty$\nbd-foncteur $\OrCS{E}\to E$ induisant l'identité sur les objets est une équivalence faible de $\ooCat$. De plus, pour tout entier $n\geq1$, le $n$\nbd-tronqué intelligent $\ti{n}\OrCS{E}\to E$ de ce $\infty$\nbd-foncteur est une équivalence faible de $\nCat{n}$.
\end{thm}

\begin{proof}
Pour présenter une preuve uniforme des deux assertions, on notera $\ti{\infty}$ l'endofoncteur identité de $\ooCat$. Il s'agit donc de montrer que pour tout $n$, $1\leq n\leq\infty$, l'unique $\infty$\nbd-foncteur $\ti{n}\OrCS{E}\to E$ induisant l'identité sur les objets est une équivalence faible. Si $n=1$, il résulte du lemme précédent qu'il s'agit d'un isomorphisme. On peut donc supposer que $n\geq2$. En vertu du théorème~\ref{eqfDwyerKaneqf}, il suffit de montrer que pour tous $i_0,i_1\in E$, le $\infty$\nbd-foncteur induit
$$\sHom_{\ti{n}\OrCS{E}}(i_0,i_1)\toto\sHom_E(i_0,i_1)$$
est une équivalence faible. Si $i_0\nleqslant i_1$, alors la source et le but de ce $\infty$\nbd-foncteur sont vides, et il n'y a rien à démontrer. Supposons donc que $i_0\leq i_1$. Comme $E$ est un ensemble ordonné, la $\infty$\nbd-catégorie $\sHom_E(i_0,i_1)$ est un objet final de $\ooCat$. Il s'agit donc de prouver que $\sHom_{\ti{n}\OrCS{E}}(i_0,i_1)$ est asphérique. Si $i_1=i_0$, en vertu du corollaire~\ref{locpseudoin}, (a), cette $\infty$\nbd-catégorie est un objet final de $\ooCat$, ce qui prouve dans ce cas l'assertion. On peut donc supposer que $i_0<i_1$. Comme $n\geq2$, les $1$\nbd-flèches de $\ti{n}\OrCS{E}$ sont les mêmes que celles de $\OrCS{E}$, autrement dit l'ensemble des objets de $\sHom_{\ti{n}\OrCS{E}}(i_0,i_1)$ s'identifie à l'ensemble des $S\in\xi E$ tels que $\min(S)=i_0$ et $\max(S)=i_1$. En particulier, $S_0=\{i_0,i_1\}$ est un objet de $\sHom_{\ti{n}\OrCS{E}}(i_0,i_1)$. En vertu du théorème~\ref{consthA}, il suffit de montrer que pour tout objet $S$ de $\sHom_{\ti{n}\OrCS{E}}(i_0,i_1)$, la $\infty$\nbd-catégorie
$$\sHom_{\sHom_{\ti{n}\OrCS{E}}(i_0,i_1)}(S_0,S)=\sHom_{\ti{n}\OrCS{E}}(S_0,S)\simeq\ti{n-2}\bigl(\sHom_{\OrCS{E}}(S_0,S)\bigr)$$
est asphérique. Or, en vertu du corollaire~\ref{locpseudoin}, (b), si $S=S_0$, la $\infty$\nbd-catégorie $\sHom_{\OrCS{E}}(S_0,S)$ est un objet final de $\ooCat$ et si $S\neq S_0$, cette $\infty$\nbd-catégorie est une $(\card(S)-3)$\nbd-catégorie admettant un objet quasi-initial, et l'assertion résulte de la proposition~\ref{tronqpseudofinal} et du corollaire~\ref{pseudoasph}.
\end{proof}

\begin{cor}
Pour tout ensemble ordonné $E$, et tout $n$, $1\leq n\leq\infty$, le morphisme d'adjonction $\cStr{n}\NStr{n}(E)\to E$ est une équivalence faible de $\nCat{n}$.
\end{cor}

\begin{proof}
Pour tout entier $n\geq1$, le morphisme d'adjonction $\cStr{n}\NStr{n}(E)\to E$ est le $n$\nbd-tronqué intelligent du morphisme d'adjonction $\cStr{\infty}\NStr{\infty}(E)\to E$. Or, ce dernier induit l'identité sur les objets (cf.~\ref{objnerfadj}) et en vertu du théorème~\ref{OrisoNkappa}, on a un isomorphisme $\cStr{\infty}\NStr{\infty}(E)=\cStr{\infty}\kappa^*(E,\xi E)\simeq\OrCS{E,\xi E}=\OrCS{E}$. On en déduit que le morphisme d'adjonction $\cStr{\infty}\NStr{\infty}(E)\to E$ s'identifie à l'unique $\infty$\nbd-foncteur $\OrCS{E}\to E$ induisant l'identité sur les objets. L'assertion résulte donc du théorème précédent.
\end{proof}

\begin{cor}\label{corconde}
Pour tout ensemble ordonné $E$, et tout $n$, $1\leq n\leq\infty$, le morphisme d'adjonction $\NStr{n}(E)\to\NStr{n}\cStr{n}\NStr{n}(E)$ est une équivalence faible de $\simpl$.
\end{cor}

\begin{proof}
L'assertion résulte aussitôt du corollaire précédent et de l'égalité du triangle
$$
\xymatrix{
\NStr{n}(E)\ar[r]\ar@/_2em/[rr]_{1_{\NStr{n}(E)}}
&\NStr{n}\cStr{n}\NStr{n}(E)\ar[r]
&\NStr{n}(E)\ \ ,
}
$$
pour les morphismes d'adjonction du couple de foncteurs adjoints $(\cStr{n},\NStr{n})$.
\end{proof}

\begin{sch}
Comme pour tout ensemble ordonné $E$ et pour tout $n$, \hbox{$1\leq n\leq\infty$,} l'ensemble simplicial $\NStr{n}(E)$ n'est autre que le nerf ordinaire de $E$, le corollaire ci-dessus prouve la condition (e) de~\cite[scholie 5.14]{DG}. Ainsi, en vertu de \emph{loc.~cit.}, pour avoir une structure de catégorie de modèles de Quillen à la Thomason sur $\nCat{n}$, il suffit de démontrer la condition (d$'$) de \emph{loc.~cit.}:
\begin{enumerate}
\item[(d${}'$)] \emph{si $E'\to E$ est un crible de $\ord$ admettant une rétraction qui est aussi un adjoint à droite, son image par le foncteur $\cStr{n}\NStr{n}$ est une équivalence faible de $\nCat{n}$ et le reste après tout cochangement de base.}
\end{enumerate}
\end{sch}

\begin{cor}
Le foncteur $\loc{\NerfADC}:\WADC^{-1}\ADC\to\Wsimpl^{-1}\simpl\simeq\Hot$ \emph{(cf.~\ref{eqfADC})} est essentiellement surjectif.
\end{cor}

\begin{proof}
On rappelle que le foncteur nerf usuel $\NStr{1}:\nCat{1}=\cat\to\simpl$ (cf.~\ref{nnerf}) définit une équivalence des catégories homotopiques $\W_\cat^{-1}\cat\to\Wsimpl^{-1}\simpl$, où $\W_\cat=\WnCat{1}=\NStr{1}^{-1}(\Wsimpl)$~\cite[chapitre VI, corollaire 3.3.1]{IL}. Par ailleurs, l'inclusion $\ord\to\cat$ induit aussi une équivalence des catégories homotopiques $\W_\ord^{-1}\ord\to\W_\cat^{-1}\cat$, où $\W_\ord=\W_\cat\cap\fl{}\ord$ (voir par exemple~\cite[théorème~41]{delHoyo}). Pour montrer que le foncteur $\loc{\NerfADC}:\WADC^{-1}\ADC\to\Wsimpl^{-1}\simpl$ est essentiellement surjectif, il suffit donc de montrer que pour tout ensemble ordonné $E$, l'ensemble simplicial $\NStr{1}(E)=\NStreet(E)$ appartient à son image essentiel. Or, on a une équivalence faible $\NStreet(E)\to\NStreet\cStreet\NStreet(E)$ (corollaire~\ref{corconde}) et des isomorphismes d'ensembles simpliciaux
$$
\NStreet\cStreet\NStreet(E)=\NStreet\cStreet\NStr{1}(E)=\NStreet\cStreet\kappa^*(E,\xi E)\simeq\NStreet\Stnu\Chnorm{\kappa^*(E,\xi E)}\simeq\NerfADC\Chnorm{\kappa^*(E,\xi E)}
$$
(le premier en vertu du théorème~\ref{OrisoNkappa} et le deuxième en vertu de la proposition~\ref{commnerf}), ce qui prouve l'assertion.
\end{proof}

\begin{rem}
Vu que pour tout ensemble ordonné $E$ et pour $1\leq n\leq\infty$, on a l'égalité \hbox{$\NStr{n}(E)=\kappa^*(E,\xi E)$}, le corollaire~\ref{corconde} signifie exactement que le morphisme d'adjonction $\kappa^*(E,\xi E)\to\NStr{n}\cStr{n}(\kappa^*(E,\xi E))$ est une équivalence faible simpliciale et que, en particulier, le type d'homotopie de la $n$\nbd-catégorie $\cStr{n}(\kappa^*(E,\xi E))$ est le même que celui du complexe simplicial $(E,\xi E)$. 
On peut se demander si cette propriété reste vraie pour un complexe simplicial général. Pour $n$ fini, ce n'est pas le cas. En effet, il est facile de montrer que pour tout ensemble simplicial $X$, la $n$\nbd-catégorie $\cStr{n}(X)$ ne dépend que du $(n+1)$\nbd-squelette de $X$. Ainsi, pour tout $p>n+1$, si l'on considère le bord $\partial\smp{p}$ du simplexe standard $\smp{p}$,
$$\partial\smp{p}=\kappa^*(\smp{p},\varPhi_p)\ ,\quad\hbox{où}\quad\varPhi_p=\{S\in\xi\smp{p}\,|\,S\neq\smp{p}\}\ ,$$
la $n$\nbd-catégorie $\cStr{n}(\kappa^*(\smp{p},\varPhi_p))$ est isomorphe à $\cStr{n}(\kappa^*(\smp{p},\xi\smp{p}))$ qui a, en vertu du corollaire~\ref{corconde}, le type d'homotopie du point, tandis que le complexe simplicial $(\smp{p},\varPhi_p)$ a le type d'homotopie du bord de $\smp{p}$, autrement dit de la sphère $S^{p-1}$. Néanmoins, pour $n=\infty$, on propose la conjecture suivante:
\end{rem}

\begin{conj}
Pour tout complexe simplicial $(E,\varPhi)$, le morphisme d'adjonction $\kappa^*(E,\varPhi)\to\NStreet\cStreet(\kappa^*(E,\varPhi))$ est une équivalence faible simpliciale. En particulier, le type d'homotopie de la $\infty$\nbd-catégorie $\cStreet(\kappa^*(E,\varPhi))$ est le même que celui du complexe simplicial $(E,\varPhi)$.
\end{conj}

\appendix

\section{La description du $2$-tronqué intelligent de $\OrCS{E}$}\label{AppA}

Le but de cette section est de donner une description explicite du $2$\nbd-tronqué intelligent de $\OrCS{E}$, pour un ensemble ordonné $E$.

\begin{paragr}
Dans la suite, on se fixe un entier $n$, et on note $B$ la base du complexe dirigé augmenté $\Chnorm{\smp{n}}$ (cf.~\ref{exADCstandard}),
$$B=\{(j_0\cdots j_p)\,|\,0\leq j_0<\cdots<j_p\leq n\,,\ 0\leq p\leq n\}\ .$$
On aura besoin des lemmes techniques suivants.
\end{paragr}

\begin{lemme}\label{avantdecasser}
Soient $k$ et $p$ des entiers tels que $0<k<n$, $p\geq2$,
$$\textstyle a=\sum\limits_{(i_0\cdots i_p)\in B_p}a_{i_0\cdots i_p}(i_0\cdots i_p)$$
un élément de $(\Chnorm{\smp{n}})^*_p$, et
$$\textstyle\Neg{(d_pa)}=\sum\limits_{(j_0\cdots j_{p-1})\in B_{p-1}}b_{j_0\cdots j_{p-1}}(j_0\cdots j_{p-1})\ $$
\emph{(pour la notation $\Neg{x}$, cf.~\ref{baseunitADC}).} S'il existe $(i_0\cdots i_p)\in B_p$ tel que $i_0<k<i_p$ et $a_{i_0\cdots i_p}\neq0$, alors il existe $(j_0\cdots j_{p-1})\in B_{p-1}$ tel que $j_0<k<j_{p-1}$ et $b_{j_0\cdots j_{p-1}}\neq0$.
\end{lemme}

\begin{proof}
On définit les ensembles
$$\begin{aligned}
&A=\{(i_0\cdots i_p)\in B_p\,|\,i_0<k<i_p\quad\hbox{et}\quad a_{i_0\cdots i_p}\neq0\}\ ,\cr
&A'=\{(i_0\cdots i_p)\in A\,|\,i_p-i_0=m_0\}\ ,\quad\hbox{où}\quad m_0=\max\{i_p-i_0\,|\,(i_0\cdots i_p)\in A\}\ ,\cr
&A''=\{(i_0\cdots i_p)\in A'\,|\,i_2-i_0=m_1\}\ ,\quad\hbox{où}\quad m_1=\max\{i_2-i_0\,|\,(i_0\cdots i_p)\in A'\}\ .
\end{aligned}$$
S'il existe $(i_0\cdots i_p)\in B_p$ tel que $i_0<k<i_p$ et $a_{i_0\cdots i_p}\neq0$, alors ces ensembles sont non vides. Soit $(i_0\cdots i_p)\in A''$. On va prouver que $b_{i_0i_2\cdots i_p}\neq0$, ce qui impliquera le lemme. Pour cela, il suffit de montrer que pour tout $(j_0\cdots j_p)\in A$ et tout $l$ \emph{pair}, $0\leq l\leq p$, on a $(j_0\cdots\widehat{j}_l\cdots j_p)\neq(i_0i_2\cdots i_p)$. Si $l=0$ ou $l=p$, cela résulte de la maximalité de $i_p-i_0$. Si $0<l<p$, pour que l'on ait $(j_0\cdots\widehat{j}_l\cdots j_p)=(i_0i_2\cdots i_p)$, il faudrait déjà que $j_0=i_0$, $j_1=i_2$ et $j_p=i_p$ et en particulier que $(j_0\cdots j_p)$ soit dans $A'$. Comme $(i_0\cdots i_p)$ est dans $A''$, on aurait donc $i_2-i_0\geq j_2-j_0$. Or, $i_2-i_0=j_1-j_0<j_2-j_0$, ce qui est absurde.
\end{proof}

\begin{lemme}\label{oncommenceacasser}
Soient $m\ge1$ un entier, $i_0,i_1\dots,i_m$ des entiers tels que 
$$0=i_0<i_1<\cdots<i_{m-1}<i_m=n\ ,$$ 
et
$$\cellule{x}=\tablg{x}{q}\,,\qquad q\geq2\ ,$$
une $q$\nbd-cellule de $\Or{n}$ telle que $x^0_0=(0)$, $x^1_0=(n)$,
$$\textstyle x^0_1=\sum\limits^m_{k=1}(i_{k-1},i_k)\qquad\hbox{et}\qquad x^1_1=\sum\limits^n_{l=1}(l-1,l)\ .$$
Pour tous $p,\e$ tels que $0\leq p\leq q$, $\e\in\{0,1\}$, si
$$\textstyle x^\e_p=\sum\limits_{(j_0\dots j_p)\in B_p}(x^\e_p)_{j_0\dots j_p}{(j_0\dots j_p)}$$
et si $(x^\e_p)_{j_0\dots j_p}\neq0$, alors il existe $k$, $1\leq k\leq m$, tel que
$$i_{k-1}\leq j_0<j_1<\cdots<j_p\leq i_k\ .$$
\end{lemme}

\begin{proof}
On raisonne par l'absurde. Supposons qu'il existe $p,\e$, $0\leq p\leq q$, $\e\in\{0,1\}$, tels que $x^\e_p$ ne satisfasse pas à la conclusion du lemme, et choisissons un tel couple avec $p$ minimal. Alors on a $p\geq2$, et il existe $(j_0\dots j_p)\in B_p$ tel que $(x^\e_p)_{j_0\dots j_p}\neq0$ et tel qu'il existe $k$, $0<k<m$, tel que $j_0<i_k<j_p$. En vertu du lemme précédent, si 
$$\textstyle\Neg{(d_px^\e_p)}=\sum\limits_{(j'_0\cdots j'_{p-1})\in B_{p-1}}y^{}_{j'_0\cdots j'_{p-1}}(j'_0\cdots j'_{p-1})\ ,$$
alors il existe $(j'_0\cdots j'_{p-1})\in B_{p-1}$ tel que $y^{}_{j'_0\cdots j'_{p-1}}\neq0$ et $j'_0<i_k<j'_{p-1}$. Or, on a $d_px^\e_p=x^1_{p-1}-x^0_{p-1}$ et $x^1_{p-1},x^0_{p-1}\in(\Chnorm{\smp{n}})^*_{p-1}$, ce qui implique que $\Neg{(d_px^\e_p)}\leq x^0_{p-1}$, et par suite que $(x^0_{p-1})_{j'_0\cdots j'_{p-1}}\neq0$, ce qui contredit la minimalité de $p$, et prouve le lemme.
\end{proof}

\begin{prop}\label{oncasse} 
Soient $m\ge1$ un entier, et $i_0,i_1\dots,i_m$ des entiers tels que \hbox{$0=i_0<i_1<\cdots<i_{m-1}<i_m=n$.} Alors le $\infty$\nbd-foncteur
$$\textstyle\prod\limits^m_{k=1}\sHom_{\Or{n}}(a_k,b_k)\toto\sHom_{\Or{n}}(a,b)\ ,$$
où
$$a_k=
\begin{pmatrix}
(i_{k-1})
&(i_{k-1},i_k)
\cr
\noalign{\vskip 3pt}
(i_k)
&(i_{k-1},i_k)
\end{pmatrix}\,,\quad
b_k=
\begin{pmatrix}
(i_{k-1})
&\textstyle\sum\limits_{i_{k-1}<l\leq i_k}(l-1,l)
\cr
\noalign{\vskip 3pt}
(i_k)
&\sum\limits_{i_{k-1}<l\leq i_k}(l-1,l)
\end{pmatrix}\,,\ \ 1\leq k\leq m\ ,
$$
$$\kern -75pt
a=
\begin{pmatrix}
(0)
&\textstyle\sum\limits^m_{k=1}\kern -3pt(i_{k-1},i_k)
\cr
\noalign{\vskip 3pt}
(n)
&\sum\limits^m_{k=1}\kern -3pt(i_{k-1},i_k)
\end{pmatrix}\,,\quad
b=
\begin{pmatrix}
(0)
&\sum\limits^n_{l=1}(l-1,l)
\cr
\noalign{\vskip 3pt}
(n)
&\sum\limits^n_{l=1}(l-1,l)
\end{pmatrix}\,,
$$
défini par la composition \og horizontale\fg{} $\comp{0}$ de $\Or{n}$
$$(x_1,x_2,\dots,x_m)\longmapsto x_1\comp{0}x_2\comp{0}\cdots\comp{0}x_m$$
est un isomorphisme de $\infty$\nbd-catégories. En particulier, $\sHom_{\Or{n}}(a,b)$ admet à la fois un objet quasi-initial et un objet quasi-final.
\end{prop}

\begin{proof}
Il s'agit de montrer que pour tout $q\geq2$, et toute $q$\nbd-cellule 
$$\cellule{x}=\tablg{x}{q}\,$$
de $\Or{n}$ telle que 
$$x^0_0=(0)\ ,\quad x^1_0=(n)\ ,\quad\textstyle x^0_1=\sum\limits^m_{k=1}(i_{k-1},i_k)\ ,\quad x^1_1=\sum\limits^n_{l=1}(l-1,l)\ ,$$
il existe une famille unique $(x_k)_{1\leq k\leq m}$, où
$$x_k=\tablg{(x_k)}{q}\ ,\qquad1\leq k\leq m\ ,$$
est une $q$\nbd-cellule de $\Or{n}$ telle que
$$(x_k)^0_0=(i_{k-1})\,,\quad (x_k)^1_0=(i_{k})\,,\quad (x_k)^0_1=(i_{k-1},i_k)\,,\quad (x_k)^1_1=\kern -6pt\textstyle\sum\limits_{i_{k-1}<l\leq i_k}\kern -5pt(l-1,l)\,,$$
satisfaisant $x=x_1\comp{0}\dots\comp{0}x_m$, autrement dit telle que
$$x^\e_p=\textstyle\sum\limits^m_{k=1}(x_k)^\e_p\ ,\qquad p\geq2\ ,\ \ \e\in\{0,1\}\ .$$
L'unicité résulte du fait que $B$ est une base, en remarquant qu'en vertu de la proposition~\ref{incliso}, les $(x_k)^\e_p$ sont forcément de la forme
$$(x_k)^\e_p=\textstyle\sum\limits_{i_{k-1}\leq j_0<\dots< j_p\leq i_k}((x_k)^\e_p)_{j_0\dots j_p}{(j_0\dots j_p)}\ .$$
Montrons l'existence. En vertu du lemme précédent, pour tout $p\geq2$ et $\e\in\{0,1\}$, $x^\e_p$~est de la forme
$$x^\e_p=\textstyle\sum\limits^m_{k=1}\sum\limits_{\kern 5pt i_{k-1}\leq j_0<\dots< j_p\leq i_k}(x^\e_p)_{j_0\dots j_p}{(j_0\dots j_p)}\ ,$$
et par hypothèse, il en est de même pour $p=1$. Définissons $x_k$ par 
$$\begin{aligned}
&(x_k)^0_0=(i_{k-1})\ ,\qquad\kern 12.5pt (x_k)^1_0=(i_{k})\ ,\cr 
\noalign{\vskip 4pt}
&(x_k)^0_1=(i_{k-1},i_k)\ ,\qquad (x_k)^1_1=\kern -6pt\textstyle\sum\limits_{i_{k-1}<l\leq i_k}\kern -5pt(l-1,l)\ ,\cr
&(x_k)^\e_p=\kern -1.5pt\textstyle\sum\limits_{i_{k-1}\leq j_0<\dots< j_p\leq i_k}\kern -2pt(x^\e_p)_{j_0\dots j_p}{(j_0\dots j_p)}\ ,\qquad p\geq2\ ,\ \e\in\{0,1\}\ ,
\end{aligned}$$
(la formule de la troisième ligne étant également vraie pour $p=1$).
Le tableau ainsi défini est une $p$\nbd-cellule de $\Or{n}$. En effet, on remarque qu'il suffit de montrer que pour $p\geq1$ et $\e\in\{0,1\}$, on a $d_p((x_k)^\e_p)=(x_k)^1_{p-1}-(x_k)^0_{p-1}$. Or, cela est évident pour $p=1$, et pour $p\geq2$, on a
$$\begin{aligned}
d_p(x^\e_p)&=x^1_{p-1}-x^0_{p-1}\cr
\noalign{\vskip 3pt}
&=\textstyle\sum\limits^m_{k=1}\sum\limits_{\kern 5pt i_{k-1}\leq j_0<\dots< j_{p-1}\leq i_k}\bigl((x^1_{p-1})_{j_0\dots j_{p-1}}-(x^0_{p-1})_{j_0\dots j_{p-1}}\bigr){(j_0\dots j_{p-1})}
\end{aligned}$$
et
$$\begin{aligned}
d_p(x^\e_p)&=d_p\Bigl(\textstyle\sum\limits^m_{k=1}\sum\limits_{\kern 5pt i_{k-1}\leq j_0<\dots< j_p\leq i_k}(x^\e_p)_{j_0\dots j_p}{(j_0\dots j_p)}\Bigr)\cr
\noalign{\vskip 3pt}
&=\textstyle\sum\limits^m_{k=1}\sum\limits_{\kern 5pt i_{k-1}\leq j_0<\dots< j_p\leq i_k}\kern 3pt\sum\limits^p_{\kern 2ptl=0}(-1)^l\,(x^\e_p)_{j_0\dots j_p}\,{(j_0\dots\widehat{j_l}\dots j_p)}\ ,\kern 43pt
\end{aligned}$$
et comme $B$ est une base, on en déduit que pour tout $k$ tel que $1\leq k\leq m$, on a
$$\begin{aligned}
\textstyle\sum\limits_{\kern 5pt i_{k-1}\leq j_0<\dots< j_{p-1}\leq i_k}&\bigl((x^1_{p-1})_{j_0\dots j_{p-1}}-(x^0_{p-1})_{j_0\dots j_{p-1}}\bigr){(j_0\dots j_{p-1})}\cr
&=\sum\limits_{\kern 5pt i_{k-1}\leq j_0<\dots< j_p\leq i_k}\kern 3pt\sum\limits^p_{\kern 2ptl=0}(-1)^l\,(x^\e_p)_{j_0\dots j_p}\,{(j_0\dots\widehat{j_l}\dots j_p)}\cr
\noalign{\vskip 5pt}
&=d_p\Bigl(\sum\limits_{\kern 5pt i_{k-1}\leq j_0<\dots< j_p\leq i_k}(x^\e_p)_{j_0\dots j_p}{(j_0\dots j_p)}\Bigr)\ ,
\end{aligned}$$
c'est-à-dire $(x_k)^1_{p-1}-(x_k)^0_{p-1}=d_p((x_k)^\e_p)$, ce qui prouve l'assertion.
\smallbreak

Pour montrer la dernière assertion, on remarque qu'en vertu de la proposition~\ref{incliso}, pour tout $k$, $1\leq k\leq m$, on a un isomorphisme $\sHom_{\Or{n}}(a_k,b_k)\simeq\sHom_{\Or{n_k}}(a'_k,b'_k)$, où $n_k=i_k-i_{k-1}$,
$$a'_k=\begin{pmatrix}
(0)
&(0,n_k)
\cr
(n_k)
&(0,n_k)
\end{pmatrix}
\qquad\hbox{et}\qquad b'_k=
\begin{pmatrix}
(0)
&(0,1)+(1,2)+\cdots+(n_k-1,n_k)
\cr
(n_k)
&(0,1)+(1,2)+\cdots+(n_k-1,n_k)
\end{pmatrix}\ .
$$
L'assertion résulte donc du théorème~\ref{bigchief} (cf.~\ref{prodpseudo}).
\end{proof}

Le théorème suivant généralise la partie (b) du corollaire~\ref{locpseudoin}.

\begin{thm}\label{thconjor1}
Soient $E$ un ensemble ordonné, et $S,S'\in\xi E$ deux $1$\nbd-flèches parallèles de $\OrCS{E}$ \emph{(cf.~\ref{untronqueor})}. Alors la $\infty$\nbd-catégorie $\sHom_{\OrCS{E}}(S',S)$ est
\begin{itemize}
\item vide si $S'\not\subset S$;
\item un objet final si $S'=S$;
\item une $n$\nbd-catégorie admettant à la fois un objet quasi-initial et un objet quasi-final, où $n=\card(S)-3$, si $S'\subsetneqq S$.
\end{itemize}
\end{thm}

\begin{proof}
L'assertion relative au cas $S'\not\subset S$ résulte du lemme~\ref{cnsexfl}. En vertu de la proposition~\ref{incliso}, dans le cas où $S'\subset S$, le théorème résulte aussitôt de la proposition précédente.
\end{proof}

\begin{cor}
Pour tout ensemble ordonné $E$, le $2$\nbd-tronqué intelligent $\ti{2}(\OrCS{E})$ de $\OrCS{E}$ est la $2$\nbd-catégorie dont les objets sont les éléments de $E$, et pour tout couple $i_0,i_1$ d'objets, $\sHom_{\ti{2}(\OrCS{E})}(i_0,i_1)$ est l'ensemble ordonné par inclusion des parties finies totalement ordonnées non vides $S$ de $E$ telles que $\min(S)=i_0$ et $\max(S)=i_1$, la composition de deux $1$\nbd-flèches composables étant leur réunion.
\end{cor}

\begin{proof}
Vu la description du $1$\nbd-tronqué bête de $\OrCS{E}$ donnée dans le paragraphe~\ref{untronqueor}, le corollaire est conséquence immédiate du théorème précédent.
\end{proof}

\section{Transformations et contractions de \pdfinfty-catégories}\label{AppB}

La notion de $n$\nbd-catégorie, pour $0\leq n<\infty$, admettant un objet quasi-final (resp.~quasi-initial), introduite dans le paragraphe~\ref{pseudofinal}, s'étend sans difficulté au cadre des $\infty$\nbd-catégories. On peut définir cette notion par \emph{coïnduction} en disant qu'une $\infty$\nbd-catégorie $C$ admet un objet quasi-final (resp.~quasi-initial) s'il existe un objet $x$ de~$C$ tel que pour tout objet $y$ de~$C$, la $\infty$\nbd-catégorie $\sHom_C(y,x)$ (resp. $\sHom_C(x,y)$) admette un objet quasi-final (resp. quasi-initial). Néanmoins, on adoptera une définition plus classique dans sa forme. Pour tout ensemble $E$ de cellules d'une $\infty$\nbd-catégorie~$C$, et tout entier $i\geq0$, on note $E_i$ l'ensemble des $i$\nbd-cellules de $C$ appartenant à $E$.

\begin{paragr}\label{oopseudofinal}
On dit qu'un ensemble $E$ de cellules d'une $\infty$\nbd-catégorie $C$ est \emph{quasi-final} (resp.~\emph{quasi-initial}) si pour tout $i\geq0$, tout $x\in E_i$, et toute $i$\nbd-cellule $y$ de $C$ parallèle à $x$ (cf.~\ref{defcelparal}), il existe $z\in E_{i+1}$ de source $y$ et but $x$ (resp.~de source $x$ et but $y$). On dit qu'un objet $x$ de $C$ est \emph{quasi-final} (resp.~\emph{quasi-initial}) s'il appartient à un ensemble quasi-final (resp. quasi-initial) de cellules de $C$. Ainsi, pour qu'une $\infty$\nbd-catégorie $C$ admette un objet quasi-final (resp.~quasi-initial), il faut et il suffit qu'elle admette un ensemble quasi-final (resp.~quasi-initial) $E$ de cellules tel que $E_0$ soit non vide. On vérifie aussitôt qu'une $\infty$\nbd-catégorie $C$ admet un objet quasi-final si et seulement si la $\infty$\nbd-catégorie duale $\op{C}$ (cf.~\ref{dualooCat}) admet un objet quasi-initial, et qu'un objet de $C$ est quasi-final si et seulement si il est un objet quasi-initial de $\op{C}$.
\end{paragr}

\begin{prop}\label{carlocpseudo}
Soit $C$ une $\infty$\nbd-catégorie et $x$ un objet de $C$. Pour que $x$ soit un objet quasi-final \emph{(resp.}~quasi-initial\emph{)} de $C$, il faut et il suffit que pour tout objet $y$ de $C$, la $\infty$\nbd-catégorie $\sHom_C(y,x)$ \emph{(resp.} $\sHom_C(x,y)$\emph{)} admette un objet quasi-final \emph{(resp.} quasi-initial\emph{).}
\end{prop}

\begin{proof}
Supposons que $x$ soit un objet quasi-final (resp.~quasi-initial) de~$C$, et soient $E$ un ensemble quasi-final (resp.~quasi-initial) de cellules de $C$ tel que $x\in E_0$, et $y$ un objet arbitraire de $C$. Alors, par définition, il existe une $1$\nbd-cellule~$z$ de~$C$ appartenant à $E$ de source $y$ et but $x$ (resp.~de source $x$ et but $y$). Si pour $i\geq0$, on note $F_i$ l'ensemble des $i$\nbd-cellules de $\sHom_C(y,x)$ (resp.~de $\sHom_C(x,y)$) telles que la $(i+1)$\nbd-cellule correspondante de~$C$ soit dans $E_{i+1}$, alors on vérifie aussitôt que l'ensemble \hbox{$F=\cup_{\,i\geq0}F_i$} est un ensemble quasi-final (resp.~quasi-initial) de cellules de $\sHom_C(y,x)$ (resp.~de $\sHom_C(x,y)$). Comme $z$ appartient à $F_0$, on en déduit qu'il est un objet quasi-final (resp. quasi-initial) de la $\infty$\nbd-catégorie $\sHom_C(y,x)$ (resp.~$\sHom_C(x,y)$).
\smallbreak

Réciproquement, supposons que pour tout objet $y$ de $C$, la $\infty$\nbd-catégorie $\sHom_C(y,x)$ (resp.~$\sHom_C(x,y)$) admette un objet quasi-final (resp.~quasi-initial), et soit $F(y)$ un ensemble quasi-final (resp.~quasi-initial) de cellules de $\sHom_C(y,x)$ (resp.~de $\sHom_C(x,y)$) tel que $F(y)_0$ soit non vide. On vérifie aussitôt que si l'on pose
$$E=\{x\}\cup\textstyle\bigcup\limits_{y\in\ob C}F(y)\ ,$$
alors $E$ est un ensemble quasi-final (resp.~quasi-initial) de cellules de $C$, ce qui prouve que $x$ est un objet quasi-final (resp.~quasi-initial) de $C$.
\end{proof}

\begin{prop}\label{eq2pseudo}
Soient $n\geq0$ un entier, et $C$ une $n$\nbd-catégorie. Pour qu'un objet de $C$ soit un objet quasi-final \emph{(resp.}~quasi-initial\emph{)} au sens de la définition du paragraphe~\emph{\ref{pseudofinal},} il faut et il suffit qu'il le soit au sens de la définition du paragraphe~\emph{\ref{oopseudofinal}.}
\end{prop}

\begin{proof}
On raisonne par récurrence sur $n$. Soit $x$ un objet de $C$. Si $n=0$, et s'il existe un ensemble quasi-final $E$ (resp.~quasi-initial) tel que $x\in E$, alors tout objet $y$ de $C$ est relié à $x$ par une $1$\nbd-flèche, et comme les seules $1$\nbd-flèches de $C$ sont les identités, la $0$\nbd-catégorie $C$ s'identifie à l'ensemble $\{x\}$. La réciproque étant évidente, ceci prouve l'assertion pour $n=0$. Si $n>0$, en vertu de la proposition précédente, pour que $x$ soit un objet quasi-final (resp.~quasi-initial) de $C$ au sens de la définition du paragraphe~\ref{oopseudofinal}, il faut et il suffit que pour tout objet $y$ de $C$, la $(n-1)$\nbd-catégorie $\sHom_C(y,x)$ (resp.~$\sHom_C(x,y)$) admette un objet quasi-final (resp.~quasi-initial) au sens de cette même définition. On conclut donc par l'hypothèse de récurrence.
\end{proof}

\begin{rem}
En vertu de la proposition précédente, il n'y aura donc par lieu de distinguer, pour une $n$\nbd-catégorie, $0\leq n<\infty$, les deux notions d'objet quasi-final (resp.~quasi-initial).
\end{rem}

\begin{prop}\label{tronqpseudoinf}
Soient $C$ une $\infty$\nbd-catégorie, et $x$ un objet quasi-final \emph{(resp.}~quasi-initial\emph{)} de $C$. Alors pour tout entier $n\geq0$, l'image de $x$ dans la $n$\nbd-catégorie $\ti{n}(C)$, $n$\nbd-tronqué intelligent de $C$ \emph{(cf.~\ref{tronq}),} est un objet quasi-final \emph{(resp.}~quasi-initial\emph{)} de~$\ti{n}(C)$.
\end{prop}

\begin{proof}
On raisonne par récurrence sur $n$. Si $n=0$, la $0$\nbd-catégorie $\ti{0}(C)$ est le quotient de l'ensemble des objets de $C$ par la relation d'équivalence 
$$y\sim y'\ \Longleftrightarrow\ \hbox{il existe un zigzag de $1$\nbd-flèches de $C$ reliant $y$ et $y'$}.$$
S'il existe un ensemble quasi-final (resp. quasi-initial) $E$ de cellules de $C$ tel que $x\in E_0$, alors
ce quotient se réduit à la classe d'équivalence de $x$, ce qui prouve l'assertion dans ce cas. Si $n>0$, en vertu de la proposition~\ref{carlocpseudo} et de l'hypothèse de récurrence, pour tout objet $y$ de $C$, la $(n-1)$\nbd-catégorie 
$$\begin{aligned}
&\ti{n-1}(\sHom_C(y,x))\simeq\sHom_{\ti{n}(C)}(y,x)\cr 
\hbox{(resp.}\ \ &\ti{n-1}(\sHom_C(x,y))\simeq\sHom_{\ti{n}(C)}(x,y)\ )
\end{aligned}$$ 
admet  un objet quasi-final (resp.~quasi-initial). 
Par une nouvelle application de la proposition~\ref{carlocpseudo}, on en déduit que $x$ est un objet quasi-final (resp.~quasi-initial) de la $n$\nbd-catégorie $\ti{n}(C)$.
\end{proof}

\begin{lemme}
Soit $C$ une $\infty$\nbd-catégorie. Si pour tout entier $n\geq0$, le $n$\nbd-tronqué intelligent $\ti{n}(C)$ est asphérique, alors $C$ est asphérique.
\end{lemme}

\begin{proof}
Il s'agit de montrer que sous les hypothèses du lemme, l'ensemble simplicial $\NStreet(C)$ est faiblement contractile, autrement dit que $\pi_0(C)$ est un singleton, et que pour tout $i>0$ et tout $0$\nbd-simplexe $x$ de $\NStreet(C)$, le groupe d'homotopie $\pi_i(\NStreet(C),x)$ est trivial. Comme pour tout $n\geq0$, l'oriental $\Or{n}$ est une $n$\nbd-catégorie, on a par adjonction une bijection naturelle
$$(\NStreet(C))_n=\Hom_{\ooCat}(\Or{n},C)\simeq\Hom_{\ooCat}(\Or{n},\ti{n}(C))=(\NStreet\ti{n}(C))_n\ ,$$
autrement dit le $n$\nbd-squelette de $\NStreet(C)$ coïncide avec celui de $\NStreet\ti{n}(C)$.
Vu que les~$\pi_i$ d'un ensemble simplicial ne dépendent que de son $(i+1)$-squelette, cela prouve l'assertion.
\end{proof}

\begin{prop}\label{asphpseudoinf}
Soit $C$ une $\infty$\nbd-catégorie admettant un objet quasi-final \emph{(resp.}~quasi-initial\emph{)}. Alors $C$ est asphérique.
\end{prop}

\begin{proof}
En vertu de la proposition~\ref{tronqpseudoinf}, pour tout entier $n\geq0$, la $n$\nbd-catégorie $\ti{n}(C)$ admet un objet quasi-final (resp. quasi-initial), et par suite, il résulte du corollaire~\ref{pseudoasph} qu'elle est asphérique. La proposition résulte donc du lemme précédent.
\end{proof}

\begin{paragr}\label{deftransf}
Soient $F,G:C\todouble D$ deux $\infty$\nbd-foncteurs de même source et même but. Une \emph{prétransformation} $\alpha$ de $F$ vers $G$ est la donnée pour tout $i\geq0$, et toute $i$\nbd-cellule $x$ de $C$, d'une $(i+1)$\nbd-flèche
$$\transf{\alpha}{x}:\transf{\alpha}{t_{i-1}(x)}\comp{i-1}\cdots\comp{1}\transf{\alpha}{t_{0}(x)}\comp{0}F(x)\toto G(x)\comp{0}\transf{\alpha}{s_{0}(x)}\comp{1}\cdots\comp{i-1}\transf{\alpha}{s_{i-1}(x)}$$
de $D$ (en se rappelant que si $k<j$ l'opération $\comp{k}$ est prioritaire sur l'opération $\comp{j}$ (cf.~\ref{notooCat})). Par exemple, si $x$ est un objet de $C$,
$$\transf{\alpha}{x}:F(x)\toto G(x)$$
est une $1$\nbd-flèche de $D$, si $x$ est une $1$\nbd-flèche de $C$,
$$\transf{\alpha}{x}:\transf{\alpha}{t_{0}(x)}\comp{0}F(x)\toto G(x)\comp{0}\transf{\alpha}{s_{0}(x)}$$
est une $2$\nbd-flèche de $D$, et si $x$ est une $2$\nbd-flèche de $C$,
$$\transf{\alpha}{x}:\transf{\alpha}{t_{1}(x)}\comp{1}\bigl(\transf{\alpha}{t_{0}(x)}\comp{0}F(x)\bigr)\toto \bigl(G(x)\comp{0}\transf{\alpha}{s_{0}(x)}\bigr)\comp{1}\transf{\alpha}{s_{1}(x)}$$
est une $3$\nbd-flèche de $D$. En revenant au cas général d'une $i$\nbd-cellule $x$, pour que cette définition soit licite, il faut vérifier que pour $i\geq0$, les $i$\nbd-cellules
$$\transf{\alpha}{t_{i-1}(x)}\comp{i-1}\cdots\comp{1}\transf{\alpha}{t_{0}(x)}\comp{0}F(x)\quad\hbox{et}\quad G(x)\comp{0}\transf{\alpha}{s_{0}(x)}\comp{1}\cdots\comp{i-1}\transf{\alpha}{s_{i-1}(x)}$$
sont parallèles et les compositions qui y figurent sont bien définies, ce qu'on démontre par récurrence, en observant que pour $i=0$, il n'y a rien à prouver et que pour $i>0$, on a un carré de $i$\nbd-flèches
$$
\xymatrixrowsep{3.4pc}
\xymatrixcolsep{2.7pc}
\xymatrix{
\transf{\alpha}{t_{i-2}x}\kern-2pt\comp{i-2}\cdots\comp{1}\kern-2pt\transf{\alpha}{t_{0}x}\kern-2pt\comp{0}\kern-2pt F(s_{i-1}x)\ar[r]^{\transf{\alpha}{s_{i-1}x}}
\ar@<5ex>[d]_{\transf{\alpha}{t_{i-2}x}\kern-2pt\comp{i-2}\cdots\comp{1}\kern-2pt\transf{\alpha}{t_{0}x}\kern-2pt\comp{0}\kern-2ptF(x)}
&G(s_{i-1}x)\kern-2pt\comp{0}\kern-2pt\transf{\alpha}{s_{0}x}\kern-2pt\comp{1}\cdots\comp{i-2}\kern-2pt\transf{\alpha}{s_{i-2}x}
\ar@<-5ex>[d]^{G(x)\kern-2pt\comp{0}\kern-2pt\transf{\alpha}{s_{0}x}\kern-2pt\comp{1}\cdots\comp{i-2}\kern-2pt\transf{\alpha}{s_{i-2}x}}
\\
\transf{\alpha}{t_{i-2}x}\kern-2pt\comp{i-2}\cdots\comp{1}\kern-2pt\transf{\alpha}{t_{0}x}\kern-2pt\comp{0}\kern-2pt F(t_{i-1}x)\ar[r]_{\transf{\alpha}{t_{i-1}x}}
&G(t_{i-1}x)\kern-2pt\comp{0}\kern-2pt\transf{\alpha}{s_{0}x}\kern-2pt\comp{1}\cdots\comp{i-2}\kern-2pt\transf{\alpha}{s_{i-2}x}\kern 5pt .\kern -7pt
}$$
On en déduit aussi par une récurrence immédiate que pour tout $j$, $0\leq j\leq i$, on~a 
\begin{equation}\label{sourcetransf}\begin{aligned}
&s_j\transf{\alpha}{x}=\transf{\alpha}{t_{j-1}x}\comp{j-1}\cdots\comp{1}\transf{\alpha}{t_{0}x}\comp{0}F(s_jx)=s_j\transf{\alpha}{s_jx}\ ,\cr
\noalign{\vskip 3pt}
&t_j\transf{\alpha}{x}=G(t_jx)\comp{0}\transf{\alpha}{s_{0}x}\comp{1}\cdots\comp{j-1}\transf{\alpha}{s_{j-1}x}=t_j\transf{\alpha}{t_jx}\ .
\end{aligned}\end{equation}
On dit que la prétransformation $\alpha$ est une \emph{transformation}\footnote{\,Il s'agit de la version \og oplax\fg{} de cette notion.} si les deux condition suivantes sont satisfaites.
\smallbreak

a) \textsc{Compatibilité aux unités.} Pour tout $i\geq0$ et toute $i$\nbd-cellule $x$ de $C$, on a 
\[ \transf{\alpha}{1_x}=1_{\transf{\alpha}{x}}. \]
\smallbreak

b) \textsc{Compatibilité aux compositions.} Pour tous $i,j$, $0\leq j<i$, et tout couple de $i$\nbd-cellules $j$\nbd-composables $x,y$ de $C$, on a
$$\transf{\alpha}{x\kern-2pt\comp{j}\kern-2pty}=G(t_{j+1}x)\kern-2pt\comp{0}\kern-2pt\transf{\alpha}{s_0}\kern-4pt\comp{1}\kern-2pt\cdots\kern-1pt\comp{j-1}\kern-2pt\transf{\alpha}{s_{j-1}}\kern-4pt\comp{j}\kern-2pt\transf{\alpha}{y}\kern-2pt\comp{j+1}\kern-2pt\transf{\alpha}{x}\kern-2pt\comp{j}\kern-2pt\transf{\alpha}{t_{j-1}}\kern-4pt\comp{j-1}\kern-2pt\cdots\kern-1pt\comp{1}\kern-2pt\transf{\alpha}{t_0}\kern-4pt\comp{0}\kern-2ptF(s_{j+1}y),$$
où pour tout $k$, $0\leq k<j$, on note
$$s_k=s_kx=s_ky\qquad\hbox{et}\qquad t_k=t_kx=t_ky\ .$$
Par exemple, si $x,y$ sont deux $1$\nbd-flèches $0$\nbd-composables de $C$, on a
$$\transf{\alpha}{x\kern-2pt\comp{0}\kern-2pty}=\bigl(G(x)\comp{0}\transf{\alpha}{y}\bigr)\comp{1}\bigl(\transf{\alpha}{x}\comp{0} F(y)\bigr)\ ,$$
si $x,y$ sont deux $2$\nbd-flèches $0$\nbd-composables de $C$, on a
$$\transf{\alpha}{x\kern-2pt\comp{0}\kern-2pty}=\bigl(G(t_1x)\comp{0}\transf{\alpha}{y}\bigr)\comp{1}\bigl(\transf{\alpha}{x}\comp{0} F(s_1y)\bigr)\ ,$$
et si $x,y$ sont deux $2$\nbd-flèches $1$\nbd-composables de $C$, on a
$$\transf{\alpha}{x\kern-2pt\comp{1}\kern-2pty}=\bigl((G(x)\comp{0}\transf{\alpha}{s_0x})\comp{1}\transf{\alpha}{y}\bigr)\comp{2}\bigl(\transf{\alpha}{x}\comp{1}(\transf{\alpha}{t_0y}\comp{0}F(y))\bigr)\ .$$
\end{paragr}

\begin{paragr}\label{defoocontraction}
En gardant les notations du paragraphe précédent, on s'intéresse plus particulièrement au cas où $D=C$, et où $G$ est l'endomorphisme identité de $C$ et $F$ un endomorphisme constant, autrement dit tel qu'il existe un objet $c_0$ de $C$ tel que pour tout objet $x$ de $C$, $F(x)=c_0$, et pour tout $i>0$ et toute $i$\nbd-flèche $x$ de $C$, $F(x)$ est la $i$\nbd-flèche identité itérée de $c_0$ (on dit alors que $F$ est le $\infty$\nbd-foncteur \emph{constant de valeur~$c_0$}). Une prétransformation $\alpha$ de $F$ à $G$ est alors la donnée, pour tout objet $x$ de $C$, d'une $1$\nbd-flèche
$$\transf{\alpha}{c}:c_0\toto x\ ,$$
et pour tout $i>0$ et toute $i$\nbd-flèche $x$ de $C$, d'une $(i+1)$\nbd-flèche
$$\transf{\alpha}{x}:\transf{\alpha}{t_{i-1}(x)}\toto x\comp{0}\transf{\alpha}{s_{0}(x)}\comp{1}\cdots\comp{i-1}\transf{\alpha}{s_{i-1}(x)}\ .$$
En effet, comme $F(x)$ est alors la $i$\nbd-flèche identité itérée de $c_0$, la $i$\nbd-flèche $\transf{\alpha}{t_{i-2}(x)}\comp{i-2}\cdots\comp{1}\transf{\alpha}{t_{0}(x)}\comp{0}F(x)$ est une identité, et par suite
$$\transf{\alpha}{t_{i-1}(x)}\comp{i-1}\transf{\alpha}{t_{i-2}(x)}\comp{i-2}\cdots\comp{1}\transf{\alpha}{t_{0}(x)}\comp{0}F(x)=\transf{\alpha}{t_{i-1}(x)}\ .$$
On remarque qu'en vertu de~\ref{sourcetransf}, on a donc pour tous $i,j$, $0\leq j\leq i$, et toute $i$\nbd-cellule $x$ de $C$,
\begin{equation}\label{sourcecontr}
s_j\transf{\alpha}{x}=s_j\transf{\alpha}{s_jx}=
\left\{\begin{aligned}
&c_0\kern10pt\phantom{\transf{\alpha}{t_{j-1}x}}\hbox{si }j=0\ ,\cr
&\transf{\alpha}{t_{j-1}x}\kern10pt\phantom{c_0}\hbox{si }j>0\ .
\end{aligned}\right.\end{equation}
Une telle prétransformation est une transformation si et seulement si les deux conditions suivantes sont satisfaites:
\smallbreak 

a) pour tout $i\geq0$ et toute $i$\nbd-cellule $x$ de $C$, on a $\transf{\alpha}{1_x}=1_{\transf{\alpha}{x}}$;
\smallbreak

b) pour tous $i,j$, $0\leq j<i$, et tout couple de $i$\nbd-cellules $j$\nbd-composables $x,y$, on a
$$\transf{\alpha}{x\kern-2pt\comp{j}\kern-2pty}=t_{j+1}x \comp{0} \transf{\alpha}{s_0x}\comp{1} \cdots\comp{j-1} \transf{\alpha}{s_{j-1}x}\comp{j} \transf{\alpha}{y} \comp{j+1} \transf{\alpha}{x}\ .$$
En effet, comme $F(s_{j+1}y)$ est alors la $(j+1)$\nbd-flèche identité itérée de $c_0$, la $(j+1)$\nbd-flèche $\transf{\alpha}{t_{j-1}y}\comp{j-1}\cdots\comp{1}\transf{\alpha}{t_0y}\comp{0}F(s_{j+1}y)$ est une identité. On dit qu'une telle transformation est une \emph{contraction de $C$ centrée en $c_0$} si elle satisfait aux deux conditions suivantes:
\smallbreak

c) $\transf{\alpha}{c_0}=1_{c_0}$;
\smallbreak

d) pour tout $i\geq0$ et toute $i$\nbd-cellule $x$, on a l'égalité $\transf{\alpha}{\transf{\alpha}{x}}=1_{\transf{\alpha}{x}}$.
\smallbreak

\noindent
Pour que la condition (d) fasse sens, il faut vérifier que $s(\transf{\alpha}{\transf{\alpha}{x}})=t(\transf{\alpha}{\transf{\alpha}{x}})=\transf{\alpha}{x}$. Cela est conséquence du lemme suivant.
\end{paragr}

\begin{lemme}\label{licite}
En gardant les hypothèses et les notations ci-dessus, si la transformation $\alpha$ satisfait à la condition \emph{(c),} et à la condition \emph{(d)} pour les $j$\nbd-cellules pour $0\leq j<i$, alors pour toute $i$\nbd-cellule $x$, on a $s(\transf{\alpha}{\transf{\alpha}{x}})=t(\transf{\alpha}{\transf{\alpha}{x}})=\transf{\alpha}{x}$.
\end{lemme}

\begin{proof}
Soit $x$ une $i$\nbd-cellule de $C$. En vertu de~\ref{sourcecontr}, on a
$$\begin{aligned}
t(\transf{\alpha}{\transf{\alpha}{x}})&= \transf{\alpha}{x}\comp{0}\transf{\alpha}{s_{0}( \transf{\alpha}{x})}\comp{1}\transf{\alpha}{s_{1}( \transf{\alpha}{x})}\comp{2}\cdots\comp{i}\transf{\alpha}{s_{i}( \transf{\alpha}{x})}\cr
&= \transf{\alpha}{x}\comp{0}\transf{\alpha}{c_0}\comp{1}\transf{\alpha}{\transf{\alpha}{t_0x}}\comp{2}\cdots\comp{i}\transf{\alpha}{\transf{\alpha}{t_{i-1}x}}=\transf{\alpha}{x}\ ,
\end{aligned}$$
la dernière égalité résultant de la condition (c) et de la condition (d) pour les $j$\nbd-cellules pour $0\leq j<i$.
\smallbreak

D'autre part, pour tout $j$ tel que $0\leq j<i$, toute $i$\nbd-cellule $y$ et toute $j$\nbd-cellule~$z$, si~$y$ et $\transf{\alpha}{z}$ sont $j$\nbd-composables, alors on a $\transf{\alpha}{y\kern-2pt\comp{j}\kern-2pt\transf{\alpha}{z}}=\transf{\alpha}{y}$. En effet, en vertu de la condition~(b), on a
$$\transf{\alpha}{y\kern-2pt\comp{j}\kern-2pt\transf{\alpha}{z}}=t_{j+1}y \comp{0} \transf{\alpha}{s_0y}\comp{1} \cdots\comp{j-1} \transf{\alpha}{s_{j-1}y}\comp{j} \transf{\alpha}{1^{i}_{\transf{\alpha}{z}}} \comp{j+1} \transf{\alpha}{y}\ ,$$
où $1^{i}_{\transf{\alpha}{z}}$ désigne la $i$\nbd-cellule identité itérée de $\transf{\alpha}{z}$ (égale à $\transf{\alpha}{z}$ si $j=i-1$). Or, il résulte de la condition (a) que $\transf{\alpha}{1^{i}_{\transf{\alpha}{z}}}=1^{i+1}_{\transf{\alpha}{\transf{\alpha}{z}}}$ et de la condition (d) pour les $j$\nbd-cellules pour $0\leq j<i$ que $\transf{\alpha}{1^{i}_{\transf{\alpha}{z}}}=1^{i+1}_{\transf{\alpha}{z}}$. Comme $\transf{\alpha}{z}$ est une $(j+1)$\nbd-cellule, on en déduit que $t_{j+1}y \comp{0} \transf{\alpha}{s_0y}\comp{1} \cdots\comp{j-1} \transf{\alpha}{s_{j-1}y}\comp{j} \transf{\alpha}{1^{i}_{\transf{\alpha}{z}}}$ est une unité pour la composition $\comp{j+1}$, ce qui prouve l'assertion. Pour toute $i$\nbd-cellule $x$ de $C$, on a donc
$$\begin{aligned}
s(\transf{\alpha}{\transf{\alpha}{x}})&=\transf{\alpha}{t(\transf{\alpha}{x})}=\transf{\alpha}{x\comp{0}\transf{\alpha}{s_{0}(x)}\comp{1}\cdots\comp{i-2}\transf{\alpha}{s_{i-2}(x)}\comp{i-1}\transf{\alpha}{s_{i-1}(x)}}\cr
&=\transf{\alpha}{x\comp{0}\transf{\alpha}{s_{0}(x)}\comp{1}\cdots\comp{i-2}\transf{\alpha}{s_{i-2}(x)}}=\cdots=\transf{\alpha}{x\comp{0}\transf{\alpha}{s_{0}(x)}}=\transf{\alpha}{x}\ ,
\end{aligned}$$
ce qui prouve le lemme.
\end{proof}

\begin{paragr}\label{oocontrduale}
Soient $C$ une $\infty$\nbd-catégorie et $c_0$ un objet de $C$. Une \emph{contraction duale} de $C$ centrée en $c_0$ est une contraction de la $\infty$\nbd-catégorie duale $\op{C}$ de $C$ (cf.~\ref{dualooCat}), centrée en $c_0$.
\end{paragr}

\begin{lemme}\label{eqbete}
Soient  $C$ une $\infty$\nbd-catégorie, $\alpha$ une contraction de $C$ centrée en $c_0$, $i\geq0$, et $x$ une $i$\nbd-cellule de $C$. Les conditions suivantes sont équivalentes:
\begin{itemize}
\item[\emph{i)}] $\transf{\alpha}{x}=1_x$;
\item[\emph{ii)}] ou bien $i=0$ et $x=c_0$, ou bien $i>0$ et il existe une $(i-1)$\nbd-cellule de $C$ telle que $x=\transf{\alpha}{y}$.
\end{itemize}
\end{lemme}

\begin{proof}
L'implication (ii) $\Rightarrow$ (i) est évidente. Réciproquement, supposons  que $x$ soit une $i$\nbd-cellule telle que $\transf{\alpha}{x}=1_x$. Alors on a
$$x=s(\transf{\alpha}{x})=\left\{\begin{aligned}
&c_0\kern10pt\phantom{\transf{\alpha}{t(x)}}\hbox{si}\ i=0\,,\cr
&\transf{\alpha}{t(x)}\kern10pt\phantom{c_0}\hbox{si}\ i>0\,,
\end{aligned}\right.$$
ce qui prouve le lemme.
\end{proof}

\begin{prop}\label{contrpseudoin}
Soient $C$ une $\infty$\nbd-catégorie, et $\alpha$ une contraction de $C$ centrée en $c_0$. Alors $c_0$ est un objet quasi-initial de $C$. Plus généralement, pour tout $i\geq0$, si $x$ est une $i$\nbd-cellule de $C$ telle que $\transf{\alpha}{x}=1_x$, alors $x$ est un objet quasi-initial de la $\infty$\nbd-catégorie $\Par_C(x)$ \emph{(cf.~\ref{defcelparal}).}
\end{prop}

\begin{proof}
On va montrer que l'ensemble $E$ formé des cellules $x$ de $C$ telles que $\transf{\alpha}{x}=1_x$ est un ensemble quasi-initial de cellules de $C$. En vertu du lemme précédent,
$$E=\{c_0\}\cup\{\transf{\alpha}{z}\,|\,z\hbox{ cellule de }C\}\ .$$
Soient $i\geq0$, $x$ une $i$\nbd-cellule de $C$ appartenant à $E$ et $y$ une $i$\nbd-cellule parallèle à $x$. Il s'agit de prouver qu'il existe une $(i+1)$\nbd-cellule appartenant à $E$ de source $x$ et de but~$y$. Montrons que la $(i+1)$\nbd-cellule $\transf{\alpha}{y}$ convient. Elle est bien dans $E$. Si $i=0$, on a $x=c_0$ et $\transf{\alpha}{y}$ est par définition une $1$\nbd-flèche de $c_0$ vers $y$. Si $i>0$, il existe une $(i-1)$\nbd-cellule $z$ telle que $x=\transf{\alpha}{z}$, et on a
$$s(\transf{\alpha}{y})=\transf{\alpha}{t(y)}=\transf{\alpha}{t(x)}=s(\transf{\alpha}{x})=s(1_x)=x$$
et
$$\begin{aligned}
t(\transf{\alpha}{y})&=y\comp{0}\transf{\alpha}{s_{0}(y)}\comp{1}\transf{\alpha}{s_{1}(y)}\comp{2}\cdots\comp{i-1}\transf{\alpha}{s_{i-1}(y)}\cr
&=y\comp{0}\transf{\alpha}{s_{0}(\transf{\alpha}{z})}\comp{1}\transf{\alpha}{s_{1}(\transf{\alpha}{z})}\comp{2}\cdots\comp{i-1}\transf{\alpha}{s_{i-1}(\transf{\alpha}{z})}\cr
&=y\comp{0}\transf{\alpha}{c_0}\comp{1}\transf{\alpha}{\transf{\alpha}{t_0z}}\comp{2}\cdots\comp{i-1}\transf{\alpha}{\transf{\alpha}{t_{i-2}z}}=y\ ,
\end{aligned}$$
l'avant dernière égalité résultant de~\ref{sourcecontr} et la dernière des conditions (c) et (d) du paragraphe~\ref{defoocontraction}. Comme $c_0\in E$, on en déduit que $c_0$ est un objet quasi-initial de $C$. De même, on en déduit que pour toute cellule $x$ de $C$ telle que $\transf{\alpha}{x}=1_x$, l'ensemble des cellules de $\Par_C(x)$ appartenant à $E$ est un ensemble quasi-initial de cellules de $\Par_C(x)$, et $x$ appartient à cet ensemble, ce qui prouve la dernière assertion.
\end{proof}

\begin{paragr}\label{hmtptransf}
Soient $(K,K^*,\augm)$ et $(K',K'^*,\augm')$ deux complexes dirigés augmentés, notés plus simplement respectivement $K$ et $K'$, $f,g:K\todouble K'$ deux morphismes de complexes dirigés augmentés, et $h$ une homotopie de morphismes de complexes dirigés augmentés de $f$ vers $g$ (cf.~\ref{ADChmt}). On en déduit deux $\infty$\nbd-catégories $\Stnu(K)$ et $\Stnu(K')$, deux $\infty$\nbd-foncteurs $\Stnu(f),\Stnu(g):\Stnu(K)\todouble\Stnu(K')$ (cf.~\ref{defStnu}), et on définit une transformation $\Stnu(h)$ de $\Stnu(f)$ vers $\Stnu(g)$ comme suit. Pour toute $i$\nbd-cellule
$$\cellule{x}=\tablnu{x}{i-1}{x_i}\ $$
de $\Stnu(K)$, on définit une $(i+1)$\nbd-cellule $\transf{\Stnu(h)}{x}$ de $\Stnu(K')$ par la formule
$$
\transf{\Stnu(h)}{\cellule{x}}=
\begin{pmatrix}
f_0x^0_0 &f_1x^0_1+h_0x^1_0 &\dots &f_ix^0_i+h_{i-1}x^1_{i-1} &h_ix_i\cr
\noalign{\vskip 3pt}
g_0x^1_0 &g_1x^1_1+h_0x^0_0 &\dots &g_ix^1_i+h_{i-1}x^0_{i-1} &h_ix_i
\end{pmatrix}\ ,
$$
où $x^0_i=x^1_i=x^{}_i$. En effet, on a
$$\begin{aligned}
d'_{i+1}h_ix_i&=g_ix_i-f_ix_i-h_{i-1}d_ix_i\cr
&=g_ix_i-f_ix_i-h_{i-1}(x^1_{i-1}-x^0_{i-1})=(g_ix^1_i+h_{i-1}x^0_{i-1})-(f_ix^0_i+h_{i-1}x^1_{i-1})\,;
\end{aligned}$$
pour $0<j<i$, on a
$$\begin{aligned}
d'_{j+1}(f_{j+1}x^0_{j+1}+h_{j}x^1_{j})&=f_jd_{j+1}x^0_{j+1}+g_jx^1_{j}-f_jx^1_{j}-h_{j-1}d_jx^1_{j}\cr
&=f_j(x^1_j-x^0_j)+g_jx^1_{j}-f_jx^1_{j}-h_{j-1}(x^1_{j-1}-x^0_{j-1})\cr
&=(g_jx^1_j+h_{j-1}x^0_{j-1})-(f_jx^0_j+h_{j-1}x^1_{j-1})\ ;
\end{aligned}$$
enfin, on a
$$\begin{aligned}
d'_1(f_1x^0_1+h_0x^1_0)&=f_0d_1x^0_1+g_0x^1_0-f_0x^1_0\cr
&=f_0(x^1_0-x^0_0)+g_0x^1_0-f_0x^1_0=g_0x^1_0-f_0x^0_0\ .
\end{aligned}$$
Comme on a $\augm'f_0=\augm$, $\augm'g_0=\augm$, et pour tout $j\geq0$, $f_j(K^*_j)\subset K'^*_j$, $g_j(K^*_j)\subset K'^*_j$ et $h_j(K^*_j)\subset K'^*_{j+1}$, ceci prouve que $\transf{\Stnu(h)}{x}$ est bien une $(i+1)$\nbd-cellule de $\Stnu(K')$.
\smallbreak

Pour montrer qu'on définit ainsi une transformation, on commence par vérifier que la source de la $(i+1)$\nbd-cellule $\transf{\Stnu(h)}{\cellule{x}}$ est 
$$\transf{\Stnu(h)}{t_{i-1}(\cellule{x})}\comp{i-1}\cdots\comp{1}\transf{\Stnu(h)}{t_{0}(\cellule{x})}\comp{0}\Stnu(f)(\cellule{x})\ \phantom{.}$$
et son but
$$\Stnu(g)(\cellule{x})\comp{0}\transf{\Stnu(h)}{s_{0}(\cellule{x})}\comp{1}\cdots\comp{i-1}\transf{\Stnu(h)}{s_{i-1}(\cellule{x})}\ .$$
Pour $i=0$, en tenant compte de l'identification de l'ensemble des objets de $\Stnu(K)$ aux éléments $x$ de $K^*_0$ tels que $\augm(x)=1$, on remarque que la $1$\nbd-flèche  
$$\transf{\Stnu(h)}{x}=
\begin{pmatrix}f_0(x)&h_0(x)\cr
\noalign{\vskip 3pt}
g_0(x)&h_0(x)
\end{pmatrix}$$
de $\Stnu(K')$ est de source $f_0(x)=\Stnu(f)(x)$ et de but $g_0(x)=\Stnu(g)(x)$, ce qui prouve l'assertion dans ce cas. Pour $i>0$, l'assertion résulte du lemme suivant, appliqué à~$j=i-1$, en tenant compte que $x^0_i=x^1_i$.
\end{paragr}

\begin{lemme}\label{hmtppseudotransf}
En gardant les notations ci-dessus, pour tout entier $j$, $0\leq j<i$, on a \emph{(}en supprimant pour abréger les indices de $f$, $g$ et $h$\emph{)}
$$\begin{aligned}
\kern 2.5pt\transf{\Stnu(h)}{t_{j}(\cellule{x})}&\comp{j}\cdots\comp{1}\transf{\Stnu(h)}{t_{0}(\cellule{x})}\comp{0}\Stnu(f)(\cellule{x})\cr
\noalign{\vskip 3pt}
&\kern -15pt=\begin{pmatrix}
fx^0_0
&fx^0_1+hx^1_0
&\dots
&fx^0_j+hx^1_{j-1}
&fx^0_{j+1}+hx^1_j
&fx^0_{j+2}
&\cdots
&fx^0_i\cr
\noalign{\vskip 3pt}
gx^1_0
&gx^1_1+hx^0_0
&\dots
&gx^1_j+hx^0_{j-1}
&fx^1_{j+1}+hx^1_j
&fx^1_{j+2}
&\dots
&fx^1_i
\end{pmatrix}\,,\cr
\noalign{\vskip 5pt}
&\kern -37.35pt\Stnu(g)(\cellule{x})\comp{0}\transf{\Stnu(h)}{s_{0}(\cellule{x})}\comp{1}\cdots\comp{j}\transf{\Stnu(h)}{s_{j}(\cellule{x})}\cr
\noalign{\vskip 3pt}
&\kern -15pt=\begin{pmatrix}
fx^0_0
&fx^0_1+hx^1_0
&\dots
&fx^0_j+hx^1_{j-1}
&gx^0_{j+1}+hx^0_j
&gx^0_{j+2}
&\dots
&gx^0_i\cr
\noalign{\vskip 3pt}
gx^1_0
&gx^1_1+hx^0_0
&\cdots
&gx^1_j+hx^0_{j-1}
&gx^1_{j+1}+hx^0_j
&gx^1_{j+2}
&\dots
&gx^1_i
\end{pmatrix}\,.
\end{aligned}$$
\end{lemme}

\begin{proof}
Montrons la première égalité. On procède par récurrence sur $j$. Pour $j=0$, on a 
$$\begin{aligned}
\transf{\Stnu(h)}{t_{0}(\cellule{x})}\comp{0}\Stnu(f)(\cellule{x})&=
\begin{pmatrix}
fx^1_0
&hx^1_0\cr
\noalign{\vskip 3pt}
gx^1_0
&hx^1_0
\end{pmatrix}
\comp{0}
\tablg{fx}{i}\cr
\noalign{\vskip 3pt}
&=\begin{pmatrix}
fx^0_0
&fx^0_1+hx^1_0
&fx^0_2
&\dots
&fx^0_i\cr
\noalign{\vskip 3pt}
gx^1_0
&fx^1_1+hx^1_0
&fx^1_2
&\dots
&fx^1_i
\end{pmatrix}\ ,
\end{aligned}$$
ce qui prouve l'assertion dans ce cas. Supposons-la pour $j-1$, $j\geq1$, et montrons-la pour $j$. On a
$$\begin{aligned}
&\transf{\Stnu(h)}{t_{j}(\cellule{x})}\comp{j}\transf{\Stnu(h)}{t_{j-1}(\cellule{x})}\comp{j-1}\cdots\comp{1}\transf{\Stnu(h)}{t_{0}(\cellule{x})}\comp{0}\Stnu(f)(\cellule{x})\cr
\noalign{\vskip 3pt}
&\kern 10pt=\begin{pmatrix}
fx^0_0 &fx^0_1+hx^1_0 &\dots&fx^0_{j-1}+hx^1_{j-2} &fx^1_j+hx^1_{j-1} &hx^1_j\cr
\noalign{\vskip 3pt}
gx^1_0 &gx^1_1+hx^0_0 &\dots&gx^1_{j-1}+hx^0_{j-2} &gx^1_j+hx^0_{j-1} &hx^1_j
\end{pmatrix}\cr
&\kern 20pt\comp{j}
\begin{pmatrix}
fx^0_0
&fx^0_1+hx^1_0
&\dots
&fx^0_{j-1}+hx^1_{j-2}
&fx^0_{j}+hx^1_{j-1}
&fx^0_{j+1}
&\cdots
&fx^0_i\cr
\noalign{\vskip 3pt}
gx^1_0
&gx^1_1+hx^0_0
&\dots
&gx^1_{j-1}+hx^0_{j-2}
&fx^1_{j}+hx^1_{j-1}
&fx^1_{j+1}
&\dots
&fx^1_i
\end{pmatrix}\cr
\noalign{\vskip 3pt}
&\kern 10pt=
\begin{pmatrix}
fx^0_0
&fx^0_1+hx^1_0
&\dots
&fx^0_j+hx^1_{j-1}
&fx^0_{j+1}+hx^1_j
&fx^0_{j+2}
&\cdots
&fx^0_i\cr
\noalign{\vskip 3pt}
gx^1_0
&gx^1_1+hx^0_0
&\dots
&gx^1_j+hx^0_{j-1}
&fx^1_{j+1}+hx^1_j
&fx^1_{j+2}
&\dots
&fx^1_i
\end{pmatrix}\,,
\end{aligned}$$
ce qui prouve l'assertion. La deuxième égalité se démontre de façon analogue.
\end{proof}

\begin{paragr}\label{hmtptransfsuite}
En gardant toujours les notations du paragraphe~\ref{hmtptransf}, la vérification pour~$\Stnu(h)$ de la condition (a) de la définition des transformations (cf.~\ref{deftransf}) est immédiate. Pour prouver la condition (b), soit $j$ un entier tel que $0\leq j<i$, et
$$\cellule{y}=\tabld{y}{i}$$
une $i$\nbd-cellule de $\Stnu(K)$ (de sorte que $y^0_i=y^1_i$) dont la $j$\nbd-cellule but itérée est égale à la $j$\nbd-cellule source itérée de $x$, autrement dit, telle que
$$x^\e_k=y^\e_k\,,\ \e\in\{0,1\}\,,\ 0\leq k<j\qquad\hbox{et}\qquad x^0_j=y^1_j\ .$$
Il s'agit de prouver l'égalité
$$\transf{\Stnu(h)}{\cellule{x}\kern-2pt\comp{j}\kern-2pt\cellule{y}}=\bigl(B\comp{j}\transf{\Stnu(h)}{\cellule{y}}\bigr)\comp{j+1}\bigl(\transf{\Stnu(h)}{\cellule{x}}\comp{j}A\bigr)\ ,$$
où
$$\begin{aligned}
A&=\transf{\Stnu(h)}{t_{j-1}y}\comp{j-1}\cdots\comp{1}\transf{\Stnu(h)}{t_0y}\comp{0}\Stnu(f)(s_{j+1}y)\ ,\cr
\noalign{\vskip 3pt}
B&=\Stnu(g)(t_{j+1}x)\comp{0}\transf{\Stnu(h)}{s_0x}\comp{1}\cdots\comp{j-1}\transf{\Stnu(h)}{s_{j-1}x}\ .
\end{aligned}
$$
On note $C$ la matrice
$$\begin{aligned}
C&=\begin{pmatrix}
fx^0_0
&fx^0_1+hx^1_0
&\dots
&fx^0_{j-1}+hx^1_{j-2}
\cr
\noalign{\vskip 3pt}
gx^1_0
&gx^1_1+hx^0_0
&\dots
&gx^1_{j-1}+hx^0_{j-2}
\end{pmatrix}\cr
\noalign{\vskip 3pt}
&=\begin{pmatrix}
fy^0_0
&fy^0_1+hy^1_0
&\dots
&fy^0_{j-1}+hy^1_{j-2}
\cr
\noalign{\vskip 3pt}
gy^1_0
&gy^1_1+hy^0_0
&\dots
&gy^1_{j-1}+hy^0_{j-2}
\end{pmatrix}
\end{aligned}$$
(vide pour $j=0$), et pour toute matrice 
$$z=\tablg{z}{k}\ ,$$
et $l$, $0\leq l\leq k$, on pose
$$z^{}_{<l}=\tablg{z}{l-1}\qquad\hbox{et}\qquad z^{}_{\geq l}=
\begin{pmatrix}
z^0_l
&z^0_{l+1}
&\dots
&z^0_k\cr
\noalign{\vskip 3pt}
z^1_l
&z^1_{l+1}
&\dots
&z^1_k
\end{pmatrix}\ .
$$
Par définition, on a
$$\bigl(\transf{\Stnu(h)}{\cellule{x}\kern-2pt\comp{j}\kern-2pt\cellule{y}}\bigr)^{}_{<j}=\bigl(\transf{\Stnu(h)}{\cellule{x}}\bigr)^{}_{<j}=\bigl(\transf{\Stnu(h)}{\cellule{y}}\bigr)^{}_{<j}=C\ ,$$
et en vertu du lemme~\ref{hmtppseudotransf}, appliqué à $s_{j+1}y$ et $t_{j+1}x$, on a aussi $A_{<j}=C=B_{<j}$, et par suite,
$$\bigl(\transf{\Stnu(h)}{\cellule{x}\kern-2pt\comp{j}\kern-2pt\cellule{y}}\bigr)_{<j}=\bigl(B\comp{j}\transf{\Stnu(h)}{\cellule{y}}\comp{j+1}\transf{\Stnu(h)}{\cellule{x}}\comp{j}A\bigr)_{<j}\ .$$
D'autre part, toujours en vertu du lemme~\ref{hmtppseudotransf}, appliqué à $s_{j+1}y$ et $t_{j+1}x$, on a
$$
A_{\geq j}=\begin{pmatrix}
fy^0_{j}+hy^1_{j-1}
&fy^0_{j+1}
\cr
\noalign{\vskip 3pt}
fy^1_{j}+hy^1_{j-1}
&fy^0_{j+1}
\end{pmatrix}\ ,\quad
B_{\geq j}=\begin{pmatrix}
gx^0_{j}+hx^0_{j-1}
&gx^1_{j+1}
\cr
\noalign{\vskip 3pt}
gx^1_{j}+hx^0_{j-1}
&gx^1_{j+1}
\end{pmatrix}\ ,
$$
et comme
$$\begin{aligned}
\bigl(\transf{\Stnu(h)}{\cellule{x}}\bigr)^{}_{\geq j}&=
\begin{pmatrix}
fx^0_j+hx^1_{j-1}
&fx^0_{j+1}+hx^1_{j}
&\dots
&fx^0_{i}+hx^1_{i-1}
&hx^1_i\cr
\noalign{\vskip 3pt}
gx^1_j+hx^0_{j-1}
&gx^1_{j+1}+hx^0_{j}
&\dots
&gx^1_{i}+hx^0_{i-1}
&hx^0_i
\end{pmatrix}\ ,
\cr
\noalign{\vskip 5pt}
\bigl(\transf{\Stnu(h)}{\cellule{y}}\bigr)^{}_{\geq j}&=
\begin{pmatrix}
fy^0_j+hy^1_{j-1}
&fy^0_{j+1}+hy^1_{j}
&\dots
&fy^0_{i}+hy^1_{i-1}
&hy^1_i\cr
\noalign{\vskip 3pt}
gy^1_j+hy^0_{j-1}
&gy^1_{j+1}+hy^0_{j}
&\dots
&gy^1_{i}+hy^0_{i-1}
&hy^0_i
\end{pmatrix}\ ,
\end{aligned}
$$
on a
$$\begin{aligned}
&\bigl(\transf{\Stnu(h)}{\cellule{x}}\comp{j}A\bigr)_{\geq j}=\cr
&\kern 3pt\begin{pmatrix}
fy^0_j+hy^1_{j-1}
&fx^0_{j+1}+fy^0_{j+1}+hx^1_{j}
&fx^0_{j+2}+hx^1_{j+1}
&\dots
&fx^0_{i}+hx^1_{i-1}
&hx^1_i\cr
\noalign{\vskip 3pt}
gx^1_j+hx^0_{j-1}
&gx^1_{j+1}+fy^0_{j+1}+hx^0_{j}
&gx^1_{j+2}+hx^0_{j+1}
&\dots
&gx^1_{i}+hx^0_{i-1}
&hx^0_i
\end{pmatrix}\,,
\cr
\noalign{\vskip 6pt}
&\bigl(B\comp{j}\transf{\Stnu(h)}{\cellule{y}}\bigr)_{\geq j}=\cr
&\kern 3pt\begin{pmatrix}
fy^0_j+hy^1_{j-1}
&fy^0_{j+1}+gx^1_{j+1}+hy^1_{j}
&fy^0_{j+2}+hy^1_{j+1}
&\dots
&fy^0_{i}+hy^1_{i-1}
&hy^1_i\cr
\noalign{\vskip 3pt}
gx^1_j+hx^0_{j-1}
&gx^1_{j+1}+gy^1_{j+1}+hy^0_{j}
&gy^1_{j+2}+hy^0_{j+1}
&\dots
&gy^1_{i}+hy^0_{i-1}
&hy^0_i
\end{pmatrix}\,.
\end{aligned}
$$
On en déduit que
$$\begin{aligned}
&\bigl(B\comp{j}\transf{\Stnu(h)}{\cellule{y}}\comp{j+1}\transf{\Stnu(h)}{\cellule{x}}\comp{j}A\bigr)_{\geq j}=\cr
\noalign{\vskip 5pt}
&\left(
\begin{matrix}
fy^0_j+hy^1_{j-1}
&fx^0_{j+1}+fy^0_{j+1}+hx^1_{j}
&fx^0_{j+2}+fy^0_{j+2}+hx^1_{j+1}+hy^1_{j+1}
&\dots
\cr
\noalign{\vskip 3pt}
gx^1_j+hx^0_{j-1}
&gx^1_{j+1}+gy^1_{j+1}+hy^0_{j}
&gx^1_{j+2}+gy^1_{j+2}+hx^0_{j+1}+hy^0_{j+1}
&\dots
\end{matrix}
\right.\cr
\noalign{\vskip 3pt}
&\kern 150pt\left.
\begin{matrix}
\dots
&fx^0_{i}+fy^0_{i}+hx^1_{i-1}+hy^1_{i-1}
&hx^1_{i}+hy^1_{i}
\cr
\noalign{\vskip 3pt}
\dots
&gx^1_{i}+gy^1_{i}+hx^0_{i-1}+hy^0_{i-1}
&hx^0_{i}+hy^0_{i}
\end{matrix}
\right)\cr
&=\bigl(\transf{\Stnu(h)}{\cellule{x}\kern-2pt\comp{j}\kern-2pt\cellule{y}}\bigr)^{}_{\geq j}\ ,
\end{aligned}$$
ce qui prouve la condition (b).
\end{paragr}

\begin{paragr}\label{contrcontr}
Soient $(K,K^*,\augm)$ un complexe dirigé augmenté décent (cf.~\ref{ADCdesc}), noté plus simplement $K$, $f$ un endomorphisme constant de $K$ de valeur $c_0$ (cf.~\ref{ADCcons}), et $h$ une homotopie de morphismes de complexes dirigés augmentés de $f$ vers le morphisme identité de $K$ (cf.~\ref{ADChmt}). On suppose que l'homotopie $h$ est de carré nul, autrement dit que pour tout $j\geq0$, $h_{j+1}h_j=0$, définissant ainsi une contraction de $K$ (cf.~\ref{remcontraction}). On rappelle qu'alors $c_0$ est dans $K^*_0$, on a $\augm(c_0)=1$, $h_0(c_0)=0$, pour tout $x\in K_0$, $f_0(x)=\augm(x).c_0$ et pour tout $i>0$ et tout $x\in K_i$, $f_i(x)=0$. L'endomorphisme $\Stnu(f)$ de la $\infty$\nbd-catégorie $\Stnu(K)$ est alors un $\infty$\nbd-foncteur constant de valeur $c_0$ (et tenant compte de l'identification de l'ensemble des objets de $\Stnu(K)$ avec l'ensemble des $x\in K^*_0$ tels que $\augm(x)=1$). L'homotopie $h$ définit une transformation $\Stnu(h)$ de $\Stnu(f)$ vers l'identité de $\Stnu(K)$ (cf.~\ref{hmtptransf}).
\end{paragr}

\begin{prop}\label{propcontrcontr}
En gardant les hypothèses et les notations du paragraphe précédent, la transformation $\Stnu(h)$ est une contraction centrée en $c_0$ de la $\infty$\nbd-catégorie~$\Stnu(K)$. 
\end{prop}

\begin{proof}
On a
$$\transf{\Stnu(h)}{\cellule{c_0}}=
\begin{pmatrix}
c_0
&h_0(c_0)
\cr
\noalign{\vskip 3pt}
c_0
&h_0(c_0)
\end{pmatrix}=
\begin{pmatrix}
c_0
&0
\cr
\noalign{\vskip 3pt}
c_0
&0
\end{pmatrix}=1_{c_0}\ ,
$$
ce qui prouve la condition (c) de la définition d'une contraction (cf.~\ref{defoocontraction}). Démontrons la condition (d). Pour tout $i\geq0$, et toute $i$\nbd-cellule 
$$\cellule{x}=\tabll{x}{i}$$
de $\Stnu(K)$, on a (en omettant les indices de $h$)
$$\transf{\Stnu(h)}{\cellule{x}}=
\begin{pmatrix}
c_0
&hx^1_0
&hx^1_1
&\dots
&hx^1_{i-1}
&hx^1_i
\cr
\noalign{\vskip 3pt}
x^1_0
&x^1_1+hx^0_0
&x^1_2+hx^0_1
&\dots
&x^1_i+hx^0_{i-1}
&hx^0_i
\end{pmatrix}\ ,$$
et par suite,
$$\begin{aligned}
\transf{\Stnu(h)}{\cellule{\transf{\Stnu(h)}{\cellule{x}}}}&=
\left(\begin{matrix}
c_0
&hx^1_0
&h(x^1_1+hx^0_0)
&\dots
\cr
\noalign{\vskip 3pt}
x^1_0
&x^1_1+hx^0_0+hc_0
&x^1_2+hx^0_1+hhx^1_0
&\dots
\end{matrix}\right.
\cr
\noalign{\vskip 3pt}
&\kern 50pt\left.\begin{matrix}
&\dots
&h(x^1_{i-1}+hx^0_{i-2})
&h(x^1_i+hx^0_{i-1})
&hhx^0_i
\cr
\noalign{\vskip 3pt}
&\dots
&x^1_i+hx^0_{i-1}+hhx^1_{i-2}
&hx^0_i+hhx^1_{i-1}
&hhx^1_i
\end{matrix}\right)
\cr
\noalign{\vskip 3pt}
&=
\begin{pmatrix}
c_0
&hx^1_0
&hx^1_1
&\dots
&hx^1_{i-1}
&hx^1_i
&0
\cr
\noalign{\vskip 3pt}
x^1_0
&x^1_1+hx^0_0
&x^1_2+hx^0_1
&\dots
&x^1_i+hx^0_{i-1}
&hx^0_i
&0
\end{pmatrix}
=1_{\transf{\Stnu(h)}{\cellule{x}}}
\ ,
\end{aligned}$$
ce qui prouve l'assertion.
\end{proof}

Le corollaire suivant est une variante de la proposition~\ref{existpsin2}.

\begin{cor}
Soit $K$ un complexe dirigé augmenté décent. S'il existe une homotopie de carré nul d'un endomorphisme constant de $K$, de valeur $c_0$, vers l'endo\-mor\-phisme identité de~$K$, alors $c_0$ est un objet quasi-initial de la $\infty$\nbd-catégorie~$\Stnu(K)$. Dualement, s'il existe une homotopie de carré nul de l'endo\-morphisme identité de~$K$ vers un endomorphisme constant, de valeur $c_0$,  alors $c_0$ est un objet quasi-final de~$\Stnu(K)$.
\end{cor}

\begin{proof}
La première assertion est conséquence immédiate des propositions~\ref{contrpseudoin} et~\ref{propcontrcontr}. En vertu de la proposition~\ref{dualStnu}
et conformément au paragraphe~\ref{oopseudofinal}, la deuxième assertion résulte de la première, appliquée au complexe dirigé augmenté dual $\op{K}$ (cf.~\ref{dualADC}) et à l'homotopie de l'endomorphisme constant de~$\op{K}$, de valeurs $c_0$, vers l'endo\-mor\-phisme identité de $\op{K}$, induite par $h$ (cf.~\ref{dualhmtpADC}).
\end{proof}

\begin{rem}
Il résulte des propositions~\ref{hmtpsimpl} et~\ref{propcontrcontr} que pour tout $n\geq0$, l'oriental $\Or{n}$ admet une contraction centrée en $(0)$. De même, il résulte des propositions~\ref{hmtpcompl}, \ref{dualStnu}, \ref{propcontrcontr}, et du paragraphe~\ref{dualhmtpADC} que $\Or{n}$ admet une contraction duale centrée en $(n)$ (cf.~\ref{oocontrduale}). Ainsi, on peut retrouver facilement le théorème~\ref{bigchief} en utilisant la proposition~\ref{contrpseudoin}.

\end{rem}

\backmatter


\bibliography{biblio}
\bibliographystyle{mysmfplain}

\end{document}